\documentclass[11pt]{amsart}

\usepackage[margin=1in]{geometry}
\usepackage{setspace}
\usepackage{amsmath}
\usepackage{amsxtra}
\usepackage{amscd}
\usepackage{amsthm}
\usepackage{amssymb}
\usepackage[normalem]{ulem}
\usepackage{comment}
\usepackage[dvipsnames,x11names]{xcolor}
\usepackage{tikz}
\usetikzlibrary{matrix,arrows,decorations.pathmorphing,automata, matrix, positioning, calc, shapes.multipart,intersections}
\usepackage{graphicx}
\usepackage{hyperref}
\usepackage[shortlabels]{enumitem}
\usepackage{float}
\usepackage{ytableau}
\usepackage{subcaption}

\usepackage{mathrsfs}

\definecolor{LightGray}{gray}{0.94}
\definecolor{MediumGray}{gray}{0.6}

\numberwithin{equation}{section}

\theoremstyle{definition}
\newtheorem{theorem}[equation]{Theorem}
\newtheorem{lemma}[equation]{Lemma}
\newtheorem{proposition}[equation]{Proposition}
\newtheorem{corollary}[equation]{Corollary}

\newtheorem{definition}[equation]{Definition}

\newtheorem{example}[equation]{Example}

\newtheorem{convention}[equation]{Convention}
\newtheorem{remark}[equation]{Remark}

\newtheorem{rmk}[equation]{Remark}
\newtheorem{defn}[equation]{Definition}
\newtheorem{eg}[equation]{Example}

\newtheorem{notn}[equation]{Notation}
\newtheorem{thm}[equation]{Theorem}
\newtheorem{prop}[equation]{Proposition}

\newtheorem*{thmA}{Theorem A}
\newtheorem*{thmB}{Theorem B}
\newtheorem*{thmC}{Theorem C}

\newtheorem*{acknowledgements}{Acknowledgements}

\newcommand{\emp}{\emptyset}

\newcommand{\sle}{\widehat{\mathfrak{sl}}_e}
\newcommand{\w}{{\underline{w}}}
\newcommand{\x}{{\underline{x}}}
\newcommand{\y}{{\underline{y}}}
\newcommand{\ctau}{{{\color{cyan}\beta}}}
\newcommand{\csigma}{{{\color{magenta}{\alpha }}}}
\newcommand{\crho}{{{\color{darkgreen}{\gamma}}}}

\newcommand{\mydots}{\dots}

\newcommand{\cH}{\mathcal{H}}
\newcommand{\cI}{\mathcal{I}}

\newcommand{\FF}{\mathsf{F}}

\newcommand{\Ext}{\operatorname{Ext}}

\newcommand{\algebra}{\mathscr{H}_\aatchpair (n,\bs )}

\newcommand{\Res}{\operatorname{Res}}

\newcommand{\la}{\lambda}
\newcommand{\bs}{\mathbf{s}}

\newcommand{\bla}{{\boldsymbol{\lambda}}}
\newcommand{\bmu}{\boldsymbol{\mu}}
\newcommand{\bnu}{\boldsymbol{\nu}}
\newcommand{\Bo}{\mathsf{B}}
\newcommand{\co}{\mathsf{co}}
\newcommand{\Cali}{\mathsf{Cali}}

\newcommand{\C}{\mathbb{C}}
\newcommand{\reflectpath}{\SSTP_\al^\flat}
\newcommand\REMOVETHESE[2]{{{{\mathsf{M}}_{#1}^{#2}	}}}
\newcommand\ADDTHIS[2]{{{{\mathsf{P}}_{#1}^{#2}}}}

 \colorlet{darkgreen}{green!50!black}
\usepackage[foot]{amsaddr}
\usepackage{mathtools}
\newcommand{\conj}[1]{\overline{#1}}
\newcommand{\R}{\mathbb{R}}
 \newcommand{\uni}{\mathfrak{U}}
\DeclareMathOperator{\aff}{aff}
\newcommand{\cal}{\mathfrak{C}}
\newcommand{\Z}{\mathbb{Z}}

\DeclareMathOperator{\content}{ct}

\newcommand{\ZZ}{{\mathbb{Z}}}

\newcommand{\Std}{{\rm Std}}

\newcommand{\Shape}{\operatorname{Shape}} 
\newcommand{\Path}{{\rm Path}}

\newcommand{\sts}{\mathsf{S}}  
\newcommand{\stt}{\mathsf{T}}  
\newcommand{\res}{\mathrm{res}}

\newcommand{\SSTS}{\mathsf{S}}  
\newcommand{\SSTT}{\mathsf{T}}  
\newcommand{\SSCC}{\mathsf{C}}  
\newcommand{\SSTP}{\mathsf{P}}  
\newcommand{\SSTU}{\mathsf{U}}  
\newcommand{\SSTQ}{\mathsf{Q}}  
\newcommand{\SSTR}{\mathsf{R}}  
 \newcommand{\al}{{{ \color{magenta}\alpha} }}
 \newcommand{\bet}{{{\color{cyan}\beta}}}

\definecolor{darkspringgreen}{rgb}{0.09, 0.45, 0.27}
  \newcommand{\gam}{{{{\color{darkgreen}\gamma}}}}

\newcommand{\great}{>}
\newcommand{\less}{<}

    \newcommand{\aatch}{h}
      \newcommand{\aatchpair}{{\underline{h}}}
      \newcommand{\enn}{{h}}

\newcommand{\Alc}{\text{\bf Alc}}

\usepackage{cleveref,scalefnt}

\crefname{defn}{Definition}{Definitions}
\crefname{definition}{Definition}{Definitions}
\crefname{thm}{Theorem}{Theorems}
\crefname{prop}{Proposition}{Propositions}
\crefname{lem}{Lemma}{Lemmas}
\crefname{cor}{Corollary}{Corollaries}
\crefname{conj}{Conjecture}{Conjectures}
\crefname{section}{Section}{Sections}
\crefname{subsection}{Subsection}{Subsections}
\crefname{eg}{Example}{Examples}
\crefname{figure}{Figure}{Figures}
\crefname{rem}{Remark}{Remarks}
\crefname{rmk}{Remark}{Remarks}
\crefname{equation}{equation}{equation}

\Crefname{defn}{Definition}{Definitions}
\Crefname{definition}{Definition}{Definitions}
\Crefname{thm}{Theorem}{Theorems}
\Crefname{prop}{Proposition}{Propositions}
\Crefname{lem}{Lemma}{Lemmas}
\Crefname{cor}{Corollary}{Corollaries}
\Crefname{conj}{Conjecture}{Conjectures}
\Crefname{section}{Section}{Sections}
\Crefname{subsection}{Subsection}{Subsections}
\Crefname{eg}{Example}{Examples}
\Crefname{figure}{Figure}{Figures}
\Crefname{rem}{Remark}{Remarks}
\Crefname{rmk}{Remark}{Remarks}

\newcommand{\braid}{B^{ea}}
\newcommand{\AHA}{\mathcal{H}^{\mathsf{aff}}}
\newcommand{\JM}{\mathbb{A}}
\newcommand{\cyc}{\mathcal{H}}
\newcommand{\fh}{\mathfrak{h}}
\newcommand{\RR}{\mathsf{R}}
\newcommand{\qq}{\mathsf{q}}
\newcommand{\ba}{\mathbf{a}}
\newcommand{\bb}{\mathbf{b}}
\newcommand{\bm}{\mathbf{m}}
\newcommand{\tab}{\SSTT}
\newcommand{\column}{\SSCC}
\newcommand{\RCA}{\mathbb{H}}
\newcommand{\KZ}{\mathsf{KZ}}

\begin{document}

\author{Chris Bowman}
    \address{Department of Mathematics, 
University of York, Heslington, York, YO10 5DD, UK}
\email{Chris.Bowman-Scargill@york.ac.uk}

\author{Emily Norton}
\email{emily.norton@uca.fr}
\address{Universit\'e Clermont Auvergne, CNRS, LMBP, F-63000 Clermont-Ferrand, France.}
\author{Jos\'e Simental}
\address{Max Planck Institute for Mathematics, Vivatsgasse 7, 53111, Bonn, Germany}
\email{jose@mpim-bonn.mpg.de}

\title[Unitary representations of cyclotomic Hecke algebras]{Unitary representations of  \\   cyclotomic Hecke algebras at roots of unity: \\combinatorial classification and BGG resolutions}

\maketitle

\begin{abstract}
We relate the classes of unitary and calibrated  representations of cyclotomic Hecke algebras and, in particular, we show that for the most important deformation parameters these two classes coincide. 
  We classify these representations in terms of both multipartition combinatorics and as the points in the fundamental alcove under the action of an affine Weyl group. Finally, we cohomologically construct these  modules via BGG resolutions.  
 \end{abstract}

  \section*{Introduction}

Unitary representations play a fundamental role in the representation theory of real, complex, and $p$-adic reductive groups  \cite{aldvreal, MR1845669, MR2037715, barbasch2015star}. 
 Unitary representations are often the most important representations appearing  ``in nature"
 via  
  quantum mechanics \cite{Wig39} and 
  harmonic analysis \cite{Mac92}. 
Furthermore, they tend to admit nice structural and homological properties, such as explicit eigenbases and resolutions by Verma modules.  


In this paper, we study unitary representations of a family of infinite discrete groups: the affine braid groups. These are groups $\braid_{n}$ of $n$ braids on the cylinder, see \cite{affinebraid}, and project onto the usual Artin braid groups by \lq\lq flattening\rq\rq\, the cylinder. Of course, the representation theory of the group $\braid_{n}$ is extremely complicated and the problem would be intractable without imposing certain conditions on our representations. The condition we impose is that our representations factor through an \emph{affine Hecke quotient} of the group algebra $\C\braid$, that is, the following skein-like relation is satisfied
\begin{equation}\label{eqn:hecke relation}\tag{$\dagger$}
(T_{i} - q)(T_{i} + 1) = 0
\end{equation}
  for some $q \in \C^{\times}$ and every $i = 1, \dots, n-1$, where $T_{i}$ is the overcrossing of the $i$-th and $(i+1)$-st strands. The algebra $\AHA_{q}(n) := \C\braid_{n}/(T_{i} - q)(T_{i} + 1) $  
 is known as the \emph{affine Hecke algebra}. Besides being interesting in and of itself, the algebra $\AHA_{q}(n)$ appears in the theory of knot invariants, categorification, and the representation theory of $p$-adic reductive groups. Let us now discuss our methods and results in more detail.

 \subsection*{First main result}
  First, we relate the unitary representations of $\AHA_{q}(n)$ to the class of \emph{calibrated} representations. These are a well-studied   class of representations of $\AHA_{q}(n)$ that are defined by the condition that the Jucys-Murphy subalgebra $\JM \subseteq \AHA_{q}(n)$ acts semisimply, see \cite{MR1976700,MR1988991,MR2266877,MR1383482}. Calibrated representations exhibit many of the properties that make unitary representations interesting (for example, by definition they come equipped with an $\JM$-eigenbasis, which is unique if the representation is simple) and it easily follows that, in fact, every unitary representation is calibrated. The converse is of course not true, but it turns out that if we restrict to certain representations (in a sense, the most complicated ones) then the story changes. To be more precise, every irreducible representation of $\AHA_{q}(n)$ factors through a \emph{cyclotomic quotient}, $\cyc_{q, Q_1, \dots, Q_{\ell}}(n)$ which depends on several parameters $q, Q_1, \dots, Q_{\ell}$. The representation theory of $\cyc_{q, Q_1, \dots, Q_{\ell}}(n)$ is most interesting when we specialise the parameter $q$ to be a root of unity and  the parameters $Q_i=q^{s_i} $ for $1\leq i \leq \ell$.  

\begin{thmA}\label{thm:main1}
Let $e > 1$ and $s_1, \dots, s_{\ell}$ be integers. Let $q = \exp(2\pi\sqrt{-1}/e)$ and $Q_{i} = q^{s_{i}}$. A representation of the algebra $\cyc_{q, Q_1, \dots, Q_{\ell}}$ is unitary if and only if it is calibrated.
\end{thmA}

\subsection*{Second main result}
Next, we use Theorem A to combinatorially classify the unitary representations.     Given an integer $e > 1$ and $\bs = (s_1, \dots, s_{\ell})\in \mathbb{Z}^\ell$ a {\em charge}, we denote the algebra $\cyc_{q, Q_1, \dots, Q_{\ell}}(n)$ as above  simply by $\cyc_{\bs}(n)$. 
The definition of the algebra $\cyc_{\bs}(n)$ depends only on the reduction modulo $e$ of the charge  $\bs$, so we can assume  that 
$s_1\leq s_2\leq\ldots \leq s_\ell< s_1+e$.  
 Our choice of  charge allows us to provide a particularly simple  classification of unitary modules in terms of multipartition combinatorics:

 \begin{thmB}
 Let $q = \exp(2\pi\sqrt{-1}/e)$ and fix a charge $\bs=(s_1,\dots s_\ell)$ such that $s_1\leq s_2\leq\ldots \leq s_\ell< s_1+e$.
 The simple $\mathcal{H}_q^{\rm aff}(n)$-module $D_\bs(\bla)$ is unitary if and only if    the following equivalent conditions hold:
 \begin{itemize}[leftmargin=*]
\item $\bla$ is cylindric, its  right border strip  has period at most $e$  and its reading word is increasing; 
\item $\bla\in \mathcal{F}_\aatchpair  $ the fundamental alcove under an  $\bs$-shifted action of an    affine Weyl group of type $A$;
 \item $D_\bs(\bla)$ is calibrated. 
 \end{itemize}
 Precise definitions for the terminology in the  first two conditions 
 are given in \cref{subsec:rightbord,fundamental-alcove}. \linebreak\vspace{-0.5cm} 
 \end{thmB}

Theorem B gives the first classification of calibrated representations for the algebra $\cyc_{
 \bs}(n)$ in terms of Young diagrams of multipartitions. Note that other combinatorial classifications in terms of weights and skew-Young diagrams are given in  \cite{MR1988991,MR1976700}.  
Theorem~B can be seen as the analogue of   \cite[Theorem 4.1]{J04} for calibrated representations
--- both results provide the first closed-form description of  the given family of irreducible modules in terms of  multipartitions.  

 \subsection*{Third main result}
The combinatorial description of the unitary representations in Theorem~B leads to our third main result, a multiplicity-free character formula for these representations  and their cohomological construction by way of BGG resolutions.

Over $\mathbb C[q,Q_1,Q_2,\dots, Q_\ell]$, the algebra $\cyc_{q, Q_1, \dots, Q_{\ell}}(n)$ is semisimple and the simple  Specht  modules $S(\bmu)$ are indexed by the set of $\ell$-multipartitions $\bmu$ of $n$.  
For $e > 1$ and    $\bs\in \mathbb Z^\ell$ one can place a corresponding  integral lattice on each Specht module $S(\bmu)$, obtaining a family of (non-semisimple) $\cyc_{\bs}(n)$-modules 
$S_\bs(\bmu)$ by specialisation of the parameters $q = \exp(2\pi\sqrt{-1}/e)$ and $Q_{i} = q^{s_{i}}$.    Our choice of $\bs \in \ZZ^\ell$ allows us to construct the unitary simple $D_\bs(\bla)$ as the head of the Specht module $S_\bs(\bla)$ for $\bla$ as in Theorem~B.     
 A {\em BGG resolution} of $D_\bs(\bla)$ is a resolution of $D_\bs(\bla)$ by a complex whose terms are direct sums of Specht modules $S_\bs(\bmu)$.

Given a unitary simple module $D_\bs(\bla)$, we consider the set of multipartitions $\bmu$ dominating $\bla$ and having the same residue multiset as $\bla$ (with respect to the charge $\bs$). For such $\bmu$ we write $\bmu\trianglerighteq\bla$, see Section
 \ref{sec:combinatorics}. The affine symmetric group $\widehat{\mathfrak{S}}_h$, where $h$ is the number of rows of $\bla$, acts naturally on this set of multipartitions, endowing it with the structure of a graded poset in which $\bla$ is the unique element of length $0$. 
We then construct a BGG resolution of $D_\bs(\bla)$ as follows:   
 
\begin{thmC}
Associated to each unitary simple module, $D_\bs(\bla)$, we have a 
  complex 
  $   C_\bullet (\bla)=   \bigoplus_{
\begin{subarray}c 
\bmu \trianglerighteq \bla  
\end{subarray}
}S_\bs(\bmu)\langle \ell(\bmu)\rangle
  $
with differential given by an alternating  sum over all ``simple reflection homomorphisms".  
This complex 
 is exact except in degree zero, where    $H_0(C_\bullet( \bla ))= D_\bs(\bla).$    
The underlying  graded character is as follows, 
\begin{equation}\label{p-KL-co}\tag{$\dagger\dagger$}
[D_\bs(\bla)] = 
\sum_{
\begin{subarray}c 
\bmu \trianglerighteq  \bla  
\end{subarray}
}  
 (-t)^{\ell(\bmu)}[S_\bs(\bmu)\langle \ell(\bmu)\rangle]
\end{equation}
 Moreover, the module $D_\bs(\bla)$ admits a characteristic-free  basis $\{c_\sts\otimes_\ZZ \Bbbk \mid \sts\in \Path_{ {\aatchpair}}^\mathcal{F}(\la) \}$  where $ \Path_{ {\aatchpair}}^\mathcal{F}(\la)\subseteq \Path_{ {\aatchpair}}  (\la)$ is the subset of paths which never leave the fundamental alcove $\mathcal{F}_\aatchpair$.  
 
\end{thmC}

In the case of the unitary representations of the Hecke algebra of the symmetric group, the resolutions of Theorem~C were the subject of  the authors' previous work \cite{bns}.  Theorem~C vastly generalises this work to all unitary representations of  all cyclotomic Hecke algebras. 

Aspects of this story should be very familiar to the experts: we have an algebraic object
 (in this case a 
 cyclotomic or affine Hecke algebra)
  for which  
 there exists a ``nice family" of irreducible representations (in this case, the unitary representations) which can be combinatorially classified and constructed via explicit bases and BGG resolutions.  
 Analogous stories exist for ladder representations of  $p$-adic groups \cite{BC15b}, 
   finite dimensional representations of  (Kac--Moody) Lie algebras \cite{bgg,BGG2}, 
and homogeneous representations of  antispherical Hecke categories \cite{withemily}.  

We remark that while the definition of a unitary module depends crucially on the ground field being $\mathbb C$, the condition to be calibrated makes sense for arbitrary fields.    In this manner, Theorems~B and C both admit characteristic-free generalisations which are proven in this paper.

\medskip
\noindent{\bf Regular $p$-Kazhdan--Lusztig theory. } The coefficients in \cref{p-KL-co} are equal to ``regular" (or ``non-singular") inverse ($p$-)Kazhdan--Lusztig polynomials.  
In fact, the proof of Theorem B involves passing from the cyclotomic Hecke algebras to the setting of Elias--Williamson's diagrammatic category for ``regular" Soergel-bimodules.  
For the general linear group, ${\rm GL}_h$, the 
``regular" Soergel-bimodules
control the representation theory of the principal block (the block containing the trivial representation $\Bbbk=\Delta(k^h)$ for $n=kh$) if and only if $p> h$.   
On the other side of Schur--Weyl duality this means that ``regular" Soergel-bimodules
control the representation theory of the Serre subcategory of $\Bbbk\mathfrak{S}_n$-${\rm mod} $ corresponding to the poset $\{\la \mid \la\trianglerighteq (k^h) \text{ for }n=kh\}$ and we remark that the simple $\Bbbk\mathfrak{S}_n$-module labelled by the partition $ (k^h)$ is   calibrated providing  $p>h$.  
 
 For higher levels $\ell>1$, one can ask ``to what extent is the cyclotomic Hecke algebra controlled by regular  ($p$-)Kazhdan--Lusztig theory?".  Of course, the Schur--Weyl duality with the general linear group no longer exists. 
 However, one can speak of calibrated representations of the cyclotomic Hecke algebra.  
 In fact, the largest Serre subcategory of (a block of) the cyclotomic Hecke algebra controlled by 
  {\em regular}  $p$-Kazhdan--Lusztig  theory is given by the poset $\{\bmu \mid \bmu \trianglerighteq   \bla\}$ where $D_\bs(\bla  )$ is the minimal calibrated simple module in the block (under the order $ \trianglerighteq $).  
 Thus the Serre quotients carved out by calibrated representations of cyclotomic Hecke algebras   play  the same role as that  of  principal blocks  of algebraic groups for $p>h$. 
 

   
  

  
\medskip
\noindent{\bf Structure of the paper. } 
Section \ref{sec:combinatorics} introduces the combinatorics that will play an important role in this paper. Then we study unitary representations in Section \ref{sec:unitary} where we prove Theorem  A, see Theorem \ref{thm:A}.  Sections \ref{sec:multicali set} and \ref{sec:multicali} are devoted to the proof of Theorem B, which involves intricate combinatorial constructions. In Section \ref{newsec3} we recall previous work of the first author together with A. Cox and A. Hazi \cite{cell4us2} that will allow us to prove Theorem~C. We do this in Section \ref{sec:catBGG}. In this section we also discuss the consequences of our work in the representation theory of rational Cherednik algebras. Finally, in Appendix \ref{sec:level1} we use our techniques to give a complete classification of unitary representations of the Hecke algebra of the symmetric group. While this has mostly appeared in the literature, see \cite{stoica2009unitary}, we believe it gives a good feeling for the usage of calibrated representations in this setting, and corrects an oversight of \cite{stoica2009unitary}. 

\section{Combinatorics}\label{sec:combinatorics}


\subsection{Charges, multipartitions and tableaux} Fix $e\in\mathbb{Z}_{\geq 2}$ throughout this paper.  

\begin{definition}\label{def:cylindricalcharge} We call an $\ell$-tuple of integers $\bs=(s_1,s_2,\ldots, s_\ell)\in\mathbb{Z}^\ell$ an {\em $\ell$-charge} or simply a {\em charge}. Given $e\in\mathbb{Z}_{\geq 2}$, we say that $\bs$ is  {\em cylindrical}  if $s_1\leq s_2\leq\ldots \leq s_\ell< s_1+e$.    
\end{definition}

We define a {\em composition} $\lambda$ of $n$ to be a  finite sequence  of non-negative integers  $ (\lambda_1,\lambda_2, \ldots)$   and whose sum $|\lambda| = \lambda_1+\lambda_2 + \dots$ equals $n$.   We say that $\la$ is a partition if, in addition, this sequence is weakly decreasing.  We let $\la^t$ denote the {\em transpose partition}.  
An    {\em $\ell $-multicomposition} (respectively {\em $\ell$-multipartition} or simply {\em $\ell$-partition})  $\bla=(\lambda^1,\la^2,\ldots,\lambda^\ell)$ of $n$ is an $\ell $-tuple of   compositions (respectively   partitions) such that $|\bla| := |\lambda^1|+|\la^2|+\ldots+ |\lambda^\ell|=n$. 
We will denote the set of $\ell $-multicompositions (respectively $\ell$-partitions) of $n$ by
$\mathcal{C}_{\ell}(n)$ (respectively by $\mathscr{P}_{\ell}(n)$).   
Given  $\bla=(\lambda^1,\lambda^2,\ldots ,\lambda^\ell) \in \mathscr{P}_{\ell}(n)$, the {\em Young diagram} of $\bla$    is defined to be the set of {\em boxes} (or {\em nodes}), 
\[
\{(r,c,m) \mid  1\leq  c\leq \lambda^m_r,\;1
\leq m \leq \ell \}.
\]

We do not distinguish between the multipartition and its Young diagram. We draw the Young diagram of a partition by letting $c$ increase from left to right and $r$ increase from top to bottom. We refer to a box $(r,c,m)$ as being in the $r$th row and $c$th column of the $m$th component of $\bla$.  
We draw the Young diagram of a multipartition by placing the Young diagrams of $\la^1,\ldots,\la^\ell$ side by side from left to right as $m$ runs from $1$ to $\ell$. Finally, a tableau $\SSTT$ on a multipartition $\bla$ is a bijection from the set of boxes of $\bla$ to $\{1, 2, \dots, |\bla|\}$. The tableau $\SSTT$ is called \emph{standard} if it is increasing along the rows and columns of each component. We let $\Std(\bla)$ denote the set of all standard tableaux on $\bla$. If $\SSTT \in \Std(\bla)$ is a standard tableau, then $\Shape(\stt{\downarrow}_{\{1,\dots ,k\}})$ is the multipartition whose Young diagram consists of all the boxes with labels $\leq k$. Finally, we denote by $\varnothing$ the empty multipartition. 

\begin{example}
Let $\ell=3$ and $\bla=(\la^1,\la^2,\la^3)=((2,1),(4,2,1),(5))$. We draw the Young diagram of $\bla$ as follows:
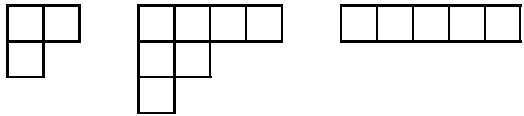
\ytableausetup{mathmode,boxsize=1.2em}
\[ \ydiagram{2,1}\qquad\ydiagram{4,2,1}\qquad\ydiagram{5}\]
\end{example}

\begin{definition} A {\em charged $\ell$-partition} is the data of an $\ell$-partition $\bla$ together with an $\ell$-charge $\bs\in \ZZ^\ell$.  
 Given a box $b=(r,c,m)\in\bla$,  
we define its {\em charged content} to be  $\co^\bs(b)  = s_m+  c - r$ and we define its {\em residue} to be 
$\mathsf{res}^\bs(b)  :=  \co^\bs(b) \pmod e$.  
We refer to a box of residue  $i\in \ZZ/e\ZZ$ as an {\em $i$-box} (or {\em $i$-node}). 
\end{definition}
 Note that the residue of a box in $\bla$ depends on the choice of the charge $\bs\in\mathbb Z^\ell$. 
  For a tableau $\SSTT$ on $\bla$, we let ${{\rm res} (\SSTT)}$ denote  the {\em residue sequence} consisting  of  $\res (\stt^{-1}(k))$ for $k=1,\dots, n$ in order.

\begin{eg}
Let $\ell=3$ and $\bla=(\la^1,\la^2,\la^3)=((2,1),(4,2,1),(5))$. Let $\bs=(-1,2,0)\in\mathbb{Z}^3$. The Young diagram of $\bla$ with its boxes labeled by their charged contents looks as follows:
\[ \ytableausetup{mathmode,boxsize=1.25em}\scalefont{0.9}
\begin{ytableau}
\hbox{--}1&0\\
\hbox{--}2\\
\end{ytableau}
\qquad
\begin{ytableau}
2&3&4&5\\
1&2\\
0\\
\end{ytableau}
\qquad
\begin{ytableau}
0&1&2&3&4\\
\end{ytableau}
\]
Now take $e=4$, then labeling the boxes of the Young diagram of $\bla$ their residues we have:
\[ \ytableausetup{mathmode,boxsize=1.25em}\scalefont{0.9}
\begin{ytableau}
3&0\\
2\\
\end{ytableau}
\qquad
\begin{ytableau}
2&3&0&1\\
1&2\\
0\\
\end{ytableau}
\qquad
\begin{ytableau}
0&1&2&3&0\\
\end{ytableau}
\]
 and we have that 
$ \res(\SSTT_{\bla})= (0,2,1,0,3,2,1,3,2,0,2,0,3,1,0) \in (
 \mathbb Z / e  \mathbb Z)^{15},$ where $\SSTT_{\bla}$ is the reverse column-reading tableau on $\bla$, see Example \ref{ex:revertable} below.  
\end{eg}

\begin{definition} 
Given $\bs\in\mathbb{Z}^\ell$ and two $i$-boxes $(r,c,m), (r',c',m')\in\bla$ for some $i\in \ZZ/e\ZZ$,  
  we write $(r,c,m)\rhd_\bs (r',c',m')$ if 
  $\co^\bs(r,c,m) > \co^\bs(r',c',m')$ or 
    $\co^\bs(r,c,m) = \co^\bs(r',c',m')$ and $m<m'$.  
\color{black}
 For $\bla,\bmu\in \mathscr{P}_{ \ell} (n)$, we write   $ \bmu \trianglelefteq  \bla $  if there is a residue preserving  bijective   map ${\sf A}:[\bla] \to [\bmu]$ 
 such that either $  {\sf A}(r,c,m)\vartriangleleft_{\bs} (r,c,m) $  or $  {\sf A}(r,c,m)=(r,c,m) $ for all $(r,c,m)\in \bla$.  
 When $\bs\in\mathbb{Z}^\ell$ is cylindrical, we often write $\rhd$ instead of $\rhd_\bs$.  
 \end{definition}

\begin{definition}Given a partition $\la=(\la_1,\dots,\la_h)$ such that $\la_h>0$,  
we set $h(\la)$ to be the {\em height} of the partition, that is $h(\la)=h$.  
Given a multipartition $\bla=(\la^1,\la^2,\dots, \la^\ell)$, we set $h(\bla)=(h_1,\dots ,h_\ell)$ to be the $\ell$-composition formed   of the   heights of the component partitions.  
Given $\bla=(\la^1,\la^2,\dots, \la^\ell)$, we   define 
the {\em height} of the multipartition to be the integer $h_{\bla}=h_1+\dots +h_\ell$.  
\end{definition}

\begin{defn} 
Fix ${\bs } \in \ZZ^\ell$. We say that    a composition   $\underline{ h }= ( h _1,\dots, h _\ell) $ of $h$ is {\em ${\bs }$-admissible} if   $ h _m\leq   s_{m}-s_{m-1 } $ for $1<  m \leq \ell$ and  $ h _1\leq   e+s_1-s_{\ell}  $ and with at least one of these inequalities being strict.
 We refer to any value $1\leq m\leq \ell$ for which the inequality is strict as a {\em step change}.  
For an  $\bs$-admissible $\underline{ h }\in \mathbb N^\ell$, we let  
$\mathscr{P}_\aatchpair (n)$ denote the set of all  $\bla\in \mathscr{P}_\ell (n)$ such that $h(\bla)=\underline{ h }$.  
\end{defn}

  \begin{defn}\label{revertable} 
  Let $\aatchpair\in\mathbb{N}^\ell$ be $\bs$-admissible and let $1\leq m \leq \ell$. 
   Given $\bla \in \mathscr{P}_{\aatchpair}(n)$, 
we define the {\em reverse column reading tableau}, $\SSTT_{(m,\bla)}$,
 to be the standard tableau obtained by
  filling the first column of the $(m-1)$th component
  then  the first column of the $(m-2)$th component
  and so on until finally filling 
   the first column of the $m$th component
   and then repeating this procedure on the second columns, the third columns, etc.  
   \end{defn}

\begin{eg}\label{ex:revertable}  Definition \ref{revertable} is best illustrated via an example.
Let $\bs =(0,3,4		)$ with $e=7$.
We have that  $\la =((2,1),(4,2,1),(5))\in \mathscr{P}_{(2,3,1)}(15)$ 
and that $h_1=2<7+0-4=e+s_1-s_\ell$ and so $m=1$ is  a  step change (one can check that 
this step change is unique).  The reverse reading  column reading tableau 
$\SSTT_{1,\bla}$ is given by filling the first columns of the $3$rd component, $2$nd component, $1$st component, and then the 2nd columns of these components\dots hence obtaining 
   \[ 
\SSTT_{(1,\bla)}=\ytableausetup{mathmode,boxsize=1.2em}\scalefont{0.9}
\left(\begin{minipage}{1.15cm}\begin{ytableau}
5&10\\
6\\
\end{ytableau}\end{minipage}\;\;,
\quad
\begin{minipage}{2.05cm}
\begin{ytableau}
2&8&12&14\\
3&9\\
4\\
\end{ytableau}\end{minipage}\;\;,
\quad
\begin{minipage}{2.5cm}
\begin{ytableau}
1&7&11&13&15\\
\end{ytableau}\end{minipage}\right).
\]
 \end{eg}

 Given $\bla \in \mathscr{P}_\aatchpair(n)$, we let ${\rm Add}_\aatchpair(\bla)$ (respectively 
${\rm Rem}_\aatchpair(\bla)$) denote the set of all  addable (respectively removable) boxes of the Young diagram so that the resulting Young diagram is the Young diagram  of a multipartition 
belonging to $\mathscr{P}_\aatchpair(n+1)$ (respectively $\mathscr{P}_\aatchpair(n-1)$).  
   We let ${\rm Add}^i_\aatchpair(\bla)$ (respectively 
${\rm Rem}^i_\aatchpair(\bla)$) denote the subset of nodes of residue equal to $i\in \ZZ/e\ZZ$.  
 Dropping  the subscript $\aatchpair$ we obtain the usual sets of addable and removable $i$-nodes of a multipartition.
 
  \begin{definition}
   Given $1\leq k\leq n$ and a standard tableau $\stt$ on a multipartition $\bla$, we let ${\mathcal A} _\stt(k)$, 
(respectively ${\mathcal R} _\stt(k)$)  denote the set of   all addable $\res (\stt^{-1}(k))$-boxes (respectively all
removable   $\res (\stt^{-1}(k))$-boxes)  of the multipartition  $\Shape(\stt{\downarrow}_{\{1,\dots ,k\}})$ which
   are more dominant  than  $\stt^{-1}(k)$.
        We define the   degree of $\stt\in \Std(\bla)$ for $\bla \in  \mathscr{P} (n)$  as follows,
$$
\deg_{\bs }  (\stt) = \sum_{k=1}^n \left(	|{\mathcal A} _\stt(k)|-|{\mathcal R} _\stt(k)|	\right).
$$
\end{definition}

\begin{rmk} 
Let $\aatchpair \in \mathbb N^\ell$  be ${\bs }$-admissible and $1\leq m \leq \ell$ be a step change.   
We have  that $\deg_{\bs } (\SSTT_{(m,\bla)})=0$  for $\bla \in \mathscr{P}_{ \aatchpair}(n)$.  
\end{rmk}

\subsection{The $\sle$-crystal}\label{sec:crystal}

Fix $e\geq 2$ and a charge $\bs \in\mathbb{Z}^\ell$. 
The $\sle$-crystal is a simply directed graph on the set of vertices consisting of all $\ell$-partitions. Its arrows are given by a rule for adding at most one box of each residue $i\in\mathbb{Z}/e\mathbb{Z}$ to a given $\ell$-partition. 
%
%
%
We define the {\em crystal operators} $\tilde{e}_i$ and $\tilde{f}_i$, which remove and add (respectively) at most one box of residue $i$.
\begin{definition}
{{\cite[Theorem 2.8]{FLOTW99}  }} \label{def:crystalops} 
Fix $i\in\mathbb{Z}/e\mathbb{Z}$, $\bs\in\mathbb{Z}^\ell$, and an $\ell$-partition $\bla$.
\begin{itemize} [leftmargin=*]
\item Form the {\em $i$-word} of $\bla$ by listing all addable and removable $i$-boxes of $\bla$ in increasing order from left to right (according to $\rhd_\bs$) and then replacing each addable box in the list by the symbol $+$ and each removable box by the symbol $-$. 
\item Next, find the
 {\em reduced $i$-word} of $\bla$ by recursively canceling all adjacent pairs $(-+)$ in the $i$-word of $\bla$. The reduced $i$-word is of the form $(+)^a(-)^b$ for some $a,b\in\mathbb{Z}_{\geq 0}$.
 \item Define $\tilde{f}_i$ as the operator that adds the addable $i$-box to $\bla$ corresponding to the rightmost $+$ in the reduced $i$-word of $\bla$. If the reduced $i$-word of $\bla$ contains no $+$ then we declare that $\tilde{f}_i\bla=0$. Likewise, define $\tilde{e}_i$ as the operator that removes the removable $i$-box from $\bla $ corresponding to the leftmost $-$ in the reduced $i$-word of $\bla$. If the reduced $i$-word of $\bla$ contains no $-$ then we declare that $\tilde{e}_i\bla=0$.
 \end{itemize}
 The directed graph with vertices all $\ell$-partitions and arrows $\bla\stackrel{i}{\rightarrow}\bmu$ if and only if $\bmu=\tilde{f}_i\bla$, $i\in\mathbb{Z}/e\mathbb{Z}$, is called the {\em $\widehat{\mathfrak{sl}}_e$-crystal}. We have $\tilde{f}_i\bmu=\bla$ if and only if $\tilde{e}_i\bla=\bmu$.
 \end{definition}
 
  \begin{definition}

The $i$-box that is added by $\tilde{f}_{i}$, if it exists, is called a \emph{good} addable $i$-box of $\bla$. Similarly, the box that is removed by $\tilde{e}_{i}$ is called a \emph{good} removable $i$-box of $\bla$.
  \end{definition}

\begin{remark}
Given $i\in\mathbb{Z}/e\mathbb{Z}$, if $\bla$ has only one removable $i$-box $b$ and no addable $i$-box of charged content greater than or equal to $\co^\bs(b)$, then $\tilde{e}_i(\bla)=\bla\setminus\{b\}$. We will use this observation without further mention.
\end{remark}

The $\sle$-crystal is in general disconnected and we will be interested in its connected component containing the empty multipartition $\varnothing$. The vertices in this connected component of the  $\widehat{\mathfrak{sl}}_e$-crystal do not, in general, admit closed formulas   (one must instead search for a sequence of good nodes by repeated use of Definition \ref{def:crystalops}).  However, for our choice of cylindric charge $\bs \in \ZZ^\ell$ (as in Definition \ref{def:cylindricalcharge}) we have the following combinatorial description from \cite{FLOTW99}.  

\begin{definition}\label{def:FLOTW}
  An $\ell$-partition $\bla$ with charge $\bs\in\mathbb{Z}^\ell$ is {\em FLOTW} if the following conditions hold:
\begin{enumerate}[leftmargin=*]
\item the charge $\bs$ is cylindrical,
\item
The multipartition $\bla$ is {\em cylindrical}, that is: \begin{enumerate}
\item $\la^j_k\geq \la^{j+1}_{k+s_{j+1}-s_j}$ for all $j\in\{1,\ldots, \ell-1\}$ and for all $k\geq 1$,
\item $\la^\ell_k\geq \la^1_{k+e+s_1-s_\ell}$ for all $k\geq 1$,
\end{enumerate}
\item for all $\alpha>0$, the residues modulo $e$ of the rightmost boxes of the rows of size $\alpha$ do not cover $\{0,\ldots,e-1\}$.
\end{enumerate} 
\end{definition}

\section{Unitary representations}\label{sec:unitary}
\subsection{The affine Hecke algebra} Let us recall that the (extended) affine braid group, $\braid_{n}$ has generators $T_{1}, \dots, T_{n-1}, x_{1}, \dots, x_{n}$ subject to the following relations:
\begin{itemize}
\item[(a)] $T_{i}T_{i+1}T_{i} = T_{i+1}T_{i}T_{i+1}$ and  $T_{i}T_{j} = T_{j}T_{i}$ if $|i - j| > 1$.
\item[(b)] $x_{i}x_{j} = x_{j}x_{i}$ for every $i, j = 1, \dots, n$. 
\item[(c)] $T_{i}x_{i}T_{i} = x_{i+1}$, $T_{i}x_{j} = x_{j}T_{i}$ if $j \neq i, i+1$. 
\end{itemize}

We may think of $\braid_{n}$ as the braid group on the cylinder. The element $T_{i}$ correspond to the usual braid generator that   crosses two adjacent strands, and the element $x_{i}$ corresponds to looping the $i$-th strand around the cylinder so that it comes back to the $i$-th position, see Figure \ref{fig:braids} below.

\begin{figure}[h!]
    \centering
     $$
x_i=
\begin{minipage}{7cm}
 \begin{tikzpicture}[scale=0.9]

 \draw(0,0) rectangle (7.5,2);

\draw[thick](0,1)--(2.9,1) to [out=0,in=90] (3.5,0);
\draw[thick](0,1)--(2.9,1) to [out=0,in=90] (3.5,0);

 \foreach \i in {0.5,1,1.5,
 3.5,4,4.5,3,
 6,7,6.5}
{
 
 \fill(\i,0) circle (1.5pt);
\fill(\i,2) circle (1.5pt);
 }

 \foreach \i in {0.5,1,1.5,3}
{\fill[white] (\i,1) circle (2.5pt);}

 \foreach \i in {0.5,1,1.5,
 4,4.5,3,
 6,7,6.5}
{
\draw[thick](\i,0)--(\i,2);
}

 \fill[white](1.85,-0.2) rectangle (2.65,0.2); 
 \draw[densely dotted] (1.85,0) -- (2.65,0); 
 \fill[white](1.85,2-0.2) rectangle (2.65,2.2); 
 \draw[densely dotted] (1.85,2) -- (2.65,2);

 \fill[white](1.85+3,-0.2) rectangle (2.65+3,0.2); 
 \draw[densely dotted] (1.85+3,0) -- (2.65+3,0); 
 \fill[white](1.85+3,2-0.2) rectangle (2.65+3,2.2); 
 \draw[densely dotted] (1.85+3,2) -- (2.65+3,2);

  \foreach \i in {3.5,4,4.5,6,6.5,7}
{\fill[white] (\i,1) circle (2.5pt);}

 \draw[thick](3.5,2) to [out=-90,in=180]  (4.1,1) -- (7.5,1); 

 \draw[very thick,white,densely dotted]   (1.85,1) -- (2.55,1); 

 \draw[very thick,white,densely dotted]   (5.55,1) -- (4.95,1);

  \end{tikzpicture}
 \end{minipage}  \qquad 
 T_i=
\begin{minipage}{8cm}
 \begin{tikzpicture}[scale=0.95=]

 \draw(0,0) rectangle (7.5,2); 
 \foreach \i in {0.5,1,1.5,
 3.5,4,4.5,3,
 6,7,6.5}
{
 
 \fill(\i,0) circle (1.5pt);
\fill(\i,2) circle (1.5pt);
 }

 \foreach \i in {0.5,1,1.5,
 4.5,3,
 6,7,6.5}
{
\draw[thick](\i,0)--(\i,2);
}

 \draw[thick] (4,0) --(3.5,2);
 
\draw[thick] (3.5,0) --(4,2)coordinate [pos=0.25] (hello)
coordinate [pos=0.75] (hello2) ; 

\draw[line width=2.5,white] (hello)--(hello2);
\draw[thick] (3.5,0) --(4,2);

 \fill[white](1.85,-0.2) rectangle (2.65,0.2); 
 \draw[densely dotted] (1.85,0) -- (2.65,0); 
 \fill[white](1.85,2-0.2) rectangle (2.65,2.2); 
 \draw[densely dotted] (1.85,2) -- (2.65,2);

 \fill[white](1.85+3,-0.2) rectangle (2.65+3,0.2); 
 \draw[densely dotted] (1.85+3,0) -- (2.65+3,0); 
 \fill[white](1.85+3,2-0.2) rectangle (2.65+3,2.2); 
 \draw[densely dotted] (1.85+3,2) -- (2.65+3,2);

  \end{tikzpicture}
 \end{minipage}\quad 
$$

    \caption{The braid generators $x_{i}$ and $T_{i}$. We remark that these are braids on a cylinder, so the left and right-sides of the rectangles are to be identified.}
    \label{fig:braids}
\end{figure}
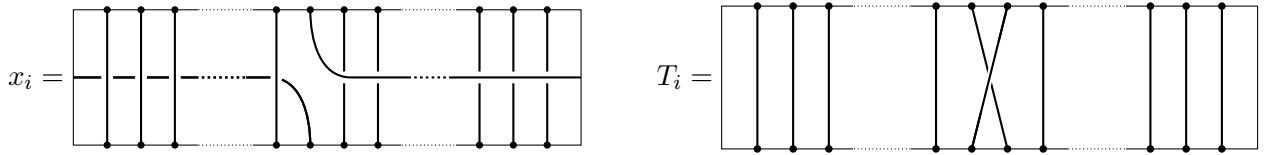
The affine Hecke algebra is the quotient of the group algebra of $\braid_{n}$ by a skein-type relation: 

\begin{definition}
Let $\RR$ be a domain and $q \in \RR^{\times}$, $q \neq \pm 1$. The (extended) affine Hecke algebra $\AHA_{q}(n, \RR)$ is the quotient of the group algebra $\RR\braid_{n}$ by the relations
$$
(T_{i} + 1)(T_{i} - q) = 0
$$
  for $i = 1, \dots, n-1$. When the domain $\RR$ is clear from the outset, we will simply denote the affine Hecke algebra by $\AHA_{q}(n)$. 
\end{definition}

\begin{remark}
 It is customary to define the elements $X_{i} := q^{1-i}x_{i}$, so that in $\AHA_{q}(n)$ we have the relation $T_{i}X_{i}T_{i} = qX_{i+1}$. 
\end{remark}
\begin{remark}
The finite Hecke algebra $H_{q}(n)$ can be realized as a subalgebra of $\AHA_{q}(n)$ generated by $T_{1}, \dots, T_{n-1}$. It can also be realized as the quotient $\AHA_{q}(n)/(X_1 - 1)$. This is akin to the finite braid group $B_{n}$ being both a subgroup and a quotient group of $\braid_{n}$. 
\end{remark}
\begin{remark}
The elements $X_{1}, \dots, X_{n}$ are known as the \emph{Jucys-Murphy} elements of $\AHA_{q}(n)$. We will call the algebra $\JM := \RR[X_{1}^{\pm 1}, \dots, X_{n}^{\pm 1}] \subseteq \AHA_{q}(n)$ the \emph{Jucys-Murphy} subalgebra. It is isomorphic to the algebra of Laurent polynomials in $n$ variables and we have a  vector space decomposition $\AHA_{q}(n) = H_{q}(n) \otimes_{\RR} \JM$. 
\end{remark}

When $\RR$ is a domain of characteristic zero it is known, see e.g. \cite[Proposition 7.1.14]{MR2838836}, that the center of $\AHA_{q}(n)$ is $Z(\AHA_{q}(n)) = \JM^{S_{n}}$, the algebra of symmetric Laurent polynomials in the Jucys-Murphy elements. Thus, $\AHA_{q}(n)$ is finite over its center and, when $\RR$ is a field $\FF$ of characteristic zero, every irreducible representation of $\AHA_{q}(n)$ is finite-dimensional. If moreover $\FF$ is algebraically closed it follows, looking at the eigenvalues of $X_{1}$, that every irreducible representation of $\AHA_{q}(n)$ factors through an algebra of the form
$$\cyc_{q, Q_1, \dots, Q_{\ell}}(n) := \frac{\AHA_{q}(n)}{\left(\prod_{i}^{\ell}(X_{1} - Q_{i})\right)}
$$
This is known as the \emph{cyclotomic Hecke algebra}, or the \emph{Ariki-Koike algebra}. It is a finite-dimensional $\FF$-algebra, of dimension precisely $n!\ell^{n}$, \cite{ak94}. 

\subsection{Unitary representations} For this subsection, we let $\RR = \C$, the  complex field. We make the convention that a Hermitian form on a finite-dimensional complex vector space is linear on the first variable and conjugate-linear on the second. Recall that we have defined the affine Hecke algebra $\AHA_{q}(n)$ as a quotient of the group algebra $\C\braid_{n}$, so the following notion makes sense.

\begin{definition}\label{def:unitary}
We say that a finite-dimensional $\AHA_{q}(n)$-representation is unitary if it admits a positive-definite $\braid_{n}$-invariant Hermitian form. 
\end{definition}

Let us remark that not all affine Hecke algebras admit nontrivial unitary representations, where by non-trivial we mean that at least one $T_{i}$ does not act by $-1$. Indeed, the parameter $q$ plays an essential role in here.

\begin{lemma}\label{lem: q unit}
Assume that $\AHA_{q}(n)$ admits a nontrivial unitary representation. Then, $q \in \C^{\times}$ lies in the unit circle. 
\end{lemma}
\begin{proof}
Let $M$ be a nontrivial unitary $\AHA_{q}(n)$-representation. Let $m \in M$ be an eigenvector for $T_{1}$ with eigenvalue $q$, which we know exists because $M$ is nontrivial. Then
$$
0 \neq q\langle m, m \rangle = \langle T_{1}m, m \rangle = \langle m, T_{1}^{-1}m\rangle = \overline{q^{-1}}\langle m, m\rangle
$$
 so that $q = \overline{q^{-1}}$. Thus, $q$ is in the unit circle.
\end{proof}

\begin{remark}\label{rmk: unit circle}
We remark that Lemma \ref{lem: q unit} follows from the following more general result. Let $G$ be any group and let $M$ be a finite-dimensional unitary representation of $G$. Then, every eigenvalue of $g$ on $M$ lies in the unit circle, for any $g \in G$. Note that this statement is obvious when $G$ is a finite group as in this case every eigenvalue of $g$ is, in fact, a root of unity. The proof in the general case is just as that of Lemma \ref{lem: q unit}.
\end{remark}

Since every cyclotomic Hecke algebra $\cyc_{q, Q_1, \dots, Q_{\ell}}(n)$ is a quotient of the affine Hecke algebra it makes sense to speak about unitary representations of $\cyc_{q, Q_1, \dots, Q_{\ell}}(n)$. Just as in Lemma \ref{lem: q unit}, we have that if $\cyc_{q, Q_1, \dots, Q_{\ell}}(n)$ admits a unitary representation then all complex numbers $q, Q_1, \dots, Q_{\ell} \in \C^{\times}$ must lie in the unit circle. 

Let us remark that the cyclotomic Hecke algebra $\cyc_{q, Q_1, \dots, Q_{\ell}}(n)$ is a quotient of the group algebra of a different braid group: the braid group $B(\ell, 1, n)$, which is defined as follows. Let $G(\ell, 1, n) := S_{n} \ltimes (\Z/\ell\Z)^{n}$ be the cyclotomic group. If we think of this group as the group of $n \times n$ permutation matrices whose nonzero entries are $\ell$-roots of unity, we get an action of $G(\ell, 1, n)$ on $\fh := \C^{n}$. Let $\fh^{reg}$ be the locus where this action is free, it can be shown that this is the complement of a hyperplane arrangement in $\C^n$. Then,
$$
B(\ell, 1, n) := \pi_{1}(\fh^{reg}/G(\ell, 1, n)).
$$

For example, when $\ell = 1$, the group $B(1, 1, n)$ is the usual Artin braid group on $n$ strands. It is clear from the definitions, see e.g. \cite{MR1637497}, that a $\cyc_{q, Q_1, \dots, Q_{\ell}}(n)$-representation is unitary if and only if it admits a positive-definite $B(\ell, 1, n)$-invariant Hermitian form. On the other hand, Rouquier has shown using the representation theory of rational Cherednik algebras, \cite[Proposition 4.5.4]{MR4054879}, that every irreducible $\cyc_{q, Q_1, \dots, Q_{\ell}}(n)$-representation admits a (unique up to $\mathbb{R}^{\times}$-scalars) non-degenerate $B(\ell, 1, n)$-invariant Hermitian form. Thus, the question of unitaricity is that of \emph{positive-definiteness} of this form. 

\subsection{Unitary representations via $*$-products} In \cite{barbasch2015star}, Barbasch and Ciubotaru define a class of representations that are closely related to the unitary representations we study in this paper. The goal of this section is to explore this relation. In order to do so, throughout this subsection $\RR = \C[\qq, \qq^{-1}]$ and $\AHA_{\qq}(n) := \AHA_{\qq}(n, \RR)$. 

\begin{definition}
A $*$-product on $\AHA_{\qq}(n)$ is a conjugate-linear, involutive antiautomorphism of $\AHA_{\qq}(n)$. Given a $*$-product on $\AHA_{\qq}(n)$, a representation $M$ of $\AHA_{\qq}(n)$ is called $*$-unitary if it admits a positive-definite Hermitian form which is $\AHA_{\qq}(n)$-invariant in the sense that
$$
\langle am, m'\rangle = \langle m, a^{*}m'\rangle
$$
for every $m, m' \in M$, $a \in \AHA_{\qq}(n)$. 
\end{definition}

In \cite{barbasch2015star}, the authors define  $*$-products on $\AHA_{\qq}(n)$ and study the notion of unitary representations for these $*$-products. These are the first two $*$-products of the following proposition.

\begin{proposition}\label{prop:star products}
The following formulas define $*$-products on $\AHA_{\qq}(n)$.
\begin{enumerate}
\item $T_{i}^{\bullet} = T_{i}$, $X_{i}^{\bullet} = X_{i}$, $\qq^{\bullet} = \qq$.
\item $T_{i}^{\star} = T_{i}$, $X_{i}^{\star} = T_{w_{0}}X^{-1}_{n-i+1}T_{w_{0}}^{-1}$,$\qq^{\star} = \qq$, where $T_{w_{0}} = T_{1}T_{2}\cdots T_{n-1}T_{1}T_{2}\cdots T_{n-2}\cdots T_{1}T_{2}T_{1}$.  
\item $T_{i}^{\dagger} = T_{i}^{-1}$, $X_{i}^{\dagger} = X_{i}^{-1}$, $\qq^{\dagger} = \qq^{-1}$. 
\end{enumerate}
\end{proposition}
\begin{proof}
For $\bullet$ and $\star$, see \cite[Definition 2.3.1]{barbasch2015star}.. For $\dagger$, note that the relation $(T_{i}+1)(T_{i}-\qq) = 0$ is equivalent to $(T_{i}^{-1} + 1)(T_{i}^{-1} - \qq^{-1}) = 0$, so this is preserved under $\dagger$.  It is also clear that $\dagger$ preserves the relations among the $X_{i}$, as well as the braid relations among the $T_{i}$. Finally, we have
$$
(T_{i}X_{i}T_{i})^{\dagger} = T_{i}^{\dagger}X_{i}^{\dagger}T_{i}^{\dagger} = T_{i}^{-1}X_{i}^{-1}T_{i}^{-1} = (T_{i}X_{i}T_{i})^{-1} = \qq^{-1}X_{i+1}^{-1} = (\qq X_{i+1})^{\dagger}
$$
  which finishes the proof.
\end{proof}

\begin{rmk}
In \cite[Definition 2.4.1]{barbasch2015star}, the authors term an $*$-product \emph{admissible} if $\qq^{*} = \qq$ and $T_{i}^{*} = T_{i}$ for $i = 1, \dots, n-1$. So the $*$-products $\bullet$ and $\star$ of Proposition \ref{prop:star products} are admissible, while $\dagger$ is not. It is conjectured, \cite[Conjecture 2.4.2]{barbasch2015star} that $\bullet$ and $\star$ are essentially the only admissible $*$-products on $\AHA_{\qq}(n)$. 
\end{rmk}

It is clear that a representation $M$ of $\AHA_{\qq}(n)$ is unitary in the sense of Definition \ref{def:unitary} if and only if it is $\dagger$-unitary in the sense of \cite{barbasch2015star}. Representations that are $\star$-unitary and $\bullet$-unitary have been studied in \emph{loc. cit.}, where they are shown to be related to unitary representations of general linear groups over $p$-adic fields. Note that both $\bullet$ and $\star$ commute with the action of $\qq$, while $\dagger$ does not. Nevertheless, the $*$-products $\bullet$ and $\dagger$ are closely related, as witnessed by the following result.

\begin{proposition}
The map $\iota$ defined on generators by $\iota(T_{i}) = T_{i}^{-1}, \iota(X_{i}) = X_{i}^{-1}$ and $\iota(\qq) = \qq^{-1}$ extends to a $\C$-linear involutive algebra automorphism of $\AHA_{\qq}(n)$. Moreover, $\iota \circ \dagger = \bullet = \dagger \circ \iota$. 
\end{proposition}
\begin{remark}
From the formulas, it may seem that $\iota$ and $\dagger$ coincide. Note, however, that $\iota$ is an automorphism while $\dagger$ is an \emph{anti}-automorphism. Also, $\iota$ is $\C$-linear while $\dagger$ is conjugate-linear. 
\end{remark}
\begin{proof}
That $\iota$ defines a $\C$-linear involutive algebra automorphism is easy to check from the relations on $\AHA_{\qq}(n)$. Since $\iota$ is $\C$-linear while $\dagger$ is conjugate linear, both compositions $\iota \circ \dagger$ and $\dagger \circ \iota$ are conjugate linear. Also, since $\iota$ is an automorphism while $\dagger$ is an antiautomorphism, both $\iota \circ \dagger$ and $\dagger \circ \iota$ are antiautomorphisms. Thus, it suffices to check the equality $\iota \circ \dagger = \bullet = \dagger \circ \iota$ on generators $T_{i}, X_{i}$ and $\qq$.
$$
\iota(T_{i}^{\dagger}) = \iota(T_{i}^{-1}) = T_{i} = T_{i}^{\bullet}, \iota(X_{i}^{\dagger}) = \iota(X_{i}^{-1}) = X_{i} = X_{i}^{\bullet}, \iota(\qq^{\dagger}) = \iota(\qq^{-1}) = \qq = \qq^{\bullet}
$$
$$
\iota(T_{i})^{\dagger} = (T_{i}^{-1})^{\dagger} = T_{i} = T_{i}^{\bullet}, \iota(X_i)^\dagger = (X_i^{-1})^{\dagger} = X_i = X_i^{\bullet},
\iota(\qq)^{\dagger} = (\qq^{-1})^{\dagger} = \qq = \qq^{\bullet}
$$
 and the result follows. 
\end{proof}

We remark, however, that we do not obtain an equivalence between categories of $\dagger$-unitary and $\bullet$-unitary representations. To do this, we must find an automorphism $\varphi$ of $\AHA_{\qq}(n)$ satisfying $\varphi \circ \dagger \circ \varphi^{-1} = \bullet$. It is unlikely that such an automorphism exists, at least in the algebraic setting. Roughly speaking, if an irreducible representation is $\dagger$-unitary then we must have that $\qq$ acts by a complex number of norm $1$ while, if a representation is $\bullet$-unitary then $\qq$ must act by a \emph{real} number. It is possible that one may obtain equivalences between $\dagger$-unitary and $\bullet$-unitary representations by working with representations of the formal affine Hecke algebra, \cite{formalaffine} so that we can take exponential and logarithms of the variable $\hbar$ (that should be thought of as $\log(\qq)$), but we do not do it here.

\subsection{Unitary representations are calibrated} We go back to the setting of $\RR = \C$. Moreover, throughout this section we let $q \in \C^{\times}$ be in the unit circle. Let us recall the following important notion from the representation theory of affine Hecke algebras, \cite{MR1976700}.

\begin{definition}
A $\AHA_{q}(n)$-representation $M$ is called \emph{calibrated} if the Jucys-Murphy subalgebra $\JM$ acts semisimply on $M$.
\end{definition}

\begin{lemma}
Let $M$ be a unitary and finite-dimensional $\AHA_q(n)$-module. Then $M$ is semisimple and calibrated.
\end{lemma}
\begin{proof}
Let $N \subseteq M$ be an $\AHA_q$-submodule of $M$. It is immediate to see that the Hermitian orthogonal $N^{\perp}$ is a complement to $N$ in $M$. Thus, every submodule of $M$ splits. The proof of semisimplicity for the $\JM$-action is the same, after observing that $\dagger$ preserves the Jucys-Murphy subalgebra.
\end{proof}

For a $\AHA_{q}(n)$-module and a vector $\ba = (a_1, \dots, a_n) \in (\C^{\times})^n$, we define the $\ba$-weight space to be
$$
M_{\ba} := \{m \in M \mid X_{i}m = a_{i}m \; \text{for every} \; i = 1, \dots, n\}.
$$

Thus, every unitary $\AHA_q(n)$-module is the direct sum of its weight spaces. 

\begin{lemma}\label{lemma:orthogonality}
Let $M$ be unitary, with invariant Hermitian form $\langle \cdot, \cdot\rangle$. Then,

\begin{enumerate}
\item $M_{\ba} = 0$ unless $|a_{1}| = \cdots = |a_{n}| = 1$.
\item $M_{\ba}$ is orthogonal to $M_{\bb}$ whenever $\bb \neq \ba$. 
\end{enumerate}
\end{lemma}
\begin{proof}
Statement (1) is clear, see e.g. Remark \ref{rmk: unit circle}. To prove (2), assume that both $M_{\ba}$ and $M_{\bb}$ are both nonzero, let $m_1 \in M_{\ba}$, $m_2 \in M_{\bb}$ and $i$ such that $a_{i} \neq b_{i}$. Then,
$$
a_i\langle m_1, m_2\rangle = \langle X_{i}m_1, m_2\rangle = \langle m_1, X_{i}^{-1}m_2\rangle = \overline{b_{i}^{-1}}\langle m_1, m_2\rangle.
$$
Since $a_{i} \neq b_{i}$ and, by (1), $|a_i| = |b_i| = 1$, we obtain $\overline{a_i} \neq b_{i}^{-1}$. Thus, we have $\langle m_1, m_2\rangle = 0$, as required. 
\end{proof}

\subsection{When are calibrated representations unitary?} The purpose of this section is to obtain necessary and sufficient conditions for a calibrated representation to be unitary. We have seen one condition a calibrated representation must satisfy in order to be unitary: all the weights appearing in it must have values in the unit circle. This necessary condition is, however, not sufficient. To obtain a complete answer we must first recall the classification of irreducible calibrated representations in terms of the weights that appear in them. This is from \cite{MR1976700} and we follow \emph{loc. cit} closely. 

\begin{definition}[\cite{MR1976700}]
We say that a weight $\ba = (a_{1}, \dots, a_{n}) \in (\C^{\times})^{n}$ is  \emph{calibrated} if the following condition holds: 
\begin{center}
for every $i < j$ such that $a_{i} = a_{j}$, $qa_{i}, q^{-1}a_{i} \in \{a_{i+1}, \dots, a_{j-1}\}$.
\end{center}
Since we are assuming that $q \neq \pm 1$ note, in particular, that if $\ba$ is a calibrated weight, then $a_{i} \neq a_{i+1}$, $a_{i} \neq a_{i+2}$ for every $i$. We denote by $\cal^{\aff} \subseteq (\C^{\times})^{n}$ the set of all calibrated weights. 
\end{definition}

For a calibrated weight $\ba \in \cal^{\aff}$, we say that $s_{i} = (i, i+1) \in S_{n}$ is an \emph{admissible transposition} if $a_{i+1} \neq q^{\pm1}a_{i}$. Note that, if $s_{i}$ is an admissible transposition, then $s_{i}\ba \in \cal^{\aff}$.

\begin{definition}
Define the equivalence relation $\sim$ on $\cal^{\aff}$ by saying that $\ba \sim \bb$ if $\ba$ can be reached from $\bb$ by applying a sequence of admissible transpositions.
\end{definition} 

The following result is due to Ram, \cite{MR1976700}. It can be thought of as asserting the existence of a Young seminormal form for calibrated representations.

\begin{theorem}\label{thm:ram}
Let $[\ba] \in \cal^{\aff}/\sim$ be an equivalence class. Then, there exists an irreducible calibrated module $M_{[\ba]}$ whose weights are precisely $[\ba]$. More precisely, $M$ has a basis $w_{\bb}$, $\bb \in [\ba]$, and the action of $\AHA_{q}(n)$ is given as follows:
$$
X_{i}w_{\bb} = b_{i}w_{\bb}, \hspace{1cm} T_{i}w_{\bb} = \frac{b_{i+1}(q-1)}{b_{i+1} - b_{i}}w_{\bb} + (1 - \delta_{qb_{i},  b_{i+1}})(1 - \delta_{q^{-1}b_{i}, b_{i+1}})\left[-q + \frac{b_{i+1}(q-1)}{b_{i+1}-b_{i}}\right]w_{s_{i}\bb}
$$
 Where we define $w_{s_{i}\bb} = 0$ if $s_{i}$ is not an admissible transposition for $\bb$. Moreover, every irreducible calibrated module is of the form $M_{[\ba]}$ for a unique equivalence class $[\ba] \in \cal^{\aff}/\sim$. 
\end{theorem}

We will characterize the calibrated weights $\ba$ for which $M_{[\ba]}$ is indeed unitary. Since every unitary representation is semisimple, this would give us all the unitary representations of $\AHA_{q}(n)$. First, we have seen in Lemma \ref{lemma:orthogonality} that a necessary condition is that $\ba \in (S^1)^{n}$, where $S^1 \subseteq \C$ is the unit circle. Note that Theorem \ref{thm:ram} tells us that  every weight space in $M_{[\ba]}$ is $1$-dimensional and, as we have seen in Lemma \ref{lemma:orthogonality}, different weight spaces must be orthogonal under an invariant, positive-definite Hermitian form. 

So let $\langle \cdot, \cdot \rangle$ be an invariant, non-degenerate Hermitian form on $M_{[\ba]}$ and assume that $\ba \in (S^1)^n$. By \cite[Proposition 4.5.4]{MR4054879}, such a form exists. By our discussion in the previous paragraph, there exist numbers $A_{\bb} \in \R$, $\bb \in [\ba]$, such that 
\begin{equation}\label{eqn:form}
\langle w_{\bb}, w_{\bb'}\rangle = A_{\bb}\delta_{\bb, \bb'}
\end{equation}
  and our job is to find conditions on $[\ba]$ guaranteeing that all numbers $A_{\bb}$ can be chosen to have the same sign. We separate in several cases, and we make heavy use of the Young seminormal form of Theorem \ref{thm:ram}. 

\medskip
\noindent {\it Case 1. $\bb' \neq \bb, s_{i}\bb$.} In this case, we have $\langle T_{i}w_{\bb}, w_{\bb'}\rangle = \langle w_{\bb}, T_{i}^{-1}w_{\bb'}\rangle = 0$.

\medskip
\noindent 
{\it Case 2. $\bb' = \bb$. } In which case $$\langle T_{i}w_{\bb}, w_{\bb}\rangle = A_{\bb}\displaystyle{\tfrac{b_{i+1}(q-1)}{b_{i+1} - b_{i}}}.$$ On the other hand, since $T_{i}^{-1} = q^{-1}(T_{i} + 1-q)$ then, using the fact that $\conj{q} = q^{-1}$, $\conj{b_{i+1}} = b_{i+1}^{-1}$ and $\conj{b_{i}} = b_{i}^{-1}$ we have that 
\begin{align*}
\langle w_{\bb}, T_{i}^{-1}w_{bb}\rangle
& = q\langle w_{\bb}, (T_{i} + 1-q)w_{\bb}\rangle\\
 &= qA_{\bb}\conj{\left[\tfrac{b_{i+1}(q-1)}{b_{i+1} - b_{i}} + (1-q)\right]}
 \\
 &  = qA_{\bb}\conj{\left[\tfrac{b_{i}(q-1)}{b_{i+1} - b_{i}}\right]} 
 \\&= 
   A_{\bb}\tfrac{b_{i+1}(1-q)}{b_{i} - b_{i+1}} 
   \\&=  \langle T_{i}w_{\bb}, w_{\bb}\rangle.  
  \end{align*}

\medskip
\noindent 
{\it Case 3. $\bb' = s_{i}\bb$.} Since in this case $s_{i}\bb$ is defined, we have that $b_{i+1} \neq q^{\pm1}b_{i}$, so 
$$
\langle T_{i}w_{\bb}, w_{s_{i}\bb}\rangle = A_{s_{i}\bb}\left(-q + \tfrac{b_{i+1}(q-1)}{b_{i+1}-b_{i}}\right) = A_{s_{i}\bb} \left(\tfrac{qb_{i} - b_{i+1}}{b_{i+1} - b_{i}} \right).
$$
On the other hand,
\begin{align*}
\langle w_{\bb}, T_{i}^{-1}w_{s_{i}\bb}\rangle & = q\langle w_{\bb}, T_{i}w_{s_{i}\bb}\rangle \\
& = qA_{\bb}\left[\conj{-q + \left(\tfrac{b_{i}(q - 1)}{b_{i} - b_{i+1}}\right)}\right]  \\
& = qA_{\bb}\left[\tfrac{\conj{q}\conj{b_{i+1}} - \conj{b_{i}}}{\conj{b_{i}} - \conj{b_{i+1}}}\right] \\
& = A_{\bb}\left[\tfrac{\conj{b_{i+1}} - q\conj{b_{i}}}{\conj{b_{i}} - \conj{b_{i+1}}} \right] \\
& =  A_{\bb}\left( \tfrac{b_{i} - qb_{i+1}}{b_{i+1} - b_{i}}\right)
\end{align*}
 And we see that, if the form $\langle \cdot, \cdot \rangle$ is to be $\AHA_{q}(n)$-invariant, we must have
\begin{equation}
\frac{A_{s_{i}\bb}}{A_{\bb}} = \frac{b_{i} - qb_{i+1}}{qb_{i} - qb_{i+1}} = \frac{q - (b_{i}/b_{i+1})}{1 - q(b_{i}/b_{i+1})} 
\end{equation}
and, if the form $\langle \cdot, \cdot \rangle$ is to be positive definite, we must have $A_{s_{i}\bb}/A_{\bb} > 0$ (note that manifestly $A_{s_{i}\bb}/A_{\bb} \in \R$, as needed). Set $x = b_{i}/b_{i+1}$. Now,
$$
\frac{q - x}{1 - qx} = \frac{(q-x)(1 - \conj{qx})}{|1 - qx|^2} = 2\frac{\Re(q) - \Re(x)}{|1 - qx|^2}
$$
so $A_{s_{i}\bb}/A_{\bb} > 0$ if and only if $\Re(b_{i}/b_{i+1}) < \Re(q)$, where $\Re(z)$ denotes the real part of   $z\in \mathbb C$.  

So, assume that every $\bb \in [\ba]$ satisfies that $\Re(b_{i+1}/b_{i}) < \Re(q)$ for every pair $b_{i}, b_{i+1}$ with $b_{i}/b_{i+1} \neq 1, q^{\pm1}$. Start with any $\bb \in [\ba]$. Any other $\bb' \in [\ba]$ is reachable from $\bb$ by a sequence of admissible transpositions. We declare $\langle w_{\bb}, w_{\bb}\rangle = 1$. By \cite[Proposition 4.5.4]{MR4054879}, there is an invariant Hermitian form on $M_{[\ba]}$ satisfying this, and we can see that we have:
$$
\langle w_{s_{i_{t}}\dots s_{i_{1}}\bb}, w_{s_{i_{t}}\dots s_{i_{1}}\bb}\rangle = \frac{q(s_{i_{t-1}}\dots s_{i_{1}}\bb)_{i_{t} + 1} - (s_{i_{t-1}}\dots s_{i_{1}}\bb)_{i_{t}}}{(s_{i_{t-1}}\dots s_{i_{1}}\bb)_{i_{t} + 1} - q(s_{i_{t-1}}\dots s_{i_{1}}\bb)_{i_{t}}}\langle w_{s_{i_{t-1}}\dots s_{i_{1}}\bb}, w_{s_{i_{t-1}}\dots s_{i_{1}}\bb}\rangle
$$
We have arrived at the following result. 

\begin{theorem}\label{thm:A}
An irreducible $\AHA_{q}(n)$-module $M$ is unitary if and only if there exists a calibrated weight $\ba = (a_1, \dots, a_n) \in \cal^{\aff}$ satisfying the following conditions.
\begin{enumerate}
\item $\ba \in (S^{1})^{n}$, i.e., $|a_{i}| = 1$ for every $i = 1, \dots, n$. 
\item For every $\bb \in [\ba]$ and every $i, i+1$, $\Re(b_{i}/b_{i+1}) \leq \Re(q)$.
\item $M = M_{[\ba]}$.
\end{enumerate}
\end{theorem}

The following easy lemma will be useful in the future.

\begin{lemma}\label{lemma:easy}
Let $\ba = (a_{1}, \dots, a_{n})  \in \cal^{\aff}$ be a calibrated weight such that $|a_{i}| = 1$ for every $i$ and $\Re(a_{i}/a_{j}) \leq \Re(q)$ whenever $a_{i} \neq a_{j}$. Then $M_{[\ba]}$ is unitary.
\end{lemma}
\begin{proof}
We need to check that the class $[\ba]$ satisfies (2) of Theorem \ref{thm:A}. Since every weight in $[\ba]$ can be obtained by applying transpositions to $\ba$ this is immediate.
\end{proof}

We now  focus on representations that factor through a fixed cyclotomic quotient of $\AHA_q(n)$.

\subsection{Unitary modules for the cyclotomic Hecke algebra} Now we seek to classify unitary representations of the cyclotomic Hecke algebra:
$$
\cyc_{q, Q_1, \dots, Q_{\ell}}(n) = \frac{\cH_{q}^{\aff}(n)}{(\prod_{i = 1}^{\ell}X_{1} - Q_{i})}.
$$
We may and will assume that all complex numbers $q, Q_1, \dots, Q_{\ell}$ live in the unit circle. Similarly to the case of the finite Hecke algebra $H_{q}(n)$, we can see that if $M$ is a calibrated representation of $\cyc_{q, Q_1, \dots, Q_{\ell}}(n)$ and $\ba$ is a weight of $M$, then for every $i = 1, \dots, n$ there exist $j \in 1, \dots, \ell$ and $m_{i} \in \Z$ such that $a_{i} = Q_{j}q^{m_{i}}$. This implies the following result, which is Theorem A from the introduction. 

\begin{lemma}\label{lemma:unitary=calibrated}
Assume that $q = \exp(2\pi\sqrt{-1}/e)$ for some $e \geq 2$ and that there exist $s_{1}, \dots, s_{\ell} \in \Z$ such that $Q_{i} = q^{s_{i}}$ for every $i = 1, \dots, \ell$. Then, a $\cyc_{q, Q_1, \dots, Q_{\ell}}(n)$-representation is unitary if and only if it is calibrated.
\end{lemma}
\begin{proof}
Under the assumptions of the lemma, if $\ba$ is a weight of a calibrated representation then every coordinate of $\ba$ is a power of $q$. Since $\Re(q^{i}) \leq \Re(q)$ unless $q^{i} = 1$, the result follows from Lemma \ref{lemma:easy}.
\end{proof}

The parameters $q, Q_1, \dots, Q_{\ell}$ of the form appearing in Lemma \ref{lemma:unitary=calibrated} are the most interesting ones from a representation-theoretic point of view. They are those for which the representation theory of the algebra $\cyc_{q, Q_1, \dots, Q_{\ell}}(n)$ cannot be broken-up in smaller pieces. More precisely, let us define an equivalence relation on the set of parameters $Q_1, \dots, Q_{\ell}$ by declaring $Q_{i} \sim Q_{j}$ if $Q_{i}/Q_{j}$ is an integral power of $q$. Let $\mathcal{E}_{1}, \dots, \mathcal{E}_{k}$ be the different equivalence classes. Thanks to work of Dipper and Mathas, see e.g. \cite[Theorem 13.30]{MR1911030} we have a category equivalence
\begin{equation}\label{eq:dipper-mathas dec}
\cyc_{q, Q_{1}, \dots, Q_{\ell}}(n)\operatorname{-mod} \cong \bigoplus_{n_1 + \dots + n_{k} = n} \bigotimes_{i = 1}^{k}\cyc_{q, \mathcal{E}_{i}}(n_{i})\operatorname{-mod}
\end{equation}
  where $\cyc_{q, \mathcal{E}_{i}} = \cyc_{q, Q_{a_{1}}, \dots, Q_{a_{j}}}$ if $\mathcal{E}_{i} = \{Q_{a_{1}}, \dots, Q_{a_{j}}\}$. Since we are interested in irreducible, unitary representations, it is enough to restrict to a single direct summand on the right-hand side of \eqref{eq:dipper-mathas dec}. 

It follows from the crystal-theoretic characterization of calibrated representations, see Theorem \ref{thm:nostuttering} below, that a representation $\bigotimes_{i = 1}^{k} M_{i}$ of $\bigotimes_{i = 1}^{k}\cyc_{q, \mathcal{E}_{i}}(n_{i})$ is calibrated if and only if each tensor factor $M_{i}$ is a calibrated representation of $\cyc_{q, \mathcal{E}_{i}}(n_{i})$. Similarly, if $\bigotimes_{i = 1}^{k} M_{i}$ is unitary, then $M_{i}$ is a unitary representation of $\cyc_{q, \mathcal{E}_{i}}(n_{i})$ for every $i$. The converse, however, is not true. For example, in the most generic case when each equivalence class $\mathcal{E}_{i}$ consists of a single element, a representation 
$$
\bigotimes_{i = 1}^{k}D(\lambda_{i}) \in \bigotimes_{i = 1}^{k}\cyc_{q, Q_{i}}(n_{i})\operatorname{-mod}
$$  is unitary if and only if $D(\lambda_{i})$ is a unitary representation of $\cyc_{q, Q_{i}}(n_{i}) \cong H_{q}(n_{i})$ and, for $i \neq j$, $\Re(Q_{i}Q_{j}^{-1}q^{a-b}) \leq \Re(q)$, where $a$ ranges over the contents of all boxes in the partition $\lambda_{i}$, and $b$ ranges over the contents of all boxes in the partition $\lambda_{j}$. 

\begin{remark}\label{rmk:pages}
More generally, in the notation of \cite{MR1988991}, the partition of $\{Q_1, \dots, Q_{\ell}\}$ into different equivalence classes induces a partition on a weight of a calibrated representation into \emph{pages}. Different pages do not interact, so a representation $\bigotimes M_{i}$ of $\cyc_{q, Q_1, \dots, Q_{\ell}}(n)$ is unitary if and only if each $M_{i}$ is a unitary for every $i$ and $\Re(a_i/a_j) \leq \Re(q)$, whenever $a_{i}, a_{j}$ are components of the weight $\ba$ corresponding to different pages.  
\end{remark}

Since checking the condition on Remark \ref{rmk:pages} quickly becomes unwieldy when there are many pages, we will restrict our attention to cyclotomic Hecke algebras of the form stated in Lemma \ref{lemma:unitary=calibrated}. We remark, however, that thanks to the decomposition \eqref{eq:dipper-mathas dec} many of the results we prove for unitary representations, including the construction of BGG resolutions, can be extended to the general setting. For a fixed parameter $q, Q_1, \dots, Q_\ell$ of the cyclotomic Hecke algebra we will find all the $\ell$-partitions labeling unitary representations. We define this set in Section \ref{sec:multicali set} and prove that it indeed labels the unitary representations in Section \ref{sec:multicali}.

\section{The set of multipartitions $\Cali^\bs(\ell)$}\label{sec:multicali set}

\subsection{The right border multiset of an $\ell$-partition with respect to an $\ell$-charge}\label{subsec:rightbord}

We define the {\em right border multiset} $\Bo^\bs(\bla)$ of $\bla$ with respect to the charge $\bs$ to be the collection of integers $\co^\bs(b)$ for each $b$ the last box of a row of $\bla$, with multiplicities. As in writing partitions, we will record right border multisets using exponential notation, i.e., if $\bla$ has $m$ rows whose last boxes $b$ satisfy $\co^\bs(b)=a$ then we will write $\Bo^\bs(\bla)=\{\ldots,a^m,\ldots\}$.

\begin{example}\label{exl:ellpartition} Let $\ell=3$, $\bla=((2^2),(2),(3,2))$, and $\bs=(0,1,4)$. 
 Then $\Bo^\bs(\bla)=\{6,4,2,1,0\}$ (see \cref{extrafig} for a visualisation). In this example $\Bo^\bs(\bla)$ is multiplicity-free. 

\end{example}

If $\ell=1$ so that $\bla$ is a single partition $\la$ and $s\in\mathbb{Z}$, then $\Bo^s(\la)$ is always multiplicity-free. However, for $\ell$-partitions with $\ell>1$ this need not be the case as the next example shows.

\begin{example}\label{exl:bordmult} Let $\ell=3$, $\bla=((2^2),(2^2,1),(3,2,1))$, and $\bs=(0,1,4)$. Then $\Bo^\bs(\bla)=\{6,4,2^2,1^2,0,\hbox{-}1\}$. 

\end{example}

\begin{figure}[ht!]
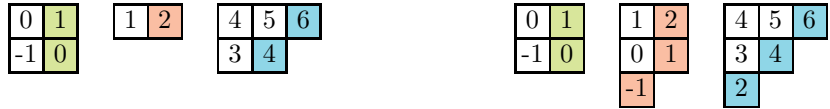

\begin{center}
\ytableausetup{mathmode,boxsize=1.25em}\scalefont{0.9}
\begin{ytableau}
0&*(SpringGreen!65)1\\
\hbox{-}1&*(SpringGreen!65)0\\
\end{ytableau} 
\quad
\begin{ytableau}
1&*(Melon!65)2\\
\end{ytableau} 
\quad
\begin{ytableau}
4&5&*(SkyBlue!65)6\\
3&*(SkyBlue!65)4\\
\end{ytableau} 
\qquad\quad\qquad\qquad\begin{ytableau}
0&*(SpringGreen!65)1\\
\hbox{-}1&*(SpringGreen!65)0\\
\end{ytableau} 
\quad
\begin{ytableau}
1&*(Melon!65)2\\
0&*(Melon!65)1\\
*(Melon!65)\hbox{-}1\\
\end{ytableau} 
\quad
\begin{ytableau}
4&5&*(SkyBlue!65)6\\
3&*(SkyBlue!65)4\\
*(SkyBlue!65)2\\
\end{ytableau} 
\end{center}
\caption{The $3$-partitions from Examples \ref{exl:ellpartition} and \ref{exl:bordmult} together with their contents. We have highlighted the boxes at the end of each row in the Young diagram of $\bla$.  }
\label{extrafig}
\end{figure}

We define the {\em reading word} of $\Bo^\bs(\bla)$ to be the sequence of integers given by listing the elements of $\Bo^{s_1}(\la^1)$ in increasing order, then the elements of $\Bo^{s_2}(\la^2)$ in increasing order, and so on. Thus the reading word is a preferred order of listing the elements of $\Bo^\bs(\bla)$. Let $w=(a_1,a_2,\ldots,a_h)$ be the reading word of $\Bo^\bs(\bla)$ (here, $h$ is the number of rows of $\bla$). We say that $w$ is {\em increasing} if $a_1<a_2<\ldots<a_h$.  In Example \ref{exl:ellpartition}, the reading word is $(0,1,2,4,6)$ and is increasing. In Example \ref{exl:bordmult}, the reading word is $(0,1,-1,1,2,2,4,6)$ and is not increasing.

\subsection{The set $\Cali^\bs(\ell)$}\label{subsec:calimultipart} 
We now define a set of multipartitions that, as we will see below in Theorem \ref{thm:cali=cali}, provides a combinatorial description of the calibrated irreducible representations of an appropriate cyclotomic Hecke algebra.  
\begin{definition}\label{def:cali} Fix $e\geq 2$. Let $\bs\in\mathbb{Z}^\ell$ be a cylindrical charge. Define $\Cali^{\bs}(\ell)$ to be the set of all $\ell$-partitions $\bla$ satisfying the following conditions:
\begin{enumerate}[leftmargin=*]
\item \begin{enumerate}
\item $\Bo^\bs(\bla)\subset[z,z+e-1]$ for some $z\in\mathbb{Z}$ (in which case we say that it has  period at most $e$);
\item The reading word of $\Bo^\bs(\bla)$ is increasing;
\end{enumerate}
\item $\bla$ is cylindrical.
\end{enumerate}
\end{definition}
\noindent It is immediate that $\bla\in\Cali^\bs(\ell)$ implies that $\bla$ is FLOTW, see Definition \ref{def:FLOTW}. 

\smallskip

In the case that $\bla$ satisfies Definition \ref{def:cali}(1), it is easy to check whether $\bla$ is cylindrical. Given $\bla=(\la^1,\la^2,\ldots,\la^\ell)$, if $\la^j\neq\emptyset$ then define $b^j_{\mathrm{min}}$ to be the box of smallest content in $\la^j$. That is, $b^j_{\mathrm{min}}$ is the leftmost box of the bottom row of $\la^j$. For each $j=1,\ldots,\ell$ such that $\la^j\neq\emptyset$, let $h_j$ be the number of nonzero rows of $\la^j$. For $\la^j\neq \emptyset$, we have $\co^\bs(b^j_{\mathrm{min}})=-h_j+1+s_j$.

\begin{lemma}\label{lemma:cylindrical}
Suppose $\bs\in\mathbb{Z}^\ell$ is cylindrical and $\bla$ satisfies Definition \ref{def:cali}(1). Then $\bla$ is cylindrical, and thus $\bla\in\Cali^\bs(\ell)$, if and only if the following two conditions hold:
\begin{enumerate}
\item For each $j=2,\ldots, \ell$, if $\la^j\neq \emptyset$ then we have $s_{j-1}<\mathsf{co}^\bs(b^j_{\mathrm{min}})$.
\item Let $j\in\{1,2,\ldots,\ell\}$ be minimal such that $\la^j\neq \emptyset$. Then $s_\ell<\co^\bs(b^j_{\mathrm{min}})+e$.
\end{enumerate}
 \end{lemma}

\begin{proof} First we show that any $\bla$ satisfying (1) and (2) is cylindrical. 
We have that (1) holds if and only if $-s_j>-\co^\bs(b^{j+1}_{\mathrm{min}})=h_{j+1}-1-s_{j+1}$, if and only if $s_{j+1}-s_j>h_{j+1}-1$. Thus (1) holds if and only if for all $k\geq 1$ we have $\la^{j+1}_{k+s_{j+1}-s_j}=0$, which is \ref{def:FLOTW}(2)(a). If $\la^1=\emptyset$ then it automatically holds that $\la_k^\ell\geq\la^1_{k+e+s_1-s_\ell}$ for all $k\geq 1$, i.e. \ref{def:FLOTW}(2)(b) holds. So assume $\la^1\neq \emptyset$. Then (2) holds if and only if $s_\ell<-h_1+1+s_1+e$ if and only if $h_1<e+s_1-s_\ell$, implying that $\la^1_{k+e+s_1-s_\ell}=0$ for all $k\geq 1$. Thus \ref{def:FLOTW}(2)(b) holds.

\smallskip

For the converse, suppose that $\bla $ is cylindrical. Suppose $j\in\{1,\ldots,\ell-1\}$ such that $\la^{j+1}\neq\emptyset$. To verify condition (1), we need to show that $h_{j+1}\leq s_{j+1}-s_j$. Suppose not, then there exists $k\geq 1$ such that $h_{j+1}=s_{j+1}-s_j+k$. Definition \ref{def:FLOTW}(2)(a) tells us that $$\la_{h_{j+1}}^{j+1}=\la_{s_{j+1}-s_j+k}\leq \la_k^j\leq \la_1^j.$$ Since $\Bo^\bs(\bla)$ is increasing, the charged content of the rightmost box of the bottom row of $\la^{j+1}$ is greater than the charged content of the rightmost box of the top row of $\la^j$. This yields the inequality: $$\la^j_1+s_j-1<\la^{j+1}_{h_{j+1}}+s_{j+1}-h_{j+1}=\la^{j+1}_{h_{j+1}}+s_{j+1}-(s_{j+1}-s_j+k)\leq \la^{j+1}_{h_{j+1}}+s_j-1$$
giving $\la_1^j<\la_{h_{j+1}}^{j+1}$, a contradiction.

\smallskip 

Next, we verify condition (2) by a similar argument. If $\la^1=\emptyset$ then using that $\bs$ is cylindrical together with condition (1) that was just proved, we have $s_\ell<s_1+e\leq s_{j-1}+e<\co^\bs(b^j_{\mathrm{min}})+e$, where $j\geq 2$ is minimal such that $\la^j\neq \emptyset$. So we can assume $j=1$ i.e. $\la^1\neq \emptyset$.
We want to show that $s_\ell<\co^\bs(b^1_{\mathrm{min}})$. We have $\co^\bs(b^1_{\mathrm{min}})=-h_1+1+s_1+e$, so we need to show that $h_1\leq s_1-s_\ell+e$. Suppose not, then we can write $h_1=s_1-s_\ell+e+k$ for some $k\geq 1$. By the assumption that $\bla$ is FLOTW, we have $\la^\ell_1\geq \la^1_{1+e+s_1-s_\ell}\geq \la^1_{h_1}>0$. Thus $\la^\ell\neq \emptyset$. On the other hand, by the assumption that $\bla$ satisfies Definition \ref{def:cali}(1), it holds that the largest element of $\Bo^\bs(\bla)$ is less than the smallest element of $\Bo^\bs(\bla)$ plus $e$. Thus
$\la^\ell_1+s_\ell-1<\la^1_{h_1}-h_1+s_1+e$. But $\la^1_{h_1}-h_1+s_1+e=\la^1_{h_1}-(s_1-s_\ell+e+k)+s_1+e=\la^1_{h_1}+s_\ell-k\leq \la^1_{h_1}+s_\ell-1$. So we get $\la^\ell_1<\la^1_{h_1}$, a contradiction with $\la^\ell_1\geq \la^1_{h_1}$ that followed above from $\bla$ being FLOTW. This concludes the proof.
\end{proof}

\begin{corollary}\label{admiss}
  If 
$\bla\in\Cali^\bs(\ell)$ then  $h(\bla)$ is $\bs$-admissible.  
\end{corollary}
\begin{proof} 
Conditions $(1)$ and $(2)$ of \ref{lemma:cylindrical} together with the condition that $\Bo^\bs(\bla)\subset[z,z+e-1]$ implies that   $h(\bla)$ is $\bs$-admissible.  
\end{proof}

\begin{example}\label{exl:cali}
Let $\ell=3$ and $\bla=((2^2),(2),(3,2))$ and $\bs=(0,1,4) $ as in Example \ref{exl:ellpartition}. Suppose $e\geq 7$. Then $\bs=(0,1,4)$ is FLOTW.
We have $\Bo^\bs(\bla)=\{0,1,2,4,6\}\subset[0,e-1]$ and the reading word is $(0,1,2,4,6)$, which is increasing. Thus Definition \ref{def:cali}(1) is satisfied. 
To check that $\bla$ is cylindrical we apply Lemma \ref{lemma:cylindrical}. 
We have $s_1=0<1=\co^\bs(b^2_{\mathrm{min}})$ and $s_2=1<3=\co^\bs(b^3_{\mathrm{min}})$, verifying Condition (1) of Lemma \ref{lemma:cylindrical}. Also, $s_3=4<6\leq -1+e=\co^\bs(b^1_{\mathrm{min}})+e$, verifying Condition (2) of Lemma \ref{lemma:cylindrical}. Therefore, $\bla$ is cylindrical, and $\bla\in\Cali^\bs(3)$.
\end{example}

\subsection{Multiplicity-free right border sets and semi-infinite Young diagrams}\label{subsec:multfreebord}

Let $\mathcal{I}\subset \mathbb{Z}$ be a finite subset of the integers. Write $\mathcal{I}=\{i_h<i_{h-1}<\ldots<i_2<i_1\}$. We define the {\em semi-infinite Young diagram with right border set $\mathcal{I}$} to be the set of boxes
$b=(x,y)$, $1\leq x\leq h$ and $y \leq i_{x}+x$, and we denote it by $\widetilde{Y}(\mathcal{I})$. We will think of $x$ as the row and $y$ as the column, and will draw the rows as descending as in our convention for ordinary Young diagrams. The {\em content} of a box $b=(x,y)\in\widetilde{Y}(\mathcal{I})$ is $\co^{\cI}(b):=y-x$.

\begin{example}\label{exl:sinfYdiagram}
Suppose $\mathcal{I}=\{0,1,2,4,6\} $, this determines   $\widetilde{Y}(\mathcal{I})$ as in the leftmost diagram in Figure \ref{lessroom1}.  
Truncating at the column whose topmost box has content $0$, we obtain the Young diagram of $\la=(7,6,5^3)$ with charge $0$  as depicted 
in the rightmost diagram in Figure \ref{lessroom1}.

\begin{figure}[ht!]
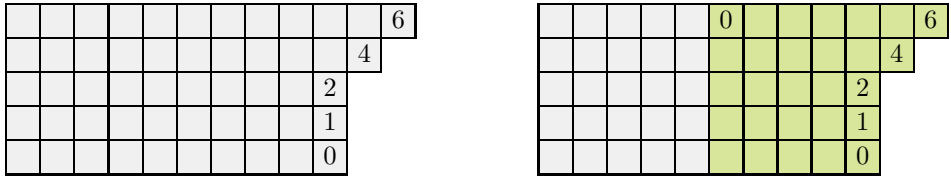
\begin{center}
\ytableausetup{mathmode,boxsize=1.25em}\scalefont{0.9} \begin{ytableau}
\none\ldots&*(LightGray)&*(LightGray)&*(LightGray)&*(LightGray)&*(LightGray)&*(LightGray)&*(LightGray)&*(LightGray)&*(LightGray)&*(LightGray)&*(LightGray)&*(LightGray)6\\
\none\ldots&*(LightGray)&*(LightGray)&*(LightGray)&*(LightGray)&*(LightGray)&*(LightGray)&*(LightGray)&*(LightGray)&*(LightGray)&*(LightGray)&*(LightGray)4\\
\none\ldots&*(LightGray)&*(LightGray)&*(LightGray)&*(LightGray)&*(LightGray)&*(LightGray)&*(LightGray)&*(LightGray)&*(LightGray)&*(LightGray)2\\
\none\ldots&*(LightGray)&*(LightGray)&*(LightGray)&*(LightGray)&*(LightGray)&*(LightGray)&*(LightGray)&*(LightGray)&*(LightGray)&*(LightGray)1\\
\none\ldots&*(LightGray)&*(LightGray)&*(LightGray)&*(LightGray)&*(LightGray)&*(LightGray)&*(LightGray)&*(LightGray)&*(LightGray)&*(LightGray)0\\
\end{ytableau}
\qquad
\quad
\begin{ytableau}
\none\ldots&*(LightGray)&*(LightGray)&*(LightGray)&*(LightGray)&*(LightGray)&*(SpringGreen!65)0&*(SpringGreen!65)&*(SpringGreen!65)&*(SpringGreen!65)&*(SpringGreen!65)&*(SpringGreen!65)&*(SpringGreen!65)6\\
\none\ldots&*(LightGray)&*(LightGray)&*(LightGray)&*(LightGray)&*(LightGray)&*(SpringGreen!65)&*(SpringGreen!65)&*(SpringGreen!65)&*(SpringGreen!65)&*(SpringGreen!65)&*(SpringGreen!65)4\\
\none\ldots&*(LightGray)&*(LightGray)&*(LightGray)&*(LightGray)&*(LightGray)&*(SpringGreen!65)&*(SpringGreen!65)&*(SpringGreen!65)&*(SpringGreen!65)&*(SpringGreen!65)2\\
\none\ldots&*(LightGray)&*(LightGray)&*(LightGray)&*(LightGray)&*(LightGray)&*(SpringGreen!65)&*(SpringGreen!65)&*(SpringGreen!65)&*(SpringGreen!65)&*(SpringGreen!65)1\\
\none\ldots&*(LightGray)&*(LightGray)&*(LightGray)&*(LightGray)&*(LightGray)&*(SpringGreen!65)&*(SpringGreen!65)&*(SpringGreen!65)&*(SpringGreen!65)&*(SpringGreen!65)0\\
\end{ytableau}
\end{center}

\caption{The diagram $\widetilde{Y}(\mathcal{I})$ as in Example \ref{exl:sinfYdiagram} 
and its 
truncation at the column whose topmost box has content $0$ to  obtain the Young diagram of $\la=(7,6,5^3)$.  }
\label{lessroom1}
\end{figure}

 \end{example}

Fix $\cI=\{i_1,\ldots,i_h\}\subset\mathbb{Z}$ and some $e\in\mathbb{N}$ such that $e>i_1-i_h$ and $e>h$. In Example \ref{exl:sinfYdiagram}, we may take any $e\geq 7$. For any $\ell\geq 1$,
we will construct a particular set of charged $\ell$-partitions $\bla$ with cylindrical charge $\bs$ and such that $\Bo^{\bs}(\bla)=\cI$. First, we choose $s_1\in\mathbb{Z}$ such that $s_1\leq i_h$. Then, we choose a box $b_1=(x_1,y_1)\in\widetilde{Y}(\cI)$ such that $\co^{\cI}(b)=s_1$. Thus $b_1$ lies on or to the left of the northwest-southeast diagonal ending at the bottom right box of the diagram $\widetilde{Y}(\cI)$. The box $b_1$ will be the top left corner of $\la^1$. We then take $\la^1$ to be the partition consisting of all boxes in $\widetilde{Y}(\cI)$ below and to the right of $b_1$. That is:
$$\la^1=\{b=(x,y)\in\widetilde{Y}(\cI)\mid x\geq x_1,\;y\geq y_1\}.$$
We then mark the boxes in $\widetilde{Y}(\cI)$ of content $\alpha$, where $\alpha=\co^{s_1}(b^1_{\mathrm{min}})$. Here, $b^1_{\mathrm{min}}$ is the bottom-leftmost box of $\la^1$, its box of smallest content. A choice of $s_1$ and $b_1$ is illustrated schematically in \cref{figA} below, with $i_h$ and $\alpha$ also indicated.

\begin{figure}[ht!]
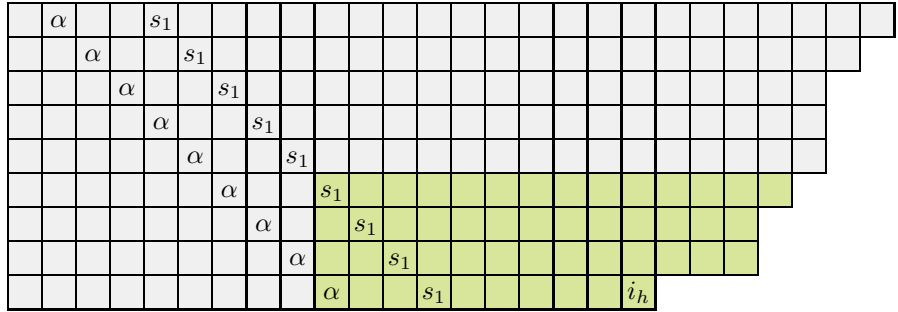
\begin{center}
\ytableausetup{mathmode,boxsize=1.25em}\scalefont{0.9} \begin{ytableau}
\none\ldots&*(LightGray)&*(LightGray)\alpha&*(LightGray)&*(LightGray)&*(LightGray)s_1&*(LightGray)&*(LightGray)&*(LightGray)&*(LightGray)&*(LightGray)&*(LightGray)&*(LightGray)&*(LightGray)&*(LightGray)&*(LightGray)&*(LightGray)&*(LightGray)&*(LightGray)&*(LightGray)&*(LightGray)&*(LightGray)&*(LightGray)&*(LightGray)&*(LightGray)&*(LightGray)&*(LightGray)\\
\none\ldots&*(LightGray)&*(LightGray)&*(LightGray)\alpha&*(LightGray)&*(LightGray)&*(LightGray)s_1&*(LightGray)&*(LightGray)&*(LightGray)&*(LightGray)&*(LightGray)&*(LightGray)&*(LightGray)&*(LightGray)&*(LightGray)&*(LightGray)&*(LightGray)&*(LightGray)&*(LightGray)&*(LightGray)&*(LightGray)&*(LightGray)&*(LightGray)&*(LightGray)&*(LightGray)\\
\none\ldots&*(LightGray)&*(LightGray)&*(LightGray)&*(LightGray)\alpha&*(LightGray)&*(LightGray)&*(LightGray)s_1&*(LightGray)&*(LightGray)&*(LightGray)&*(LightGray)&*(LightGray)&*(LightGray)&*(LightGray)&*(LightGray)&*(LightGray)&*(LightGray)&*(LightGray)&*(LightGray)&*(LightGray)&*(LightGray)&*(LightGray)&*(LightGray)&*(LightGray)\\
\none\ldots&*(LightGray)&*(LightGray)&*(LightGray)&*(LightGray)&*(LightGray)\alpha&*(LightGray)&*(LightGray)&*(LightGray)s_1&*(LightGray)&*(LightGray)&*(LightGray)&*(LightGray)&*(LightGray)&*(LightGray)&*(LightGray)&*(LightGray)&*(LightGray)&*(LightGray)&*(LightGray)&*(LightGray)&*(LightGray)&*(LightGray)&*(LightGray)&*(LightGray)\\
\none\ldots&*(LightGray)&*(LightGray)&*(LightGray)&*(LightGray)&*(LightGray)&*(LightGray)\alpha&*(LightGray)&*(LightGray)&*(LightGray)s_1&*(LightGray)&*(LightGray)&*(LightGray)&*(LightGray)&*(LightGray)&*(LightGray)&*(LightGray)&*(LightGray)&*(LightGray)&*(LightGray)&*(LightGray)&*(LightGray)&*(LightGray)&*(LightGray)&*(LightGray)\\
\none\ldots&*(LightGray)&*(LightGray)&*(LightGray)&*(LightGray)&*(LightGray)&*(LightGray)&*(LightGray)\alpha&*(LightGray)&*(LightGray)&*(SpringGreen!65)s_1&*(SpringGreen!65)&*(SpringGreen!65)&*(SpringGreen!65)&*(SpringGreen!65)&*(SpringGreen!65)&*(SpringGreen!65)&*(SpringGreen!65)&*(SpringGreen!65)&*(SpringGreen!65)&*(SpringGreen!65)&*(SpringGreen!65)&*(SpringGreen!65)&*(SpringGreen!65)\\
\none\ldots&*(LightGray)&*(LightGray)&*(LightGray)&*(LightGray)&*(LightGray)&*(LightGray)&*(LightGray)&*(LightGray)\alpha&*(LightGray)&*(SpringGreen!65)&*(SpringGreen!65)s_1&*(SpringGreen!65)&*(SpringGreen!65)&*(SpringGreen!65)&*(SpringGreen!65)&*(SpringGreen!65)&*(SpringGreen!65)&*(SpringGreen!65)&*(SpringGreen!65)&*(SpringGreen!65)&*(SpringGreen!65)&*(SpringGreen!65)\\
\none\ldots&*(LightGray)&*(LightGray)&*(LightGray)&*(LightGray)&*(LightGray)&*(LightGray)&*(LightGray)&*(LightGray)&*(LightGray)\alpha&*(SpringGreen!65)&*(SpringGreen!65)&*(SpringGreen!65)s_1&*(SpringGreen!65)&*(SpringGreen!65)&*(SpringGreen!65)&*(SpringGreen!65)&*(SpringGreen!65)&*(SpringGreen!65)&*(SpringGreen!65)&*(SpringGreen!65)&*(SpringGreen!65)&*(SpringGreen!65)\\
\none\ldots&*(LightGray)&*(LightGray)&*(LightGray)&*(LightGray)&*(LightGray)&*(LightGray)&*(LightGray)&*(LightGray)&*(LightGray)&*(SpringGreen!65)\alpha&*(SpringGreen!65)&*(SpringGreen!65)&*(SpringGreen!65)s_1&*(SpringGreen!65)&*(SpringGreen!65)&*(SpringGreen!65)&*(SpringGreen!65)&*(SpringGreen!65)&*(SpringGreen!65)i_h\\
\end{ytableau}
\end{center}
\caption{Illustration of the first step of producing a calibrated multipartition from a semi-infinite Young diagram}
\label{figA}
\end{figure}

If $b_1$ is in the top row of $\widetilde{Y}(\cI)$, then we stop here: we have constructed a charged $1$-partition with right border $\cI$ as above. Otherwise, we go to the next step. We choose some $s_2$ such that $s_1<s_2<\alpha+e$.  Observe that since $\widetilde{Y}(\cI)$ has less than $e$ rows, this is always possible. On the diagonal of boxes with content $s_2$, we pick a box $b_2=(x_2,y_2)$ such that $\co(b_2)=s_2$, $x_2<x_1$, and $y_2\geq y_1$. If $s_2=\alpha+e-1$ then we require that $x_2=1$, that is, that $b_2$ is in the top row. We then take the partition $\la^2$ to consist of those boxes of $\widetilde{Y}(\cI)$ lying below and right of $b_2$ and strictly above $\la^1$:
$$\la^2=\{b=(x,y)\in\widetilde{Y}(\cI)\mid x_2\leq x<x_1,\;y\geq y_2\}.$$

\smallskip

If $x_2\neq 1$, that is, if $b_2$ is not in the top row of the diagram, then we continue the process with Step 3 choosing $s_3$ and $\la^3$... At Step $i$ of this process, which occurs if $b_{i-1}$ is not in the top row of $\widetilde{Y}(\cI)$, we choose $s_i$ such that $s_{i-1}<s_i<\alpha+e$, and we choose $b_i=(x_i,y_i)\in\widetilde{Y}(\cI)$ satisfying $\co(b_i)=s_i$, $x_i<x_{i-1}$, and $y_i\geq y_{i-1}$. We require that $x_i=1$ if $s_i=\alpha+e-1$. We then define $\la^i$ as
$$\la^i=\{b=(x,y)\in\widetilde{Y}(\cI)\mid x_i\leq x<x_{i-1},\;y\geq y_i\}.$$
The process terminates after $\ell\leq h$ steps since $\widetilde{Y}(\cI)$ has $h$ rows and each step chooses a box $b_i$ in a row above the row containing $b_{i-1}$. We take $\bla=(\la^1,\la^2,\ldots,\la^\ell)$ and $\bs=(s_1,s_2,\ldots,s_\ell)$. Then $\bs$ satisfies $s_1<s_2<\ldots<s_\ell<\alpha+e \leq s_1+e$ so is cylindric, and each $\la^j$ is non-empty, $1\leq j\leq \ell$.

The result of constructing $({\color{SpringGreen}\la^1},{\color{Melon}\la^2},{\color{Thistle}\la^3},\ldots,{\color{SkyBlue}\la^\ell}),({\color{SpringGreen}s_1},{\color{Melon}s_2},{\color{Thistle}s_3},\ldots,{\color{SkyBlue}s_\ell})$ for $\ell=4$ from the semi-infinite Young diagram of some $\cI=\{i_1,\ldots,i_h\}\subset \mathbb{Z}$ is depicted schematically in \cref{another1} below. The partitions must be stacked in such a way that they make a staircase shape. 

\begin{figure}[ht!]
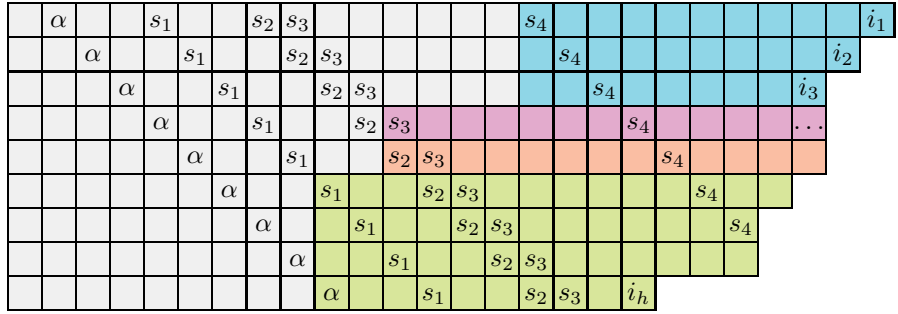

\begin{center}\ytableausetup{mathmode,boxsize=1.25em}\scalefont{0.9}
\begin{ytableau}
\none\ldots&*(LightGray)&*(LightGray)\alpha&*(LightGray)&*(LightGray)&*(LightGray)s_1&*(LightGray)&*(LightGray)&*(LightGray)s_2&*(LightGray)s_3&*(LightGray)&*(LightGray)&*(LightGray)&*(LightGray)&*(LightGray)&*(LightGray)&*(SkyBlue!65)s_4&*(SkyBlue!65)&*(SkyBlue!65)&*(SkyBlue!65)&*(SkyBlue!65)&*(SkyBlue!65)&*(SkyBlue!65)&*(SkyBlue!65)&*(SkyBlue!65)&*(SkyBlue!65)&*(SkyBlue!65)i_1\\
\none\ldots&*(LightGray)&*(LightGray)&*(LightGray)\alpha&*(LightGray)&*(LightGray)&*(LightGray)s_1&*(LightGray)&*(LightGray)&*(LightGray)s_2&*(LightGray)s_3&*(LightGray)&*(LightGray)&*(LightGray)&*(LightGray)&*(LightGray)&*(SkyBlue!65)&*(SkyBlue!65)s_4&*(SkyBlue!65)&*(SkyBlue!65)&*(SkyBlue!65)&*(SkyBlue!65)&*(SkyBlue!65)&*(SkyBlue!65)&*(SkyBlue!65)&*(SkyBlue!65)i_2\\
\none\ldots&*(LightGray)&*(LightGray)&*(LightGray)&*(LightGray)\alpha&*(LightGray)&*(LightGray)&*(LightGray)s_1&*(LightGray)&*(LightGray)&*(LightGray)s_2&*(LightGray)s_3&*(LightGray)&*(LightGray)&*(LightGray)&*(LightGray)&*(SkyBlue!65)&*(SkyBlue!65)&*(SkyBlue!65)s_4&*(SkyBlue!65)&*(SkyBlue!65)&*(SkyBlue!65)&*(SkyBlue!65)&*(SkyBlue!65)&*(SkyBlue!65)i_3\\
\none\ldots&*(LightGray)&*(LightGray)&*(LightGray)&*(LightGray)&*(LightGray)\alpha&*(LightGray)&*(LightGray)&*(LightGray)s_1&*(LightGray)&*(LightGray)&*(LightGray)s_2&*(Thistle!65)s_3&*(Thistle!65)&*(Thistle!65)&*(Thistle!65)&*(Thistle!65)&*(Thistle!65)&*(Thistle!65)&*(Thistle!65)s_4&*(Thistle!65)&*(Thistle!65)&*(Thistle!65)&*(Thistle!65)&*(Thistle!65)\ldots\\
\none\ldots&*(LightGray)&*(LightGray)&*(LightGray)&*(LightGray)&*(LightGray)&*(LightGray)\alpha&*(LightGray)&*(LightGray)&*(LightGray)s_1&*(LightGray)&*(LightGray)&*(Melon!65)s_2&*(Melon!65)s_3&*(Melon!65)&*(Melon!65)&*(Melon!65)&*(Melon!65)&*(Melon!65)&*(Melon!65)&*(Melon!65)s_4&*(Melon!65)&*(Melon!65)&*(Melon!65)&*(Melon!65)\\
\none\ldots&*(LightGray)&*(LightGray)&*(LightGray)&*(LightGray)&*(LightGray)&*(LightGray)&*(LightGray)\alpha&*(LightGray)&*(LightGray)&*(SpringGreen!65)s_1&*(SpringGreen!65)&*(SpringGreen!65)&*(SpringGreen!65)s_2&*(SpringGreen!65)s_3&*(SpringGreen!65)&*(SpringGreen!65)&*(SpringGreen!65)&*(SpringGreen!65)&*(SpringGreen!65)&*(SpringGreen!65)&*(SpringGreen!65)s_4&*(SpringGreen!65)&*(SpringGreen!65)\\
\none\ldots&*(LightGray)&*(LightGray)&*(LightGray)&*(LightGray)&*(LightGray)&*(LightGray)&*(LightGray)&*(LightGray)\alpha&*(LightGray)&*(SpringGreen!65)&*(SpringGreen!65)s_1&*(SpringGreen!65)&*(SpringGreen!65)&*(SpringGreen!65)s_2&*(SpringGreen!65)s_3&*(SpringGreen!65)&*(SpringGreen!65)&*(SpringGreen!65)&*(SpringGreen!65)&*(SpringGreen!65)&*(SpringGreen!65)&*(SpringGreen!65)s_4\\
\none\ldots&*(LightGray)&*(LightGray)&*(LightGray)&*(LightGray)&*(LightGray)&*(LightGray)&*(LightGray)&*(LightGray)&*(LightGray)\alpha&*(SpringGreen!65)&*(SpringGreen!65)&*(SpringGreen!65)s_1&*(SpringGreen!65)&*(SpringGreen!65)&*(SpringGreen!65)s_2&*(SpringGreen!65)s_3&*(SpringGreen!65)&*(SpringGreen!65)&*(SpringGreen!65)&*(SpringGreen!65)&*(SpringGreen!65)&*(SpringGreen!65)\\
\none\ldots&*(LightGray)&*(LightGray)&*(LightGray)&*(LightGray)&*(LightGray)&*(LightGray)&*(LightGray)&*(LightGray)&*(LightGray)&*(SpringGreen!65)\alpha&*(SpringGreen!65)&*(SpringGreen!65)&*(SpringGreen!65)s_1&*(SpringGreen!65)&*(SpringGreen!65)&*(SpringGreen!65)s_2&*(SpringGreen!65)s_3&*(SpringGreen!65)&*(SpringGreen!65)i_h\\
\end{ytableau}
\end{center}

\caption{Constructing a charged $4$-partition in $\Cali^\bs(4)$ with a given border set $\{i_1,i_2,i_3,\ldots,i_h\}$}.
\label{another1}
\end{figure}

\smallskip

For $\bla\in\Cali^\bs(\ell)$, removing all empty components $\lambda^j=\emptyset$ yields a charged $m$-partition obtained from $\widetilde{Y}(\cI)$ by the procedure of Section \ref{subsec:multfreebord} with $\cI=\Bo^\bs(\bla)$ and where $m=\ell-\#\{\la^j=\emptyset\}$. Conversely, for a given $\cI\subset\mathbb{Z}$ such that $\cI\subset[z,z+e-1]$ for some $z$ and $|\cI|<e$, we can start from  some $\bla\in\Cali^\bs(\ell)$ produced from $\widetilde{Y}(\cI)$, then insert empty components to obtain some $\bmu\in\Cali^{\bs'}(\ell')$, $\ell'>\ell$, with cylindrical charge $\bs'\in\mathbb{Z}^{\ell'}$ such that  $\{s_1,\ldots,s_\ell\}\subset\{s_1',\ldots,s_{\ell'}'\}$. The components of the charge for the empty components must be chosen so that the $\ell'$-partition remains cylindrical, i.e. by respecting the conditions in Lemma \ref{lemma:cylindrical}.

\begin{example}\label{exl:Young diagram} Let $\ell=3$, $\bs=(0,1,4)$, and $\bla=((2^2),(2),(3,2))$ as in  Example \ref{exl:ellpartition}. Suppose $e\geq 7$. Then $\bla\in\Cali^\bs(3)$, as was shown in Example \ref{exl:cali}.
We can construct $\bla$ with charge $\bs$ from $\widetilde{Y}(\{0,1,2,4,6\})$ as depicted in \cref{ppppsss}.  
 Then by inserting empty components following the conditions of Lemma \ref{lemma:cylindrical} on their charges, we can, for example, obtain a $6$-partition $\bmu\in\Cali^{\bs'}(6)$ having three non-empty components by taking $\bs'=(-2,0,1,2,2,4)$ and $\bmu=(\emptyset, (2^2),(2),\emptyset,\emptyset,(3,2))$.
 
 \begin{figure}[ht!]
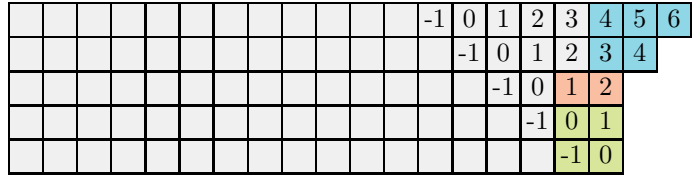
 
\begin{center}\ytableausetup{mathmode,boxsize=1.25em}\scalefont{0.9}
\begin{ytableau}
\none\ldots&*(LightGray)&*(LightGray)&*(LightGray)&*(LightGray)&*(LightGray)&*(LightGray)&*(LightGray)&*(LightGray)&*(LightGray)&*(LightGray)&*(LightGray)&*(LightGray)&*(LightGray)\hbox{-}1&*(LightGray)0&*(LightGray)1&*(LightGray)2&*(LightGray)3&*(SkyBlue!65)4&*(SkyBlue!65)5&*(SkyBlue!65)6\\
\none\ldots&*(LightGray)&*(LightGray)&*(LightGray)&*(LightGray)&*(LightGray)&*(LightGray)&*(LightGray)&*(LightGray)&*(LightGray)&*(LightGray)&*(LightGray)&*(LightGray)&*(LightGray)&*(LightGray)\hbox{-}1&*(LightGray)0&*(LightGray)1&*(LightGray)2&*(SkyBlue!65)3&*(SkyBlue!65)4\\
\none\ldots&*(LightGray)&*(LightGray)&*(LightGray)&*(LightGray)&*(LightGray)&*(LightGray)&*(LightGray)&*(LightGray)&*(LightGray)&*(LightGray)&*(LightGray)&*(LightGray)&*(LightGray)&*(LightGray)&*(LightGray)\hbox{-}1&*(LightGray)0&*(Melon!65)1&*(Melon!65)2\\
\none\ldots&*(LightGray)&*(LightGray)&*(LightGray)&*(LightGray)&*(LightGray)&*(LightGray)&*(LightGray)&*(LightGray)&*(LightGray)&*(LightGray)&*(LightGray)&*(LightGray)&*(LightGray)&*(LightGray)&*(LightGray)&*(LightGray)\hbox{-}1&*(SpringGreen!65)0&*(SpringGreen!65)1\\
\none\ldots&*(LightGray)&*(LightGray)&*(LightGray)&*(LightGray)&*(LightGray)&*(LightGray)&*(LightGray)&*(LightGray)&*(LightGray)&*(LightGray)&*(LightGray)&*(LightGray)&*(LightGray)&*(LightGray)&*(LightGray)&*(LightGray)&*(SpringGreen!65)\hbox{-}1&*(SpringGreen!65)0\\
\end{ytableau}
\end{center}
\caption{Realizing $\bla=((2^2),(2),(3,2))\in\Cali^{(0,1,4)}(3)$ inside the semi-infinite Young diagram associated to its border set $\{0,1,2,4,6\}$. }
\label{ppppsss}
\end{figure}

More generally, suppose that upon removing all empty components from a charged $\ell$-partition $\bla$ with some $\ell$-charge $\bs$ for some $\ell\geq3$, we get the $3$-partition $((2^2),(2),(3,2))$ with charge $(0,1,4)$ as pictured above. Then $\bla\in\Cali^\bs(\ell)$ if and only if $\bla=(\emptyset^{a_{-2}},(2^2),\emptyset^{a_0},(2),\emptyset^{a_1+a_2},(3,2),\emptyset^{a_4},\emptyset^{a_5})$  and $\bs=(-2^{a_{-2}},0^{1+a_0},1^{1+a_1},2^{a_2},4^{1+a_4},5^{a_5})$ for some $a_{-2},a_0,a_1,a_2,a_4,a_5\in\mathbb{Z}_{\geq 0}$, where $a_5$ and $a_{-2}$ cannot both be nonzero. Here we have used exponential notation. For example, taking $a_0=a_1=a_2=1,$ $a_5=2$, and $a_4=a_{-2}=0$ we get that $\bla=((2^2),\emptyset,(2),\emptyset,\emptyset,(3,2),\emptyset,\emptyset)\in\Cali^\bs(8)$ for $\bs=(0,0,1,1,2,4,5,5)$.
\end{example}

\smallskip


\subsection{Skew shapes} 
From now on, when we talk about the Young diagram of $\bla\in\Cali^\bs(\ell)$, we will mean the embedding of the Young diagram of $\bla$ in $\widetilde{Y}(\Bo^\bs(\bla))$ given by stacking the $\la^j$'s so that their right borders make up the right border of $\widetilde{Y}(\Bo^\bs(\bla))$. We now jettison the semi-infinite blank region to the left of the $\la^j$'s to obtain a diagram as in the leftmost depicted in \cref{ghghghghg}.

\begin{figure}[ht!]
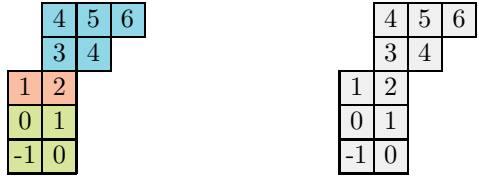

$$\ytableausetup{mathmode,boxsize=1.25em}\scalefont{0.9}
\begin{minipage}{3cm}\begin{ytableau}
\none &*(SkyBlue!65)4&*(SkyBlue!65)5&*(SkyBlue!65)6\\
\none &*(SkyBlue!65)3&*(SkyBlue!65)4\\
   *(Melon!65)1&*(Melon!65)2\\
  *(SpringGreen!65)0&*(SpringGreen!65)1\\
 *(SpringGreen!65)\hbox{-}1&*(SpringGreen!65)0\\
\end{ytableau}\end{minipage}
\qquad \qquad 
\begin{minipage}{2cm}
\begin{ytableau}
\none &                 *(LightGray)4&*(LightGray)5&*(LightGray)6\\
\none &                 *(LightGray)3&*(LightGray)4\\
            *(LightGray)1&*(LightGray)2\\
          *(LightGray)0&*(LightGray)1\\
         *(LightGray)\hbox{-}1&*(LightGray)0\\
\end{ytableau}
\end{minipage}\qquad
$$
\caption{The diagram for $\bla=((2^2),(2),(3,2))$ for $\bs = (0,1,4)$ and the skew-shape obtained by ``forgetting components".  }
\label{ghghghghg}
\end{figure}

\begin{convention}\label{conv:Young diagram}
\noindent We will write $\mathsf{Y}^\bs(\bla)$ for the Young diagram of $\bla\in\mathsf{Cali}^\bs(\ell)$ stacked as in the leftmost diagram in \cref{ghghghghg}.
\end{convention}

 Suppose  $\bla\in\mathsf{Cali}^\bs(\ell)$. 
If we ``forget components" in its Young diagram $\mathsf{Y}^\bs(\bla)$, we get a skew partition. In the picture of the Young diagram of the charged $3$-partition above, simply forget the different colors. In Example \ref{exl:Young diagram}, the skew partition associated to $\mathsf{Y}^\bs(\bla)$ has Young diagram given by the region shown  in the rightmost diagram in \cref{ghghghghg} (remembering the  contents of the boxes).  

Suppose that every component $\lambda^j$ of $\bla$ is nonempty. Then we may interpret Definition \ref{def:cali} as saying that when we express $\mathsf{Y}^\bs(\bla)$ as $\mu\setminus\nu$ for some partitions $\mu$ and $\nu$, then $\mu$ and $\nu$ both belong to $\Cali(1)$. By \cite{MR2266877}, $\mu\in\Cali(1)$ if and only if the irreducible representation $S_\mu$ of the finite Hecke algebra $H_q(S_n)$ for $q=\exp(2\pi i/e)$ is calibrated. By \cite{MR1988991}, the calibrated representations of the affine Hecke algebra $\widetilde{H}_q(S_n)$ for {\em generic} values of $q$ are labeled by skew Young diagrams $\mu\setminus \nu$. The condition for an $\ell$-partition to be in $\bla\in\Cali^\bs(\ell)$ thus arises as the intersection of two combinatorial conditions: on the one hand, the skew shape condition identifying calibrated representations of the affine type $A$ Hecke algebra for $q$ not a root of unity; and on the other hand, the condition identifying calibrated representations of the finite type $A$ Hecke algebra for $q=\exp(2\pi i/e)$, applied to the left and right borders of the skew shape. Now, if we further allow $\bla$ to contain empty components $\la^j=\emptyset $, then we just need to fine-tune this description by the conditions on the charges of these empty components given by Definition \ref{def:FLOTW} and Lemma \ref{lemma:cylindrical}. These extra conditions on where the empty components may occur and with what charge make sure that the resulting charge $\bs$ and $\ell$-partition $\bla$ remain cylindrical.

\begin{rmk}
We remark that a single skew diagram  can often be broken into 
many different charged multipartitions in an intuitive fashion.
  This is illustrated in \cref{allthecuts}.   
\end{rmk}

\begin{figure}[ht!]
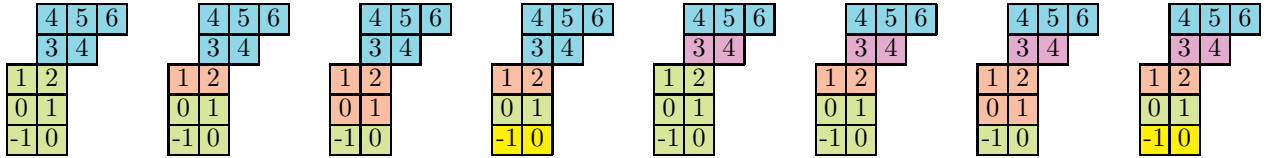

$$\ytableausetup{mathmode,boxsize=1.1em}\scalefont{0.9}
\begin{minipage}{1.8cm}
 \begin{ytableau}
\none &*(SkyBlue!65)4&*(SkyBlue!65)5&*(SkyBlue!65)6\\
\none &*(SkyBlue!65)3&*(SkyBlue!65)4\\
   *(SpringGreen!65)1&*(SpringGreen!65)2\\
  *(SpringGreen!65)0&*(SpringGreen!65)1\\
 *(SpringGreen!65)\hbox{-}1&*(SpringGreen!65)0\\
\end{ytableau}\end{minipage}
\quad 
 \begin{minipage}{1.8cm}\begin{ytableau}
\none &*(SkyBlue!65)4&*(SkyBlue!65)5&*(SkyBlue!65)6\\
\none &*(SkyBlue!65)3&*(SkyBlue!65)4\\
   *(Melon!65)1&*(Melon!65)2\\
  *(SpringGreen!65)0&*(SpringGreen!65)1\\
 *(SpringGreen!65)\hbox{-}1&*(SpringGreen!65)0\\
\end{ytableau}\end{minipage}
\quad
 \begin{minipage}{1.8cm}\begin{ytableau}
\none &*(SkyBlue!65)4&*(SkyBlue!65)5&*(SkyBlue!65)6\\
\none &*(SkyBlue!65)3&*(SkyBlue!65)4\\
   *(Melon!65)1&*(Melon!65)2\\
  *(Melon!65)0&*(Melon!65)1\\
 *(SpringGreen!65)\hbox{-}1&*(SpringGreen!65)0\\
\end{ytableau}\end{minipage}
\quad
\begin{minipage}{1.8cm}\begin{ytableau}
\none &*(SkyBlue!65)4&*(SkyBlue!65)5&*(SkyBlue!65)6\\
\none &*(SkyBlue!65)3&*(SkyBlue!65)4\\
  *(Melon!65)1&*(Melon!65)2\\
   *(SpringGreen!65)0&*(SpringGreen!65)1\\
 *(yellow)\hbox{-}1&*(yellow)0\\
\end{ytableau}\end{minipage}
\quad
\begin{minipage}{1.8cm}
 \begin{ytableau}
\none &*(SkyBlue!65)4&*(SkyBlue!65)5&*(SkyBlue!65)6\\
\none &*(Thistle!65)3&*(Thistle!65)4\\
   *(SpringGreen!65)1&*(SpringGreen!65)2\\
  *(SpringGreen!65)0&*(SpringGreen!65)1\\
 *(SpringGreen!65)\hbox{-}1&*(SpringGreen!65)0\\
\end{ytableau}\end{minipage}
\quad 
 \begin{minipage}{1.8cm}\begin{ytableau}
\none &*(SkyBlue!65)4&*(SkyBlue!65)5&*(SkyBlue!65)6\\
\none &*(Thistle!65)3&*(Thistle!65)4\\
   *(Melon!65)1&*(Melon!65)2\\
  *(SpringGreen!65)0&*(SpringGreen!65)1\\
 *(SpringGreen!65)\hbox{-}1&*(SpringGreen!65)0\\
\end{ytableau}\end{minipage}
\quad
 \begin{minipage}{1.8cm}\begin{ytableau}
\none &*(SkyBlue!65)4&*(SkyBlue!65)5&*(SkyBlue!65)6\\
\none &*(Thistle!65)3&*(Thistle!65)4\\
   *(Melon!65)1&*(Melon!65)2\\
  *(Melon!65)0&*(Melon!65)1\\
 *(SpringGreen!65)\hbox{-}1&*(SpringGreen!65)0\\
\end{ytableau}\end{minipage}
\quad
\begin{minipage}{1.8cm}\begin{ytableau}
\none &*(SkyBlue!65)4&*(SkyBlue!65)5&*(SkyBlue!65)6\\
\none &*(Thistle!65)3&*(Thistle!65)4\\
  *(Melon!65)1&*(Melon!65)2\\
   *(SpringGreen!65)0&*(SpringGreen!65)1\\
 *(yellow)\hbox{-}1&*(yellow)0\\
\end{ytableau}\end{minipage}
$$
 \caption{The 8 possible ways of breaking a fixed 
  skew shape into charged multipartitions with cylindric charge.
The first two  are
  $\bla=( (2^3),(3,2))$ for $\bs = (1,4)$ and 
  $\bla=( (2^2),(2), (3,2))$ for $\bs = (0,1,4)$.  The rest are left for the reader. 
  Notice that the multipartitions all have $2\leq \ell \leq 5$ components.   }
\label{allthecuts}
\end{figure}

To finish this section, we relate our constructions to the crystal operators $\tilde{e}_{i}$ of Section \ref{sec:crystal}.

\begin{lemma}\label{lemma:slecaliweak} Fix $e\geq 2$ and a cylindrical charge $\bs\in\mathbb{Z}^\ell$. The $\sle$-crystal operators $\tilde{e}_i$, $i\in\mathbb{Z}/e\mathbb{Z}$, preserve the set of $\ell$-partitions satisfying Definition \ref{def:cali}(1).
\end{lemma}
\begin{proof}
Let $\bla$ be an $\ell$-partition and $\bs\in\mathbb{Z}^\ell$ such that $\bla$ satisfies Definition \ref{def:cali}(1).
Then $\Bo^\bs(\bla)$ contains at most one element of residue $i$ for each $i\in\mathbb{Z}/e\mathbb{Z}$. The set $\{\co^\bs(b)\mid b\hbox{ is a removable box of }\bla\}$ is a subset of $\Bo^\bs(\bla)$. Thus $\bla$ has at most one removable $i$-box for each $i\in\mathbb{Z}/e\mathbb{Z}$. Suppose $i\in\mathbb{Z}/e\mathbb{Z}$ such that $\tilde{e}_i(\bla)\neq 0$. Then $\bla$ has exactly one removable $i$-box, call it $b$. We have $\co^\bs(b)\in\Bo^\bs(\bla)$. 

We claim that $\co^\bs(b)-1\notin\Bo^\bs(\bla)$. Observe that $\Bo^\bs(\bla)=\bigcup\limits_{j=1}^\ell\Bo^{s_j}(\la^j)$. Let us check that $\co^\bs(b)-1\notin\Bo^{s_j}(\la^j)$ for each $j=1,\ldots,\ell$. Let $j(b)\in\{1,\ldots,\ell\}$ be the integer such that $b\in\la^{j(b)}$, i.e. $j(b)$ is the component  of $\bla$ containing the box $b$. Since $b$ is a removable box of $\la^{j(b)}$, $\co^\bs(b)-1\notin\Bo^{s_{j(b)}}(\la^j)$. If $j(b)<j<\ell$, then $x\in\Bo^{s_j}(\la^j)$ implies $x>\co^\bs(b)$ by Definition \ref{def:cali}(1)(b), so $\co^\bs(b)-1\notin\Bo^{s_j}(\la^j)$ for all $j>j(b)$. Suppose there exists $j<j(b)$ such that $\co^{\bs}(b)-1\in\Bo^{s_j}(\la^j)$. Then $\co^\bs(b)-1$ must be the maximal element of $\Bo^{s_j}(\la^j)$ and $\co^\bs(b)$ the minimal element of $\Bo^{s_{j(b)}}(\la^{j(b)})$ ( and for any $j<k<j(b)$ we must have $\la^k=\emptyset$) in order that Definition \ref{def:cali}(1)(b) be satisfied. Thus $\co^\bs(b)-1$ is the charged content of the last box in the top row of $\la^j$. The top row of any partition always has an addable box. Therefore $\la^j$ has an addable box of charged content $\co^\bs(b)\cong i \mod e$. Since $b$ is the only removable $i$-box in $\bla$ and $j(b)>j$, the $i$-word of $\bla$ contains the subword $-+$. It follows that $b$ is not good removable, contradicting the assumption that $\tilde{e}_i(\bla)\neq 0$. 

Since $\co^\bs(b)-1\notin\Bo^\bs(\bla)$, it follows that the reading word of $\Bo^\bs(\tilde{e}_i(\bla))=\left(\Bo^\bs(\bla)\setminus \{\co^\bs(b)\}\right)\cup\{\co^\bs(b)-1\}$ is increasing. This verifies Definition \ref{def:cali}(1)(b) for $\tilde{e}_i(\bla)$. Definition \ref{def:cali}(1)(a) holds automatically for $\tilde{e}_i(\bla)$ unless $\co^\bs(b)+e-1\in\Bo^\bs(\bla)$. But if $\co^\bs(b)+e-1\in\Bo^\bs(\bla)$ then it follows from Definition \ref{def:cali}(1)(a),(b) that $\co^\bs(b)+e-1$ is the charged content of the last box in the top row of $\la^j$ for $j\geq j(b)$ such that $\la^k=\emptyset $ for all $j<k\leq \ell$. Then $\la^j$ has an addable box of content $\co^\bs(b)+e\cong i \mod e$, and so (as in the previous paragraph) $b$ is not a good removable $i$-box, contradicting the assumption that $\tilde{e}_i(\bla)\neq0$. We conclude that Definition \ref{def:cali}(1) holds for $\tilde{e}_i(\bla)$. 
\end{proof}

\section{Calibrated representations of cyclotomic Hecke algebras at roots of unity}\label{sec:multicali}


\subsection{Irreducible representations of cyclotomic Hecke algebras at roots of unity} Recall that the {\em cyclotomic Hecke algebra} (or {\em Ariki-Koike algebra}) is the following quotient of the affine Hecke algebra:
$$
\cyc_{q, Q_1, \dots, Q_{\ell}}(n) = \frac{\cH_{q}^{\aff}(n)}{(\prod_{i = 1}^{\ell}X_{1} - Q_{i})}.
$$
Fix $e\geq 2$, $\ell\geq 1$, and $\bs\in\mathbb{Z}^\ell$. Set $q=\exp(2\pi i/e)$ and $Q_i=q^{s_i}$. For such a choice of parameters (sometimes called ``integral parameters" in the literature), we will denote $\cyc_{q, Q_1, \dots, Q_{\ell}}(n)$ by $\cyc_{e,\bs}(n)$.

\begin{theorem}\label{thm:maxdepth} Fix $e\geq 2$ and $\bs\in\mathbb{Z}^\ell$. The irreducible representations of $\cyc_{e,\bs}(n) $ are labeled by the $\ell$-partitions of the form $\bla=\tilde{f}_{i_n}\tilde{f}_{i_{n-1}}\ldots \tilde{f}_{i_2}\tilde{f}_{i_1}\varnothing$ for some $i_1,\ldots i_n\in\mathbb{Z}/e\mathbb{Z}$.
\end{theorem}

\noindent  
In the case that $\bs$ is a cylindrical charge (see Definition \ref{def:cylindricalcharge}), Foda, Leclerc, Okado, Thibon, and Welsh gave a closed-form description of these $\ell$-partitions. Recall the definition of the set of FLOTW $\ell$-partitions (Definition \ref{def:FLOTW}).

\begin{theorem}\cite{FLOTW99}\label{thm:FLOTW}
Let $e\geq 2$ and let $\bs\in\mathbb{Z}^\ell$ be a cylindrical charge. Then the irreducible representations of $\cyc_{e,\bs}(n)$ are labeled by the FLOTW $\ell$-partitions of size $n$.
\end{theorem}

The goal of this section is to give the analogous closed-form description of the $\ell$-partitions labeling the irreducible calibrated representations of $\cyc_{e,\bs}(n) $ by the set of $\ell$-partitions $\mathsf{Cali}^\bs(\ell)$ (see Definition \ref{def:cali}). First, we will need the following lemma, which uses Theorem \ref{thm:FLOTW}.

\begin{lemma}\label{lemma:slecali} The $\sle$-crystal operators $\tilde{e}_i$ that remove boxes, $i\in\mathbb{Z}/e\mathbb{Z}$, preserve the set $\Cali^\bs(\ell)$.
\end{lemma}
\begin{proof}
Let $\bla\in\Cali^\bs(\ell)$. By Lemma \ref{lemma:slecaliweak}, it holds that $\tilde{e}_i(\bla)$ satisfies  Definition \ref{def:cali}(1) for all $i\in\mathbb{Z}/e\mathbb{Z}$ such that $\tilde{e}_i(\bla)\neq 0$. 
Any $\bla\in\Cali^\bs(\ell)$ is FLOTW. By Theorems \ref{thm:FLOTW} and \ref{thm:maxdepth}, then $\tilde{e}_i(\bla)$ is also FLOTW for all $i\in\mathbb{Z}/e\mathbb{Z}$ such that $\tilde{e}_i(\bla)\neq 0$. In particular, $\tilde{e}_i(\bla)$ is cylindrical, and Definition \ref{def:cali}(2) holds for $\tilde{e}_i(\bla)$. Therefore $\tilde{e}_i(\bla)\in\Cali^\bs(\ell)$.
\end{proof}



The analogue of Theorem \ref{thm:maxdepth} for the irreducible calibrated representations of $\cyc_{e,\bs}(n) $ identifies the $\ell$-partitions labeling them in terms of certain paths in the $\sle$-crystal.

\begin{theorem}\cite{grojnowski1999affine}\label{thm:nostuttering}
Let $e\geq 2$ and let $\bs\in\mathbb{Z}^\ell$. The irreducible calibrated representations of $\cyc_{e,\bs}(n) $ are labeled by the $\ell$-partitions of the form $\bla=\tilde{f}_{i_n}\tilde{f}_{i_{n-1}}\ldots\tilde{f}_{i_2}\tilde{f}_{i_1}\varnothing$ such that $i\neq i+1$ for all $i=i_1,\ldots, i_{n-1}$ in any such expression for $\bla$.

 \end{theorem}
That is, we consider all possible ways to build up a FLOTW $\ell$-partition $\bla$ from the empty $\ell$-partition $\varnothing$ by adding one box at a time such that at each step, the box we add is the good addable box for its residue. The theorem says that if $\bla$ labels an irreducible calibrated representation, then in all such sequences building up $\bla$ one box at a time, we never add an $i$-box immediately followed by another $i$-box.

We now arrive at the main result of this section, which is the analog of Theorem \ref{thm:FLOTW} for calibrated representations. 
\begin{theorem}\label{thm:cali=cali}
Let $e\geq 2$ and let $\bs\in\mathbb{Z}^\ell$ be a cylindrical charge. Then the irreducible calibrated representations of $\cyc_{e,\bs}(n)$ are labeled by the $\ell$-partitions in $\Cali^\bs(\ell)$ of size $n$.
\end{theorem}
\begin{proof}
We will prove the statement by induction on $n$.
%
%
The base case is $n=1$. By Theorem \ref{thm:nostuttering}, if $|\bla|=1$ then $\bla$ labels an irreducible calibrated representation if and only if $\bla$ is FLOTW. By Definition \ref{def:FLOTW} and Lemma \ref{lemma:cylindrical}, if $|\bla|=1$ then $\bla\in\Cali^\bs(\ell)$ if and only if $\bla$ is FLOTW. 


Now suppose by induction that for all FLOTW $\ell$-partitions $\bla$ of $n$, it holds that $\bla$ labels an irreducible calibrated representation of $\cyc_{e,\bs}(n)$ if and only if $\bla\in\Cali^\bs(\ell)$.

First, we will show that any $\ell$-partition of size $n+1$ in $\Cali^\bs(\ell)$ labels an irreducible calibrated representation of $\cyc_{e,\bs}(n+1)$. 
Suppose that $\bmu\in\Cali^\bs(\ell)$ and $|\bmu|=n+1$. Then $\bmu$ is FLOTW, so by Theorems \ref{thm:maxdepth} and \ref{thm:FLOTW}, $\bmu=\tilde{f}_{i_{n+1}}\tilde{f}_{i_n}\ldots\tilde{f}_{i_2}\tilde{f}_{i_1}\varnothing$ for some $i_1,\ldots,i_{n+1}\in\mathbb{Z}/e\mathbb{Z}$. We must show that every such expression satisfies $i_k\neq i_{k+1}$ for each $k=1,\ldots, n$. Let $i\in\mathbb{Z}/e\mathbb{Z}$ such that $\tilde{e}_i(\bmu)\neq 0$. By Lemma \ref{lemma:slecali}, $\tilde{e}_i\bmu\in\Cali^\bs(\ell)$. By induction, every expression $\tilde{e}_i\bmu=\tilde{f}_{i_{n}}\ldots\tilde{f}_{i_2}\tilde{f}_{i_1}\varnothing$ satisfies $i_k\neq i_{k+1}$ for all $k=1,\ldots,n-1$. It thus suffices to check that $i_n\neq i$.
The charged contents of removable boxes of $\bmu$ belong to $\Bo^\bs(\bmu)$. By Definition \ref{def:cali}, $\Bo^\bs(\bmu)$ contains at most one element of residue $i$ mod $e$, so $\bmu$ has at most one removable $i$-box. It follows that $\tilde{e}_i\bmu$ has no removable $i$-box. Therefore $\tilde{e}_i\tilde{e}_i\bmu=0$, implying $i_n\neq i$. We conclude using Theorem \ref{thm:nostuttering} that $\bmu$ labels an irreducible calibrated representation of $\cyc_{e,\bs}(n+1)$.

The remainder of the proof is dedicated to showing that if $\bla$ is a FLOTW $\ell$-partition of size $n+1$ that labels an irreducible calibrated representation of $\cyc_{e,\bs}(n+1)$, then $\bla\in\Cali^\bs(\ell)$.
  Suppose that $\bmu$ is FLOTW, $|\bmu|=n+1$, and $\bmu$ labels an irreducible calibrated representation. By Theorem \ref{thm:nostuttering}, for every expression $\bmu=\tilde{f}_{i_{n+1}}\tilde{f}_{i_n}\ldots\tilde{f}_{i_2}\tilde{f}_{i_1}\varnothing$ it holds that $i_{k+1}\neq i_k$ for all $k=1,\ldots,n$. We need to show that  $\bmu$ satisfies Definition \ref{def:cali}(1). By induction, $\bmu=\tilde{f}_i\bla$ for some $\bla\in\Cali^\bs(\ell)$ and some $i\in\mathbb{Z}/e\mathbb{Z}$. Either $\tilde{f}_i$ adds a box to a non-zero row of $\bla$, or $\tilde{f}_i$ creates a new row.  In the case that $\tilde{f}_i$ adds a box to an already existing row, the result is forced by the induction hypothesis that $\bla\in\Cali^\bs(\ell)$ using arguments similar to those in the proof of Lemma \ref{lemma:slecaliweak}. This is straightforward to check. The work consists in dealing with the case that $\tilde{f}_i$ adds the good $i$-box in a new row of $\bla$. 
Thus we suppose from now on that $\tilde{f}_i$ adds a new row to $\bla$, i.e. that $|\Bo^\bs(\bmu)|=|\Bo^\bs(\bla)|+1$. We will show that if Definition \ref{def:cali}(1)(a) or Definition \ref{def:cali}(1)(b) fails for $\tilde{f}_i\bla=\bmu$, then there exists 
an expression $\tilde{f}_{i_{n+1}}\tilde{f}_{i_n}\ldots\tilde{f}_{i_2}\tilde{f}_{i_1}\varnothing=\bmu$ with $i_k=i=i_{k+1}$ for some $k\in\{1,\ldots,n\}$.

First, suppose Definition \ref{def:cali}(1)(a) fails, so suppose $\Bo^\bs(\bmu)\not\subset[z,z+e-1]$ for all $z\in\mathbb{Z}$. Let $b$ be the $i$-box added by $\tilde{f}_i$ to $\bla$, i.e. $\bmu=\tilde{f}_i\bla\cup\{b\}$, and let $j(b)\in\{1,\ldots,\ell\}$ be such that $b$ is a box in $\mu^{j(b)}$. First, we show that $\co^\bs(b)$ must be the smallest element of $\Bo^\bs(\bmu)$. Suppose that $\co^\bs(b)$ is neither the smallest nor the largest element of $\Bo^\bs(\bmu)$. Since $\Bo^\bs(\bla)\subset [z,z+e-1]$ for some $z\in\mathbb{Z}$, this implies that $\Bo^\bs(\bmu)=[z,z+e-1]$ for some $z\in\mathbb{Z}$, and thus $\co^\bs(b)-1\in\Bo^\bs(\bla)$. Then $\co^\bs(b)-1$ is the content of the box of largest content in $\mu^k$ for some $k<j(b)$. But then $\mu^k$ has an addable box of content $\co^\bs(b)$. It follows that the good addable box of $\bla$ is then in $\mu^k$, not in $\mu^{j(b)}$, contradicting the assumption.  Next, since $\bla\in\Cali^\bs(\ell)$ then by Lemma \ref{lemma:cylindrical} we have $s_{j(b)}-\co^\bs(b_{\mathrm{min}})< e$ where $b_{\mathrm{min}}$ is the box of smallest charged content in $\bla$. Thus if $\co^\bs(b)$ is the largest element of $\Bo^\bs(\bmu)$ then $\Bo^\bs(\bmu)\subseteq[z,z+e-1]$ and we repeat the argument above to get a contradiction. We conclude that $\co^\bs(b)$ the smallest element of $\Bo^\bs(\bmu)$, and moreover (again by the same argument), $\co^\bs(b)\neq \co^\bs(b_{\mathrm{max}})-e+1$.

It follows that there exists a unique $w\in\mathbb{Z}$ such that $w=\co^\bs(b)+qe$ for some positive integer $q$, $w-e<x<w+e$ for all $x\in\Bo^\bs(\bla)=\Bo^\bs(\bmu)\setminus\{\co^\bs(b)\}$, and there is some $x\in\Bo^\bs(\bla)$ with $x\geq w$. We observe that $\bla$ cannot have an addable box of content $w$ in one of its nonzero rows, for if it did then $\tilde{e}_i$ would have added this box rather than $b$. Thus if $w-1\in\Bo^\bs(\bmu)$, it must also hold that $w\in\Bo^\bs(\bmu)$.

\smallskip\noindent {\em Step 1.} {\em Case (i)}. Suppose $w-1\notin\Bo^\bs(\bmu)$. Consider $\Bo^\bs(\bmu)\cap\mathbb{Z}_{\geq w}$, and the corresponding boxes in the right border of $\bmu$ that are of charged content at least $w$. Let $r=|\Bo^\bs(\bmu)\cap\mathbb{Z}_{\geq w}|$. Since  the reading word of $\Bo^\bs(\bla)$ is increasing, these boxes belong to the top $r$ rows  of $\mathsf{Y}^\bs(\bla)$. We have $0<r\leq e-1$. If $\Bo^\bs(\bmu)\cap\mathbb{Z}_{\geq w}=[w,w+r-1]$,  proceed to Step 2. Otherwise, let $y$ be the minimal element of $\Bo^\bs(\bmu)\cap\mathbb{Z}_{\geq w}$. Apply $\tilde{e}:=\tilde{e}_{i+1}\tilde{e}_{i+2}\ldots\tilde{e}_{y-1}\tilde{e}_y$ to $\bmu$. This removes a horizontal strip of $y-w+1$ boxes from the $r$'th row of $\mathsf{Y}^\bs(\bmu)$, because $\bmu=\bla\cup\{b\}$ has no addable boxes of these residues with charged content $w$ or greater nor does it have any (other) removable boxes of these residues. Applying $\tilde{e}_i\tilde{e}_i$ to $\tilde{e}\bmu$ now removes $b$ and the removable $i$-box in row $r$ yielding $\bnu$ such that  $\bmu=\tilde{f}_y\tilde{f}_{y-1}\ldots\tilde{f}_{i+2}\tilde{f}_{i+1}\tilde{f}_i\tilde{f}_i\bnu$, a contradiction.
This step of the proof is illustrated in Figure \ref{pf:calistep1case1}.


\begin{figure}[h!]
\begin{subfigure}[h]{.8\textwidth}
\begin{ytableau}
\none&\none&\none  &\none&\none&\none  &\none  &\none&\none&\none&\none&\none&*(SkyBlue!65)&*(SkyBlue!65)  &*(SkyBlue!65)&*(SkyBlue!65)&*(SkyBlue!65)&*(SkyBlue!65)&*(SkyBlue!65)&*(SkyBlue!65)&*(SkyBlue!65)&*(SkyBlue!65)32\\
\none  &\none&\none&\none  &\none&\none&\none  &\none  &\none&\none&\none&\none&*(SkyBlue!65)&*(SkyBlue!65)&*(SkyBlue!65)  &*(SkyBlue!65)&*(SkyBlue!65)&*(SkyBlue!65)&*(SkyBlue!65)&*(SkyBlue!65)&*(SkyBlue!65)30\\
\none&\none  &\none&\none&\none  &\none&\none&\none  &*(Melon!65)  &*(Melon!65)&*(Melon!65)&*(Melon!65)&*(Melon!65)&*(Melon!65)&*(Melon!65)&*(Melon!65)  &*(Melon!65)&*(Melon!65)&*(Melon!65)&*(Melon!65)&*(Melon!65)29\\
\none&\none&\none  &\none&\none&\none  &\none&\none&*(Thistle!65)  &*(Thistle!65)  &*(Thistle!65)&*(Thistle!65)&*(Thistle!65)&*(Thistle!65)&*(Thistle!65)&*(Thistle!65)&*(Thistle!65)  &*(Thistle!65)&*(Thistle!65)&*(Thistle!65)&*(Thistle!65)28\\
\none&\none&\none&\none  &\none&\none&\none  &*(SpringGreen!65)&*(Sepia!60)15&*(SpringGreen!65)  &*(SpringGreen!65)  &*(SpringGreen!65)&*(SpringGreen!65)&*(SpringGreen!65)&*(SpringGreen!65)&*(SpringGreen!65)&*(SpringGreen!65)&*(SpringGreen!65)  &*(SpringGreen!65)25\\
\none&\none&\none&\none&\none  &\none&\none&*(SpringGreen!65)  &*(SpringGreen!65)&*(SpringGreen!65)&*(SpringGreen!65)  &*(SpringGreen!65)  &*(SpringGreen!65)&*(SpringGreen!65)&*(SpringGreen!65)&*(SpringGreen!65)&*(SpringGreen!65)&*(SpringGreen!65)&*(SpringGreen!65)24 \\
\none&\none&\none&\none&\none&\none  &\none &*(SpringGreen!65)&*(SpringGreen!65)  &*(SpringGreen!65)&*(SpringGreen!65)&*(SpringGreen!65)  &*(SpringGreen!65)  &*(SpringGreen!65)&*(SpringGreen!65)&*(SpringGreen!65)&*(SpringGreen!65)21\\
\end{ytableau}
\caption{Suppose $e=12$. Then $\bla=((12^2,10),(13),(13),(10,9))\in\Cali^{\bs}(4)$ for $\bs=(14,16,17,23)$. Suppose $\tilde{e}_3$ was applied to $\bla$ adding the darkened box of content $15$ to the bottom of $\la^2$ to yield $\bmu=\tilde{e}_3\bla=((12^2,10),(13,1),(13),(10,9))$. Definition \ref{def:cali}(1)(a) fails for $\bmu$ since $15$ and $32$ are in $\Bo^\bs(\bla)$ and $32-15>11=e-1$. We have $w=15+e=27$, and $w-1=26\notin\Bo^\bs(\bla)$, thus we apply Case (i). }
\end{subfigure}
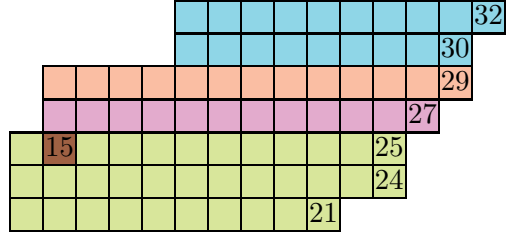
\begin{subfigure}[h]{0.8\textwidth}
\begin{ytableau}
 \none&\none&\none  &\none&\none&\none  &\none  &\none&\none&\none&\none&\none&*(SkyBlue!65)&*(SkyBlue!65)  &*(SkyBlue!65)&*(SkyBlue!65)&*(SkyBlue!65)&*(SkyBlue!65)&*(SkyBlue!65)&*(SkyBlue!65)&*(SkyBlue!65)&*(SkyBlue!65)32\\
\none  &\none&\none&\none  &\none&\none&\none  &\none  &\none&\none&\none&\none&*(SkyBlue!65)&*(SkyBlue!65)&*(SkyBlue!65)  &*(SkyBlue!65)&*(SkyBlue!65)&*(SkyBlue!65)&*(SkyBlue!65)&*(SkyBlue!65)&*(SkyBlue!65)30\\
\none&\none  &\none&\none&\none  &\none&\none&\none  &*(Melon!65)  &*(Melon!65)&*(Melon!65)&*(Melon!65)&*(Melon!65)&*(Melon!65)&*(Melon!65)&*(Melon!65)  &*(Melon!65)&*(Melon!65)&*(Melon!65)&*(Melon!65)&*(Melon!65)29\\
\none&\none&\none  &\none&\none&\none  &\none&\none&*(Thistle!65)  &*(Thistle!65)  &*(Thistle!65)&*(Thistle!65)&*(Thistle!65)&*(Thistle!65)&*(Thistle!65)&*(Thistle!65)&*(Thistle!65)  &*(Thistle!65)&*(Thistle!65)&*(Thistle!65)27\\
\none&\none&\none&\none  &\none&\none&\none&*(SpringGreen!65)&*(Sepia!60)15&*(SpringGreen!65)  &*(SpringGreen!65)  &*(SpringGreen!65)&*(SpringGreen!65)&*(SpringGreen!65)&*(SpringGreen!65)&*(SpringGreen!65)&*(SpringGreen!65)&*(SpringGreen!65)  &*(SpringGreen!65)25\\
\none&\none&\none&\none&\none  &\none&\none&*(SpringGreen!65)  &*(SpringGreen!65)&*(SpringGreen!65)&*(SpringGreen!65)  &*(SpringGreen!65)  &*(SpringGreen!65)&*(SpringGreen!65)&*(SpringGreen!65)&*(SpringGreen!65)&*(SpringGreen!65)&*(SpringGreen!65)&*(SpringGreen!65)24 \\
\none&\none&\none&\none&\none&\none  &\none&*(SpringGreen!65)&*(SpringGreen!65)  &*(SpringGreen!65)&*(SpringGreen!65)&*(SpringGreen!65)  &*(SpringGreen!65)  &*(SpringGreen!65)&*(SpringGreen!65)&*(SpringGreen!65)&*(SpringGreen!65)21\\
\end{ytableau}
\caption{We have $r=4$. The horizontal strip of boxes in row $4$ with charged content greater than $w$ is a single box. Removing it, we obtain the diagram shown. This is given by the crystal operator $\tilde{e}_4$, since there is no addable $4$-box above or to the right in the diagram or below on the same diagonal.
We have $\bnu=((12^2,10),(12,1),(13),(10,9))=\tilde{e}_4\bmu$.}
\end{subfigure}
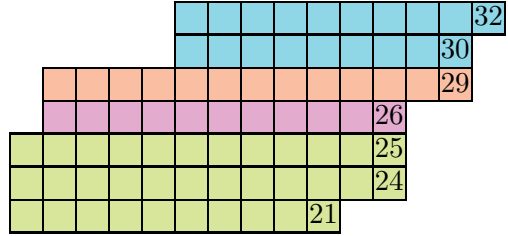
\begin{subfigure}[h]{0.8\textwidth}

\begin{ytableau}
\none&\none&\none  &\none&\none&\none  &\none  &\none&\none&\none&\none&\none&*(SkyBlue!65)&*(SkyBlue!65)  &*(SkyBlue!65)&*(SkyBlue!65)&*(SkyBlue!65)&*(SkyBlue!65)&*(SkyBlue!65)&*(SkyBlue!65)&*(SkyBlue!65)&*(SkyBlue!65)32\\
\none  &\none&\none&\none  &\none&\none&\none  &\none  &\none&\none&\none&\none&*(SkyBlue!65)&*(SkyBlue!65)&*(SkyBlue!65)  &*(SkyBlue!65)&*(SkyBlue!65)&*(SkyBlue!65)&*(SkyBlue!65)&*(SkyBlue!65)&*(SkyBlue!65)30\\
\none&\none  &\none&\none&\none  &\none&\none&\none  &*(Melon!65)  &*(Melon!65)&*(Melon!65)&*(Melon!65)&*(Melon!65)&*(Melon!65)&*(Melon!65)&*(Melon!65)  &*(Melon!65)&*(Melon!65)&*(Melon!65)&*(Melon!65)&*(Melon!65)29\\
\none&\none&\none  &\none&\none&\none  &\none&\none&*(Thistle!65)  &*(Thistle!65)  &*(Thistle!65)&*(Thistle!65)&*(Thistle!65)&*(Thistle!65)&*(Thistle!65)&*(Thistle!65)&*(Thistle!65)  &*(Thistle!65)&*(Thistle!65)26\\
\none&\none&\none&\none  &\none&\none&\none  &*(SpringGreen!65)&*(SpringGreen!65)&*(SpringGreen!65)  &*(SpringGreen!65)  &*(SpringGreen!65)&*(SpringGreen!65)&*(SpringGreen!65)&*(SpringGreen!65)&*(SpringGreen!65)&*(SpringGreen!65)&*(SpringGreen!65)  &*(SpringGreen!65)25\\
\none&\none&\none&\none&\none  &\none&\none&*(SpringGreen!65)  &*(SpringGreen!65)&*(SpringGreen!65)&*(SpringGreen!65)  &*(SpringGreen!65)  &*(SpringGreen!65)&*(SpringGreen!65)&*(SpringGreen!65)&*(SpringGreen!65)&*(SpringGreen!65)&*(SpringGreen!65)&*(SpringGreen!65)24 \\
\none&\none&\none&\none&\none&\none  &\none&*(SpringGreen!65)&*(SpringGreen!65)  &*(SpringGreen!65)&*(SpringGreen!65)&*(SpringGreen!65)  &*(SpringGreen!65)  &*(SpringGreen!65)&*(SpringGreen!65)&*(SpringGreen!65)&*(SpringGreen!65)21\\
\end{ytableau}

\caption{Now we may remove the two removable boxes of residue $3$ from $\bnu$ by applying $\tilde{e}_3$ twice in a row, yielding the diagram shown. Thus $\bmu=\tilde{f}_4\tilde{f}_3\tilde{f}_3((12^2,10),(11),(13),(10,9))$ and so $\bmu$ does not label a calibrated representation.}
\end{subfigure}
\caption{Illustration of the proof of Theorem \ref{thm:cali=cali}, induction step verifying that failure of Definition \ref{def:cali}(1)(a) implies failure to be calibrated. Step 1, Case (i).}\label{fig:1aCasei}
\label{pf:calistep1case1}
\end{figure}

\smallskip\noindent {\em Step 1.  Case (ii)}. Suppose $w-1\in\Bo^\bs(\bmu)$.  Let $w-e<x_t<\ldots<x_2<x_1<w$ be the elements of $\Bo^\bs(\bla)\cap\mathbb{Z}_{<w}$. If $x_1<w-1$, proceed to Step 2. Otherwise, let $h\in\{1,\ldots,t\}$ be minimal such that $x_h-1\notin\Bo^\bs(\bla)$. If $x_h-1+e\notin\Bo^\bs(\bla)$ then set $\bnu=\tilde{e}_{i-1}\ldots\tilde{e}_{i-h}\bmu$. In $\mathsf{Y}^\bs(\bla)$, this corresponds to removing a vertical strip consisting of the boxes lying in the same column and strictly below the box in row $r$ of content $w$. If $x_h-1+e\in\Bo^\bs(\bla)$ (which can only happen if $x_h$ is the minimal element of $\Bo^\bs(\bla)$ and $x_h-1+e$ is the maximal element of $\Bo^\bs(\bla)$), apply the procedure in Case (i) to obtain $\bnu_1$ from $\bmu$ such that $\Bo^\bs(\bnu_1)\cap\mathbb{Z}_{\geq w}=[w,w+r-1]$, and then set $\bnu=\tilde{e}_{i-1}\ldots\tilde{e}_{i-h}\bnu_1$.
Since $w-1\in\Bo^\bs(\bmu)$ forced $w\in\Bo^\bs(\bmu)$ as well, we obtain $\bnu$ with $w\in\Bo^\bs(\bnu)$ and $w-1\notin\Bo^\bs(\bnu)$.

\smallskip\noindent {\em Step 2}. Recall that $w\cong \co^\bs(b)\cong i \mod e$. The $\ell$-partition $\bnu$ has no addable $i$-box in any of its nonzero rows. It has two removable $i$-boxes: $b$, and the box of content $w$ in its border. Since $\tilde{f}_i$ added the box $\co^\bs(b)$ to $\bla$, $\bnu$ has no addable $i$-box bigger than $b$ in an empty row. The $i$-word for $\bnu$ is thus some number of plusses followed by $--$. It follows that $\tilde{e}_i\tilde{e}_i$ removes the boxes of charged contents $w$ and $\co^\bs(b)$ from $\bnu$. This shows that $\bmu=  \tilde{f}_i\tilde{f}_i(\tilde{e}_i\tilde{e}_i\bnu)$,  a contradiction with the assumption that $\bmu=\tilde{f}_i\bla$ labels a calibrated representation. Therefore Definition \ref{def:cali}(1)(a) must hold if $\tilde{f}_i\bla$ satisfies the condition of Theorem \ref{thm:nostuttering}.

Finally, we suppose that $\tilde{f}_i\bla\neq 0$ satisfies the condition of Theorem \ref{thm:nostuttering} but that Definition \ref{def:cali}(1)(b) fails for $\tilde{f}_i\bla$. The argument by contradiction is similar to the one checking Definition \ref{def:cali}(1)(a). Define $b$ and $j(b)$ be as above. 
By assumption the reading word of $\Bo^\bs(\bmu)$ is not increasing, but the reading word of $\Bo^\bs(\bla)$ is. By the previous step of the proof, we know that the elements of $\Bo^\bs(\bmu)$ all belong to an interval $[z,z+e-1]$ for some $z\in\mathbb{Z}$. Since we are in the case that adding $b$ creates a new row in component $j(b)$, and $\aatch(\bla)$ is $\bs$-admissible by Lemma \ref{admiss} we have $\co^\bs(b)<\co^\bs(b')$ for all $\co^\bs(b')\in\Bo^{s_k}(\mu^k)$, $k\geq j(b)$.
Thus, there must exist a box $b'\in\mu^m=\la^m$ for some $m<j(b)$ such that $b'$ is the last box in its row and $\co^\bs(b)\leq\co^\bs(b')$, and we take $m$ to be minimal such that this happens. Now we remove either a horizontal strip, a vertical strip, or a combination of such from $\la^m$ by crystal operators in order to arrive at a situation where $\tilde{e}_i$ can be applied twice in a row to get a non-zero $\ell$-partition, as in previous part of the proof. We arrive at that situation in the following ways depending on how the diagonal of charged content $\co^\bs(b)$ intersects the Young diagram of $\la^m$ in $\mathsf{Y}^\bs(\bla)$. 

First, if $\co^\bs(b)\leq\co^\bs(b')$ for all boxes $b'$ the last in their rows in $\la^m$ then we remove all boxes of content larger than $\co^\bs(b)$ from the bottom row of $\la^m$. We then have two removable boxes of residue $i$. These are successively good removable unless the largest element of $\Bo^\bs(\bla)$ has residue $i-1$. If the latter is the case then we remove the topmost vertical strip in the border of $\mathsf{Y}^\bs(\bla)$. Then we may remove the two boxes of residue $i$ by applying $\tilde{e}_i$ twice in a row.

Otherwise, $i$ is the residue of a box in the right border ribbon of $\la^m$. We consider whether it occurs in an arm, a leg, or a corner of this ribbon. We then remove, respectively, the arm to its right, the leg below it, or the arm and the leg below it and to its right. In case the bottom box of the leg below it is the minimal element $z$ of $\Bo^\bs(\bla)$ and both $z,z+e-1\in\Bo^\bs(\bla)$ then again we have to remove the topmost vertical strip in the border of $\mathsf{Y}^\bs(\bla)$ before we can remove that leg via crystal operators $\tilde{e}_j$. We then arrive again at the situation that we may apply $\tilde{e}_i$ twice in a row to remove two boxes of residue $i$. Therefore, the irreducible representation labeled by $\bmu$ is not calibrated.  
\end{proof}

%
%
%
%

%



\section{Alcove geometries  and path combinatorics} 
\label{newsec3}

We now set about providing a homological construction of the  calibrated representations via BGG resolutions.  
This means understanding these simple modules in terms of the  Specht theory  of the cyclotomic Hecke algebra.  
%
  The homological and representation theoretic structure of the cyclotomic Hecke algebra is governed by strong uni-triangularity properties ---
 thus we can understand a given  calibrated simple $D_\bs(\bla)$ 
  in terms of the Serre subcategory arising from the poset  $\Lambda
  =\{\bmu \mid  \bmu \trianglerighteq \bla   \}\subseteq \mathscr{P} _\ell(n)$. We will cast each     $\bmu\in \Lambda$ as a point in  an $h$-dimensional alcove geometry under the action of the affine symmetric group $  \widehat{\mathfrak{S}}_h$;
the calibrated simple  $D_\bs(\bla)$  will belong to the fundamental alcove  under this action. 
  We regard   $  \widehat{\mathfrak{S}}_h$ as a Coxeter group, we  let $\ell : \widehat{\mathfrak{S}}_h \to \mathbb N$  denote the corresponding length function and we let  $\leq $ denote the strong Bruhat order on $\widehat{\mathfrak{S}}_h$.


\subsection{The alcove geometry} \label{fundamental-alcove}
Set $\aatch= h _1+\dots+ h _\ell$.  
 For each   $1\leq m \leq  \ell$ and $ 1\leq i \leq h_m$ we let
  $\varepsilon_{i,m}:=\varepsilon_{( h _\ell+\dots+ h _{m+1}) + i}$    denote a
formal symbol, and define an   $\aatch  $-dimensional real vector space 
\[
{\mathbb E}_{\aatchpair   }
=\bigoplus_{
	\begin{subarray}c 
	1\leq m \leq \ell   \\ 
		1\leq i \leq   h _m    
	\end{subarray}
} \mathbb{R}\varepsilon_{i,m}
\]
and $\overline{\mathbb E}_{\aatchpair  }$ to be the quotient of this space by the one-dimensional subspace spanned by 
\[\sum_{
	\begin{subarray}c 
	1\leq m \leq \ell   \\ 
		1\leq i \leq   h _m    
	\end{subarray}
} \varepsilon_{i,m}.\]
We have an inner product $\langle \; , \; \rangle$ on ${\mathbb E}_{\aatchpair  }$ given by extending
linearly the relations 
\[
\langle \varepsilon_{i,p} , \varepsilon_{j,q} \rangle= 
\delta_{i,j}\delta_{p,q}
\]
for all   $1 \leq p,q \leq  \ell $, $1\leq i \leq h_p$ and 
$1\leq  j\leq h_q$, where
$\delta_{i,j}$ is the Kronecker delta.
We identify $\lambda \in  {\mathcal P}_{\aatchpair  }(n)$ with an element of the integer lattice inside $\mathbb{E}_{\aatchpair  }  $ via the map
$$\bla \longmapsto
 \sum_{\begin{subarray}c   1\leq {m}\leq \ell \\ 1\leq i\leq  h _m  \end{subarray}}
 \lambda^{m}_i \varepsilon_{i,m}.$$
 We let $\Phi$ denote the root system of type $A_{\aatch  -1}$ consisting of the roots 
$$\{\varepsilon_{i,p}-\varepsilon_{j,q}:  \ 0\leq p,q<\ell, \
1\leq i \leq  h _p ,1\leq j \leq  h _q, 
 \text{with}\ (i,p)\neq (j,q)\}$$
and $\Phi_0$ denote the root system of type $A_{ h _1-1}\times\cdots\times A_{ h _\ell-1}$ consisting of the roots 
$$\{\varepsilon_{i,m}-\varepsilon_{j,m}: 
  1\leq m\leq \ell, 1\leq i \neq j\leq  h _m \}.$$
We choose $\Delta$ (respectively $\Delta_0$) to be the set of simple roots inside $\Phi$ (respectively  $\Phi_0$) of the form $\varepsilon_t-\varepsilon_{t+1}$ for some $t$.
Given $r\in\ZZ$ and $\alpha\in\Phi$ we define $s_{\alpha,re}$ to be the reflection which acts on ${\mathbb E}_{\aatchpair  }$ by
$$s_{\alpha,re}x=x-(\langle x,\alpha\rangle -re)\alpha$$
The group generated by the $s_{\alpha,0}$ with $\alpha\in\Phi$ (respectively $\alpha\in\Phi_0$) is isomorphic to the symmetric group $\mathfrak{S}_{\enn   }$ (respectively to $\mathfrak{S}_{\aatchpair}:=\mathfrak{S}_{ h _1}\times\cdots\times\mathfrak{S}_{ h _\ell}$), while the group generated by the $s_{\alpha,re}$ with $\alpha \in\Phi$ and $r\in\ZZ$ is isomorphic to $\widehat{\mathfrak{S}}_{\enn   }$, the affine Weyl group of type $A_{\enn   -1}$. 
    We set $\alpha_0=\varepsilon_{\enn}-\varepsilon_1$ and $\Pi=\Delta\cup\{\alpha_0\}$.
The elements
 $S=
\{s_{\alpha,0}:\alpha\in\Delta\}\cup\{s_{\alpha_0,-e}\}
$ 
generate $\widehat{\mathfrak S}_{\enn}$. 

\begin{notn}
	We shall frequently find it convenient to refer to the generators in $S$ in terms of the elements of $\Pi$, and will abuse notation in two different ways. First, we will write $s_{\alpha}$ for $s_{\alpha,0}$ when $\alpha\in\Delta$ and $s_{\alpha_0}$ for $s_{\alpha_0,-e}$. This is unambiguous except in the case of the affine reflection $s_{\alpha_0,-e}$, where this notation has previously been used for the element $s_{\alpha,0}$. As the element $s_{\alpha_0,0}$ will not be referred to hereafter this should not cause confusion.
	Second, we will write $\alpha=\varepsilon_i-\varepsilon _{i+1}$ in all cases; if $i=\enn$ then all occurrences of $i+1$ should be interpreted modulo $\enn$ to refer to the index $1$.
\end{notn}

We shall consider a shifted action of the affine Weyl group $\widehat{\mathfrak{S}}_{\aatch  }$  on ${\mathbb E}_{h,l}$ 
by the element
$$ \rho:= (\rho_{\ell},  \ldots, \rho_{2},\rho_1) \in \ZZ^{\aatch  } \quad\text{where}\quad
\rho_m := (s_m , s_m-1, \dots , s_m -h_m+1) \in \ZZ^{h_m},$$    
\noindent that is, given an element $w\in \widehat{\mathfrak{S}}_{\aatch  } $,  we set 
$
w\cdot x=w(x+\rho)-\rho.
$
 This shifted action induces a well-defined action on $\overline{\mathbb E}_{\aatchpair  }$; we will define various geometric objects in ${\mathbb E}_{\aatchpair  }$ in terms of this action, and denote the corresponding objects in the quotient with a bar without further comment.
We let ${\mathbb E} ({\alpha, re})$ denote the affine hyperplane
consisting of the points  
$${\mathbb E} ({\alpha, re}) = 
\{ x\in{\mathbb E}_{\aatchpair  } \mid  s_{\alpha,re} \cdot x = x\} .$$
Note that our assumption that $e> h _1+\dots+ h _\ell$  implies that the origin does not lie on any hyperplane.   Given a
hyperplane ${\mathbb E} ( \alpha,re)$ we 
remove the hyperplane from ${\mathbb E}_{\aatchpair  }$ to obtain two
distinct subsets ${\mathbb 
	E}^{\great}(\alpha,re)$ and ${\mathbb E}^{\less}(\alpha,re)$
where the origin  lies in $  {\mathbb E}^{\less }(\alpha,re)$. The connected components of 
$$\overline{\mathbb E}_{\aatchpair  }  \setminus (\cup_{\alpha \in \Phi_0}\overline{\mathbb E}(\alpha,0))$$ are called chambers.
 The dominant chamber, denoted
$\overline{\mathbb E}_{\aatchpair   }^+ $, is defined to be 
$$\overline{\mathbb E}_{\aatchpair   }^+=\bigcap_{ \begin{subarray}c
	\alpha \in \Phi_0
	\end{subarray}
} \overline{\mathbb E}^{\less} (\alpha,0).$$
The connected components of $$\overline{\mathbb E}_{\aatchpair  }  \setminus (\cup_{\alpha \in \Phi,r\in \ZZ}\overline{\mathbb E}(\alpha,re))$$
are called alcoves, and any such alcove is a fundamental domain for the action of the group $  \widehat{\mathfrak{S}}_{\aatch  }$ on
the set $\Alc$ of all such alcoves. We define the {\em fundamental alcove}, 
which we denote by $\mathcal{F}_\aatchpair\subseteq \mathbb E_\aatchpair $,  to be the alcove containing the origin (which is inside the dominant chamber) and we set $$\mathcal{F}_\aatchpair(n)= \mathcal{F}_\aatchpair \cap \{\bla \in \mathbb E_\aatchpair \mid 
\textstyle \sum_{1\leq m \leq \ell }\sum_{i\geq 1} \la_i^{m}=n \}.$$
  We have a bijection from $\widehat{\mathfrak{S}}_{\aatch  }$ to $\Alc$ given by $w\longmapsto w\mathcal{F}_\aatchpair $. Under this identification $\Alc$ inherits a right action from the right action of $\widehat{\mathfrak{S}}_{\aatch  }$ on itself.
Consider the subgroup 
$$
\mathfrak{S}_{\aatchpair}:=
\mathfrak{S}_{ h _1}\times\cdots\times\mathfrak{S}_{ h _\ell} \leq \widehat{\mathfrak{S}}_{\aatch  }.
$$
The dominant chamber is a fundamental domain for the action of $\mathfrak{S}_{\aatchpair}$ on the set of chambers in $\overline{\mathbb E}_{\aatchpair  }$.  
We let $ \mathfrak{S}^{\aatchpair}$ denote the set of minimal length representatives for right cosets $\mathfrak{S}_{\aatchpair} \backslash \widehat{\mathfrak{S}}_{\aatch  }$.  So multiplication gives a bijection $\mathfrak{S}_{\aatchpair}\times \mathfrak{S}^{\aatchpair} \to  \widehat{\mathfrak{S}}_{\aatch  }$. 
 This induces a bijection between right cosets and the alcoves in our dominant chamber. 

 \color{black}
If the intersection of a hyperplane $\overline{\mathbb E}(\alpha,re)$ with the closure of an alcove $A$ is generically  of codimension one in $\overline{\mathbb E}_{\aatchpair  }$ then we call this intersection a {\em wall} of $A$. The fundamental alcove $\mathcal{F}_\aatchpair $ has walls corresponding to $\overline{\mathbb E}(\alpha,0)$ with $\alpha\in\Delta$ together with an affine wall $\overline{\mathbb E}(\alpha_0,-e)$. We will usually just write $\overline{\mathbb E}(\alpha)$ for the walls $\overline{\mathbb E}(\alpha,0)$ (when $\alpha\in\Delta$) and $\overline{\mathbb E}(\alpha,-e)$ (when $\alpha=\alpha_0$). We regard each of these walls as being labelled by a distinct colour (and assign the same colour to the corresponding element of $S$). Under the action of $\widehat{\mathfrak{S}}_{\aatch  }$ each wall of a given alcove $A$ is in the orbit of a unique wall of $\mathcal{F}_\aatchpair $, and thus inherits a colour from that wall. We will sometimes use the right action of $\widehat{\mathfrak{S}}_{\aatch  }$ on $\Alc$. Given an alcove $A$ and an element $s\in S$, the alcove $As$ is obtained by reflecting $A$ in the wall of $A$ with colour corresponding to the colour of $s$. With this observation it is now easy to see that if $w=s_{1}\ldots s_{t}$ where the $s_i$ are in $S$ then $w\mathcal{F}_\aatchpair $ is the alcove obtained from $\mathcal{F}_\aatchpair $ by successively reflecting through the walls corresponding to $s_1$ up to $s_t$. 
%

\subsection{Paths in the geometry}\color{black}
We now  
develop a path combinatorics inside our geometry.  
Given   a map  $p: 
\{1,\dots 
, n\}\to \{1,\dots  ,   \aatch      \}$ we define points $\SSTP(k)\in
{\mathbb E}_{\aatchpair  }$ by
\[
\SSTP(k)=\sum_{1\leq i \leq k}\varepsilon_{p(i)} 
\]
for $1\leq i \leq n$. 
We define the associated path    by $$\SSTP=\left(
\varnothing=\SSTP(0),\SSTP(1),\SSTP(2), \ldots, \SSTP(n) \right) $$ and we say that 
the path has shape $\pi= \SSTP(n) \in   {\mathbb E}_{\aatchpair  }$.    
We also denote this path by 
$$\SSTP=(\varepsilon_{p(1)},\ldots,\varepsilon_{p(n)}).$$
Given $\bla \in \mathbb{E}_{\aatchpair} $, we let $\Path(\lambda)$ denote the set of paths of length $n$ with shape $\lambda$. We define $\Path_{\aatchpair  }(\lambda)$ to be the subset of $\Path(\lambda)$ consisting of those paths lying entirely inside the dominant chamber; i.e. those $\SSTP$ such that $\SSTP(i)$ is dominant for all $0\leq i\leq n$. 
 We let $\Path_{\aatchpair  }(n) = \cup _{\la\in \mathscr{P}_{\underline{ h }}(n)}\Path_{\aatchpair  }(\la)$.  
 %

  \color{black}

Given paths $\SSTP=(\varepsilon_{p(1)},\ldots,\varepsilon_{p(n)})$ and $\SSTQ=(\varepsilon_{q(1)},\ldots,\varepsilon_{q(n)})$ we say that $\SSTP\sim\SSTQ$ if there exists an $\alpha 
 \in\Phi$ and $r\in\ZZ$ and $s\leq n$ such that
$$\SSTP(s)\in{\mathbb E}(\alpha,re)\qquad \text{ 
and  }
\qquad \varepsilon_{q(t)}=\left\{\begin{array}{ll}
\varepsilon_{p(t)}& \ \text{for}\ 1\leq t\leq s\\
s_{\alpha}\varepsilon_{p(t)}& \ \text{for}\ s+1\leq t\leq n.\end{array}\right.$$ 

In other words the paths $\SSTP$ and $\SSTQ$ agree up to some point $\SSTP(s)=\SSTQ(s)$ which lies on ${\mathbb E}(\alpha,re)$, after which each $\SSTQ(t)$ is obtained from $\SSTP(t)$ by reflection in ${\mathbb E}(\alpha,re)$. We extend $\sim$ by transitivity to give an equivalence relation on paths, and say that two paths in the same equivalence class are related by a series of {\em wall reflections of paths} and given $\SSTS \in \Path_{\aatchpair  }(n)$ we set $[\SSTS]  = \{ \SSTT \in \Path_{\aatchpair  }(n) \mid \SSTS\sim\SSTT\}$.  
Given a path $\SSTP$ we define $$\res(\SSTP)=(\res_\SSTP(1),\ldots,\res_\SSTP(n))$$ where $\res_\SSTP(i)$ denotes the residue of the box labelled by $i$ in the tableau corresponding to $\SSTP$.
We have that $\res(\SSTP)=\res(\SSTQ)$  is and only if 
$\SSTP \sim \SSTQ$.

\begin{defn}\label{Soergeldegreee}
Given a path $\sts=(\sts(0),\sts(1),\sts(2), \ldots, \sts(n))$  we set
$\deg_{\bs } (\sts(0))=0$ and define  
 \[
 \deg_{\bs } (\sts ) = \sum_{1\leq k \leq n} d (\sts(k),\sts(k-1)), 
 \]
 where $d(\sts(k),\sts(k-1))$ is defined as follows. 
For $\alpha\in\Phi$ we set $d_{\alpha}(\sts(k),\sts(k-1))$ to be
\begin{itemize}
\item $+1$ if $\sts(k-1) \in 
   {\mathbb E}(\alpha,re)$ and 
   $\sts(k) \in 
   {\mathbb E}^{\less}(\alpha,re)$;
   
\item $-1$ if $\sts(k-1) \in 
   {\mathbb E}^{\great}(\alpha,re)$ and 
   $\sts(k) \in 
   {\mathbb E}(\alpha,re)$;
\item $0$ otherwise.  
   \end{itemize}
We let 
$$ 
{\rm deg}(\SSTS)= \sum _{1\leq k \leq n }\sum_{\alpha \in \Phi}d_\alpha(\sts(k-1),\sts(k)).$$  
%
%
%
%
 \end{defn}

\begin{defn} Given two paths 
	$$\SSTP=(\varepsilon_{i_1},\varepsilon_{i_2},\dots, \varepsilon_{i_p}) \in \Path(\mu)
	\quad\text{and}\quad
	\SSTQ=(\varepsilon_{j_1},\varepsilon_{j_2},\dots, \varepsilon_{j_q}) 
	\in \Path(\nu)$$  we define the {\em  naive concatenated path} 
	$$\SSTP\boxtimes \SSTQ =
	(\varepsilon_{i_1},\varepsilon_{i_2},\dots, \varepsilon_{i_p}, \varepsilon_{j_1},\varepsilon_{j_2},\dots, \varepsilon_{j_q}) 
	\in \Path (\mu+\nu).$$ 
\end{defn}
 
For $\bla \in \mathscr{P}_\aatchpair (n)$, we identify $\Path_{\aatchpair  }(\la)$ with the set of standard $\la$-tableaux in the obvious manner (see \cite{cell4us2} for more details).  This identification preserves the grading.  
  This identification is best illustrated via an example:

  \begin{eg}
  The tableau $\SSTT_{1,\la}$  in Example \ref{ex:revertable} corresponds to the path 
   $$
\SSTT_{1,\la}= ( \varepsilon_1, \varepsilon_2, \varepsilon_3, \varepsilon_4, \varepsilon_5, \varepsilon_6,
  \varepsilon_1, \varepsilon_2, \varepsilon_3, \varepsilon_5,
   \varepsilon_1, \varepsilon_2, \varepsilon_1, \varepsilon_2,\varepsilon_1).
 $$
\end{eg}

\subsection{The Bott--Samelson truncation and Soergel diagrammatics}
  We now recall the construction an idempotent subalgebra 
  of $\algebra$ which is   isomorphic to 
 Elias--Williamson's diagrammatic Hecke categories.  
In order to do this,  we must  restrict our attention to paths labelled by (enhanced) words in the  
 affine Weyl group.

\begin{defn}\label{alphabet} We will associate alcove paths to certain words in the {\em  alphabet}
	$$
	S\cup\{1\}=\{ s_\alpha \mid \alpha \in \Pi \cup \{\emptyset\}\}
	$$
	where $s_\emptyset =1$.  That is, we will consider words in the generators of the affine Weyl group, but enriched with explicit occurrences of the identity 
	in these expressions.  We refer to the number of elements in such an expression (including the occurrences of the identity) as the {\em  enhanced length} of this expression. 
We say that an enriched word is reduced if, upon forgetting occurrences of the identity in the expression, the resulting  word   is reduced.  	
 
	 \end{defn}

Given a path $\SSTP$ between points in the principal linkage class, the end point lies in the interior of an alcove of the form $w\mathcal{F}_\aatchpair $ for some $w\in\widehat{\mathfrak S}_{\aatch  }$. If we write $w$ as a word in our alphabet, and then replace each element $s_{\alpha}$ by the corresponding non-affine reflection $s_{\alpha}$ in ${\mathfrak S}_{\aatch  }$ to form the element $\overline{w}\in{\mathfrak S}_{\aatch  }$ then the basis vectors $\varepsilon_i$ are permuted by the corresponding action of $\overline{w}$ to give $\varepsilon_{\overline{w}(i)}$, and there is an isomorphism from $\overline{\mathbb E}_{h,l}$ to itself which maps $\mathcal{F}_\aatchpair $ to $w\mathcal{F}_\aatchpair $ such that $0$ maps to $w\cdot 0$, coloured walls map to walls of the same colour, and each basis element $\varepsilon_i$ map to $\varepsilon_{\overline{w}(i)}$. Under this map we can transform a path $\SSTQ$ starting at the origin to a path starting at $w\cdot 0$ which passes through the same sequence of coloured walls as $\SSTQ$ does.

\begin{defn}
	Given two paths 
	 $\SSTP=(\varepsilon_{i_1},\varepsilon_{i_2},\dots, \varepsilon_{i_p}) \in \Path(\mu)$
and $	\SSTQ=(\varepsilon_{j_1},\varepsilon_{j_2},\dots, \varepsilon_{j_q}) 
	\in \Path(\nu)$   with the endpoint of $\SSTP$ lying in  some alcove $w\mathcal{F}_\aatchpair $
	we define the {\em  contextualised concatenated path} 
	$$\SSTP\otimes_w \SSTQ =
	(\varepsilon_{i_1},\varepsilon_{i_2},\dots, \varepsilon_{i_p})\boxtimes  
	(\varepsilon_{\overline{w}(j_1)},\varepsilon_{\overline{w}(j_2) },\dots, \varepsilon_{\overline{w}(j_q) }) 
	\in \Path (\mu+(w \cdot \nu)).$$ 
	If $w=s_{ \color{magenta}\alpha} $ we will simply write $\SSTP\otimes_{ \color{magenta}\alpha} \SSTQ$.    \end{defn}
%

We now define the building blocks from which all of our distinguished paths will be constructed.  We begin by defining certain integers that describe the position of the origin in our fundamental alcove.

\begin{defn} Given $   { \color{magenta}\alpha}  \in \Pi$ we define 
	${b_{ \color{magenta}\alpha} }$ to be the distance from the origin to the wall corresponding to 
	$ {{ \color{magenta}\alpha} }$, and let $b_\emptyset =1$.    
	Given our earlier conventions this corresponds to setting
	$$
	b_{\varepsilon_{h_\ell+\dots+h_m + i }-\varepsilon_{h_\ell+\dots+h_m+i+1 }} = 1
	$$
	for $m>1$ and $1\leq i <h_{m-1}$   and that 
	$$
	b_{\varepsilon_{h_\ell+\dots+h_m }-\varepsilon_{ h_\ell+\dots+h_m+1 }} 
	= s_{m+1}-s_m -h_m+1 
	\qquad
	b_{\varepsilon_{\aatch  }-\varepsilon_{1}} = e+ s_{1}-s_{\ell } -h_{\ell}+1 
	$$
	for $m>1$.  
	Given ${ \color{magenta}\alpha},\bet\in \Pi$  we set $b_{\al\bet}=b_{ \color{magenta}\alpha} + b_\bet$.   
\end{defn}

We let $\delta_k=((k^{h_1}),(k^{h_2}),\dots, (k^{h_\ell}))$ and we note that these multipartitions always lie in the principal linkage class.
We sometimes write $\delta_{{ \color{magenta}\alpha}}$ for the element $\delta_{{b_{ \color{magenta}\alpha} }}$.  
 We can now define our basic building blocks for paths.

\begin{defn} \label{base} Given ${\al}  =\varepsilon_i -\varepsilon_{i+1}\in \Pi$, we consider the multicomposition 
	$ 
	s_\al\cdot \delta_{\al} 
	$ 
	with all columns of length ${b_{ \color{magenta}\alpha} }$, with the exception of the $i$th and $(i+1)$st columns, which are of length  $0$ and $2{b_{ \color{magenta}\alpha} }$, respectively.  
	We set 
	$$
	\REMOVETHESE {i} {} =  (\varepsilon_{1 }, \dots, \varepsilon_{i-1 }
	, \widehat{\varepsilon_{ i  }},   \varepsilon_{ i+1  },\dots, \varepsilon_{\aatch   })
	\quad \text{and}\quad 
	\ADDTHIS {i} {} =(+\varepsilon_i)$$
	where $\widehat{.}$ denotes omission of a coordinate.
	Then our distinguished path corresponding to $s_\al$ is given by
	$$ \SSTP_\al=
 	\REMOVETHESE{i} {{b_{ \color{magenta}\alpha} }} \boxtimes \ADDTHIS {i+1}{{b_{ \color{magenta}\alpha} }} \in \Path(s_\al\cdot\delta _{\al}).
	$$
 The distinguished path corresponding to $\emptyset$ is labelled by 
	$ 
	\SSTP_{\emptyset}
	\in \Path(\delta) 
	$ 
	and is  
 fixed to be any choice of tableau $\SSTT_{m,\delta}$ for which $1\leq m \leq \ell$ is a step change.  
We set $\SSTP_\emp=(\SSTP_\emptyset)^{{b_{ \color{magenta}\alpha} }}$.   
\end{defn}

\begin{rmk}
If $\ell$  is a step change, then we can take 
 $\SSTP_{\emptyset} = 
	(\varepsilon_{1 },\varepsilon_{2 }, \dots,   \varepsilon_{\aatch   })
$ and indeed this is the path used in \cite{cell4us2} (where it is implicitly assumed that $\ell$  is a step change).  
We further remark that one can always reorder the charge $\bs\in \ZZ^\ell$ to obtain some 
 $\widehat{\bs}\in \ZZ^\ell$ for which 
$\ell$  is a step change (using the trivial algebra isomorphism $\cyc_{n}(\bs)\cong \cyc_{n}(\widehat{\bs})$).  
\end{rmk}

We  are now ready to define our distinguished paths for general words in our alphabet.   

\begin{defn}\label{thepathwewatn}
	We now define a {\em  distinguished path} $\SSTP_{\w}$ for each word $\w$ in our alphabet $S\cup\{1\}$ by induction on the enhanced length of $\w$. If $\w$ is $s_\emptyset$ or a simple reflection $s_\alpha$ we have already defined the distinguished path in Definition \ref{base}. Otherwise if $\w=s_\alpha\w'$ then we define  
	$$
	\SSTP_{\underline{w}}:= \SSTP_{{ \color{magenta}\alpha} }
	\;\otimes_{\al}  \; \SSTP_{\underline{w}'}	.$$ 
	If the enriched word $\w$ is  reduced,  then   the corresponding  
	 path $\SSTP_{\w}$ is said to be a  {\em reduced path}.
  \end{defn}

%
%


\begin{defn}
	{ \renewcommand{\SSTU}{{\sf P}}\renewcommand{\mu}{{\pi}} 
		Given ${\al} \in \Pi$ we set 
		$$
		\reflectpath =
		\REMOVETHESE {i} {b_\al} \boxtimes \ADDTHIS {i} {{b_{ \color{magenta}\alpha} }} =
		\REMOVETHESE {i} {b_\al} \otimes_{ \color{magenta}\alpha} \ADDTHIS {i+1} {{b_{ \color{magenta}\alpha} }} =
		(+\varepsilon_{1 }, \dots, +\varepsilon_{i-1 }
		, \widehat{+\varepsilon_{ i  }},   +\varepsilon_{ i+1  },\dots, +\varepsilon_{\aatch   })^{b_{ \color{magenta}\alpha} } 
		\boxtimes  (\varepsilon_i)^{{b_{ \color{magenta}\alpha} }}
 		$$
		the path obtained by reflecting the second part of $\SSTP_\al$ in the wall through which it passes.   
	} 
	
\end{defn}

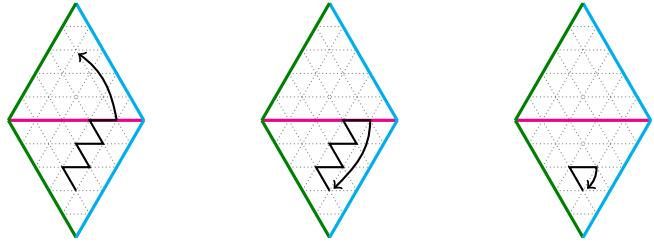
\begin{figure}[ht!]
 \vspace{-0.4cm}
$$
  \begin{minipage}{2.6cm}\begin{tikzpicture}[scale=0.6] 
    
    \path(0,0) coordinate (origin);
          \foreach \i in {0,1,2,3,4,5}
  {
    \path  (origin)++(60:0.6*\i cm)  coordinate (a\i);
    \path (origin)++(0:0.6*\i cm)  coordinate (b\i);
     \path (origin)++(-60:0.6*\i cm)  coordinate (c\i);
    }
  
      \path(3,0) coordinate (origin);
          \foreach \i in {0,1,2,3,4,5}
  {
    \path  (origin)++(120:0.6*\i cm)  coordinate (d\i);
    \path (origin)++(180:0.6*\i cm)  coordinate (e\i);
     \path (origin)++(-120:0.6*\i cm)  coordinate (f\i);
    }

  \foreach \i in {0,1,2,3,4,5}
  {
    \draw[gray, densely dotted] (a\i)--(b\i);
        \draw[gray, densely dotted] (c\i)--(b\i);
    \draw[gray, densely dotted] (d\i)--(e\i);
        \draw[gray, densely dotted] (f\i)--(e\i);
            \draw[gray, densely dotted] (a\i)--(d\i);
                \draw[gray, densely dotted] (c\i)--(f\i);
 
     }
  \draw[very thick, magenta] (0,0)--++(0:3) ;
    \draw[very thick, cyan] (3,0)--++(120:3) coordinate (hi);
        \draw[very thick, darkgreen] (hi)--++(-120:3) coordinate  (hi);

        \draw[very thick, darkgreen] (hi)--++(-60:3) coordinate  (hi);

    \draw[very thick, cyan] (hi)--++(60:3) coordinate (hi);

      \path(0,0)--++(-60:5*0.6)--++(120:2*0.6)--++(0:0.6) coordinate (hi);
     \path(0,0)--++(0:4*0.6)--++(120:3*0.6)           coordinate(hi2) ; 
     \path(0,0)--++(0:4*0.6)          coordinate(hi3) ;

          \path(hi)  --++(120:0.6)
           coordinate(step1) 
          --++(0:0.6)
                     coordinate(step2) 
      --++(120:0.6)
                 coordinate(step3) 
                 --++(0:0.6)           coordinate(step4) 
      --++(120:0.6) 
                 coordinate(step5)  --++(0:0.6) 
                 coordinate(step6)   ;    
     
                \path(0,0)--++(0:4*0.6)--++(120:2*0.6)           coordinate(hin) ; 

                     \path(0,0)--++(0:4*0.6)--++(120:2.85*0.6)           coordinate(hi4) ; 

        \draw[ thick,->]    (hi)  to    (step1) 
          to [out=0,in=-180]
                              (step2) 
 to   
            (step3) 
                     to [out=0,in=-180]       (step4) 
  to   
             (step5)       to [out=0,in=-180]   (step6) 
              (step6) 
              to [out=100,in=-35] (hi4)  ;

 \end{tikzpicture}\end{minipage}
\qquad
   \begin{minipage}{2.6cm}\begin{tikzpicture}[scale=0.6] 
    
    \path(0,0) coordinate (origin);
          \foreach \i in {0,1,2,3,4,5}
  {
    \path  (origin)++(60:0.6*\i cm)  coordinate (a\i);
    \path (origin)++(0:0.6*\i cm)  coordinate (b\i);
     \path (origin)++(-60:0.6*\i cm)  coordinate (c\i);
    }
  
      \path(3,0) coordinate (origin);
          \foreach \i in {0,1,2,3,4,5}
  {
    \path  (origin)++(120:0.6*\i cm)  coordinate (d\i);
    \path (origin)++(180:0.6*\i cm)  coordinate (e\i);
     \path (origin)++(-120:0.6*\i cm)  coordinate (f\i);
    }

  \foreach \i in {0,1,2,3,4,5}
  {
    \draw[gray, densely dotted] (a\i)--(b\i);
        \draw[gray, densely dotted] (c\i)--(b\i);
    \draw[gray, densely dotted] (d\i)--(e\i);
        \draw[gray, densely dotted] (f\i)--(e\i);
            \draw[gray, densely dotted] (a\i)--(d\i);
                \draw[gray, densely dotted] (c\i)--(f\i);
 
     }
  \draw[very thick, magenta] (0,0)--++(0:3) ;
    \draw[very thick, cyan] (3,0)--++(120:3) coordinate (hi);
        \draw[very thick, darkgreen] (hi)--++(-120:3) coordinate  (hi);

        \draw[very thick, darkgreen] (hi)--++(-60:3) coordinate  (hi);

    \draw[very thick, cyan] (hi)--++(60:3) coordinate (hi);

      \path(0,0)--++(-60:5*0.6)--++(120:2*0.6)--++(0:0.6) coordinate (hi);
     \path(0,0)--++(0:4*0.6)--++(120:3*0.6)           coordinate(hi2) ; 
     \path(0,0)--++(0:4*0.6)          coordinate(hi3) ;

          \path(hi)  --++(120:0.6)
           coordinate(step1) 
          --++(0:0.6)
                     coordinate(step2) 
      --++(120:0.6)
                 coordinate(step3) 
                 --++(0:0.6)           coordinate(step4) 
      --++(120:0.6) 
                 coordinate(step5)  --++(0:0.6) 
                 coordinate(step6)   ;

           \path(0,0)--++(0:4*0.6)--++(-120:3*0.6) --++(25:0.1)          coordinate(hin) ;

        \draw[ thick,->]    (hi)  to    (step1) 
          to [out=0,in=-180]
                              (step2) 
 to   
            (step3) 
                     to [out=0,in=-180]       (step4) 
  to   
             (step5)       to [out=0,in=-180]   (step6) to [out=-90,in=45] (hin)  ;

 \end{tikzpicture}\end{minipage} \qquad 
   \begin{minipage}{2.6cm}\begin{tikzpicture}[scale=0.6] 
    
    \path(0,0) coordinate (origin);
          \foreach \i in {0,1,2,3,4,5}
  {
    \path  (origin)++(60:0.6*\i cm)  coordinate (a\i);
    \path (origin)++(0:0.6*\i cm)  coordinate (b\i);
     \path (origin)++(-60:0.6*\i cm)  coordinate (c\i);
    }
  
      \path(3,0) coordinate (origin);
          \foreach \i in {0,1,2,3,4,5}
  {
    \path  (origin)++(120:0.6*\i cm)  coordinate (d\i);
    \path (origin)++(180:0.6*\i cm)  coordinate (e\i);
     \path (origin)++(-120:0.6*\i cm)  coordinate (f\i);
    }

  \foreach \i in {0,1,2,3,4,5}
  {
    \draw[gray, densely dotted] (a\i)--(b\i);
        \draw[gray, densely dotted] (c\i)--(b\i);
    \draw[gray, densely dotted] (d\i)--(e\i);
        \draw[gray, densely dotted] (f\i)--(e\i);
            \draw[gray, densely dotted] (a\i)--(d\i);
                \draw[gray, densely dotted] (c\i)--(f\i);
 
     }
  \draw[very thick, magenta] (0,0)--++(0:3) ;
    \draw[very thick, cyan] (3,0)--++(120:3) coordinate (hi);
        \draw[very thick, darkgreen] (hi)--++(-120:3) coordinate  (hi);

        \draw[very thick, darkgreen] (hi)--++(-60:3) coordinate  (hi);

    \draw[very thick, cyan] (hi)--++(60:3) coordinate (hi);

      \path(0,0)--++(-60:5*0.6)--++(120:2*0.6)--++(0:0.6) coordinate (hi);
     \path(0,0)--++(0:4*0.6)--++(120:3*0.6)           coordinate(hi2) ; 
     \path(0,0)--++(0:4*0.6)          coordinate(hi3) ;

          \path(hi)  --++(120:0.6)
           coordinate(step1) 
          --++(0:0.6)
                     coordinate(step2) 
      --++(120:0.6)
                 coordinate(step3) 
                 --++(0:0.6)           coordinate(step4) 
      --++(120:0.6) 
                 coordinate(step5)  --++(0:0.6) 
                 coordinate(step6)   ;

           \path(0,0)--++(0:4*0.6)--++(-120:3*0.6) --++(25:0.1)          coordinate(hin) ;

        \draw[ thick,->]    (hi)  to    (step1) 
          to [out=0,in=-180]
                              (step2) 
 to [out=-90,in=45] (hin)  ;

 \end{tikzpicture}\end{minipage}  
 $$
\caption{The first two diagrams are a path $\SSTP_\al$  walking through an ${\al}$-hyperplane, and a path $\reflectpath $  obtained by reflecting    $\SSTP_\al$ through this ${\al}$-hyperplane.
The final diagram is the path $\SSTP_\emptyset$.     } \label{figure1}
\end{figure}

 We let $\Std_{n,\bs} (\la)\subseteq \Std_\aatchpair (\la)$,  to be the set of all {\em alcove-tableaux} which can be obtained by contextualised concatenation of paths from the set 
$$\{ \SSTP_{ \color{magenta}\alpha} \mid { \color{magenta}\alpha} \in \Pi  \}\cup\{ \SSTP_\al^\flat  \mid { \color{magenta}\alpha} \in \Pi  \}\cup \{ \SSTP_\emptyset   \} .$$
We let $\mathscr{P}_{\underline{h}}(n,{\bs })=\{ \la\in    \mathscr{P}_{\underline{h}}(n) \mid \Std_{n,\bs} (\la)\neq \emptyset\}$.  
We let $\Std_{n,\bs}^+ (\la) \subseteq \Std_{n,\bs} (\la)$ denote the subset of {\em strict alcove-tableaux} for which all instances of $s_\emptyset$ occur as a prefix to the enhanced word.     
We set 
$$\Lambda_{\underline{h}}(n,{\bs }):=\{\w \mid \SSTP^\w \in \Std^+_{n,\bs} (\la) \text{ for some }\la \in \mathscr{P}_{\underline{h}}(n,{\bs })\}.$$

    \color{black}


\section{Categorification and  BGG resolutions}\label{sec:catBGG}

We now introduce the graded diagrammatic algebras which 
provide the necessary context for constructing our BGG resolutions.  

\subsection{The cyclotomic quiver Hecke algebras}

Given $\underline{i}=(i_1,\dots, i_n)\in (\ZZ/e\ZZ)^n$ and $s_r=(r,r+1)\in \mathfrak{S}_n$ we set 
 $s_r(\underline{i})= (i_1,\dots, i_{r-1}, i_{r+1}, i_r ,i_{r+2}, \dots ,i_n)$.  
 
\begin{defn}[\cite{MR2551762,MR2525917,ROUQ}]\label{defintino1}
Fix $\Bbbk$ an integral domain, $e > 2 $ and let ${\bs }\in \ZZ^\ell$.   
The {\em  cyclotomic quiver Hecke algebra} (or KLR algebra),  $\mathcal{H}_n(\bs ) $,   is defined to be the unital, associative $\Bbbk$-algebra with generators
$$ 
\{e_{\underline{i}}  \ | \ {\underline{i}}=(i_1,\dots ,i_n)\in   (\ZZ/e\ZZ)^n\}\cup\{y_1,\dots ,y_n\}\cup\{\psi_1,\dots ,\psi_{n-1}\},
$$
subject to the    relations 
\begin{align}
\tag{R1}\label{rel1}  e_{\underline{i}} e_{\underline{j}} =\delta_{{\underline{i}},{\underline{j}}} e_{\underline{i}} 
 \qquad \textstyle 
\sum_{{\underline{i}} \in   (\ZZ/e\ZZ )^n } e_{{\underline{i}}} =1_{\mathcal{H} _n}   
\qquad y_r		e_{{\underline{i}}}=e_{{\underline{i}}}y_r 
\qquad
\psi_r e_{{\underline{i}}} = e_{s_r({\underline{i}})} \psi_r 
\qquad
y_ry_s =y_sy_r
 \end{align}
for all $r,s,{\underline{i}},{\underline{j}}$ and 
\begin{align}
\tag{R2}\label{rel2}
\psi_ry	_s  = y_s\psi_r \  \text{ for } s\neq r,r+1&
\qquad  
&\psi_r\psi_s = \psi_s\psi_r \ \text{ for } |r-s|>1
 \\ \tag{R3}\label{rel3}
y_r \psi_r e_{\underline{i}}
  =
(\psi_r y_{r+1} \color{black}-\color{black} 
\delta_{i_r,i_{r+1}})e_{\underline{i}}  &
 \qquad 
&y_{r+1} \psi_r e_{\underline{i}}   =(\psi_r y_r \color{black}+\color{black} 
\delta_{i_r,i_{r+1}})e_{\underline{i}} 
\end{align}
\begin{align}
\tag{R4}\label{rel4}
\psi_r \psi_r  e_{\underline{i}} &=\begin{cases}
\mathrlap0\phantom{(\psi_{r+1}\psi_r\psi_{r+1} + 1)e_{\underline{i}} \qquad}& \text{if }i_r=i_{r+1},\\
e_{\underline{i}}  & \text{if }i_{r+1}\neq i_r, i_r\pm1,\\
(y_{r+1} - y_r) e_{\underline{i}}  & \text{if }i_{r+1}=i_r +1  ,\\
(y_r - y_{r+1}) e_{\underline{i}} & \text{if }i_{r+1}=i_r -1   
\end{cases}\\
\tag{R5}\label{rel5}
\psi_r \psi_{r+1} \psi_r e_{\underline{i}} &=\begin{cases}
(\psi_{r+1}\psi_r\psi_{r+1} - 1)e_{\underline{i}}\qquad& \text{if }i_r=i_{r+2}=i_{r+1}+1  ,\\
(\psi_{r+1}\psi_r\psi_{r+1} + 1)e_{\underline{i}}& \text{if }i_r=i_{r+2}=i_{r+1}-1    \\
 \psi_{r+1} \psi_r\psi_{r+1}e_{\underline{i}}   &\text{otherwise} 
\end{cases} 
\end{align}
 for all  permitted $r,s,i,j$ and finally, we set 
 \begin{align}
\label{rel1.12} y_1^{\sharp\{s_m | s_m= i_1 ,1\leq m \leq \ell 	\}} e_{\underline{i}} &=0 
\quad 
\text{ for    ${\underline{i}}\in (\ZZ/e\ZZ)^n$.}
\end{align}  
We let $\ast$ denote the anti-involution which fixes the generators.  
 
%
 \end{defn}

\begin{thm}[{\cite[Main Theorem]{MR2551762}}]\label{citemmmmmee}
 Let $\Bbbk$ be a field and let $ \xi\in \Bbbk$ be a primitive $e$th root of unity.   
 Set $q=\xi$ and 
 set $Q_m = \xi^{s_m}$ for $ 1\leq m \leq \ell$.  
 The algebras $\mathcal{H} _n(\bs )$ and  
 $\cyc_{q, Q_1, \dots, Q_{\ell}}(n)  $
 are isomorphic as $\Bbbk$-algebras.  
 \end{thm}

 \subsubsection{The quotient algebras of interest}
We set $e_\SSTT:= e_{{\rm res} (\SSTT)}\in \mathcal{H}_n(\bs )$.  
For $ \bla \in \mathscr{P}^\ell_n$, we set 
\begin{align}
\label{tableau not}
y_ \bla = \prod _{k=1}^n y_k^{ |{\mathcal A} _{\SSTT_{(m, \bla)}}(k)|}e_{\SSTT_{(m, \bla)}}
 \end{align}
 We remark that $y_ \bla = e_{\SSTT_{(m, \bla)}}$ for $ \bla \in 
  \mathscr{P}_{\aatchpair} (n)$.   
Given $\aatchpair \in \mathbb N^\ell$    we define  
   \begin{equation} \label{idempotent} {\sf y}_{\underline{ h }  }=
 \sum_{\begin{subarray}c
     1\leq m \leq \ell 
      \end{subarray}} y _{(\emptyset,\dots,\emptyset , (1^{h_m+1}) ,\emptyset, \dots,\emptyset)}\boxtimes 1_{\mathcal{H}_{n-h_m-1}(\bs )}.
\end{equation}  
Given $\SSTS, \SSTT \in \Std( \bla)$ and $\w$ any fixed  reduced word for $w^\SSTS_\SSTT$ we 
let $\psi^\SSTS_\SSTT:=e_\SSTS\psi_\w e_\SSTT$.  
 We have already seen that if $\bla$ is a calibrated $\ell$-partition, then 
$ h(\bla)$ is ${\bs }$-admissible.   

 \begin{defn}
Given ${\bs } \in \ZZ^\ell$, we let    $\underline{h}= (h_1,\dots,h_{\ell })  \in \mathbb N^\ell$
      be $\bs$-admissible.  
 We define $ \algebra:=  \mathcal{H}_n (\bs ) /\mathcal{H}_n (\bs ) {\sf y}_{h}\mathcal{H}_n (\bs )$. 
 \end{defn}

 \begin{thm} [\cite{cell4us}] \label{cellularstructure}
  For $\Bbbk$ an integral domain, the algebra
  $\algebra$ is free as a $\Bbbk$-module and has graded cellular basis 
$$
\{\psi_{\SSTS\SSTT} :=\psi^\SSTS_{\SSTP_ \bla}  
 \psi_\SSTT^{\SSTP_ \bla}\mid 
\SSTS, \SSTT \in \Path_{\aatchpair  }( \bla), \bla\in \mathscr{P}_{\underline{ h }}(n)\}
$$
with respect to the order $\rhd$ and the anti-involution $\ast$. 
That is, we have that 
  \begin{enumerate}[leftmargin=*]
    \item[$(1)$] Each    $\psi _{\SSTS\SSTT}$ is homogeneous
	of degree 
${\rm deg}
        (\psi _{\SSTS\SSTT})={\rm deg}(\SSTS)+{\rm deg}(\SSTT),$ for
        $ \bla \in\mathscr{P}_{\underline{ h }}(n)$ and 
      $\SSTS,\SSTT\in \Path_{\underline{h}}( \bla )$.
    \item[$(2)$] The set $\{\psi _{\SSTS\SSTT}\mid\SSTS,\SSTT\in \Path_{\underline{h}}( \bla ), \,
       \bla \in\mathscr{P}_{\underline{ h }}(n) \}$ is a  
      $\Bbbk$-basis of $\algebra$.
    \item[$(3)$]  If $\SSTS,\SSTT\in \Path_{\underline{h}}( \bla )$, for some
      $ \bla \in\mathscr{P}_{\underline{ h }}(n)$, and $a\in \algebra$ then 
    there exist scalars $r_{\SSTS\SSTU}(a)$, which do not depend on
    $\SSTT$, such that 
      \[a\psi _{\SSTS\SSTT}  =\sum_{\SSTU\in
      \Path_{\underline{h}}( \bla )}r_{\SSTS\SSTU}(a)\psi _{\SSTU\SSTT}\pmod 
      {\mathscr{H}     ^{\vartriangleright   \bla }},\]
      where $\mathscr{H}^{\vartriangleright   \bla } $ is the $ \Bbbk$-submodule of $\algebra$ spanned by
$\{\psi _{\SSTQ\SSTR}\mid\bmu \vartriangleright   \bla \text{ and }\SSTQ,\SSTR\in \Path_{\underline{h}}(\bmu )\}.$ 
    \item[$(4)$]  The $\Bbbk$-linear map $*:\algebra\to \algebra$ determined by
      $(\psi _{\SSTS\SSTT})^*=\psi _{\SSTT\SSTS}$, for all $ \bla \in\mathscr{P}_{\underline{ h }}(n)$ and
      all $\SSTS,\SSTT\in\Path_{\underline{h}}( \bla )$, is an anti-isomorphism of $\algebra$.
   \end{enumerate}   \end{thm}

\begin{defn}  Given   $ \bla \in\mathscr{P}_{\aatchpair}(n)$, the   {\em  Specht module} $S_\bs(\bla)$ is the graded left ${\algebra}$-module
  with basis
    $\{\psi _{\SSTS } \mid \SSTS\in \Path_{\underline{h}}( \bla  ) \}$.  
    The action of ${\algebra}$ on $S_\bs(\bla)$ is given by
    \[a \psi _{ \SSTS  }  =\textstyle \sum_{ \SSTU \in \Path_{\underline{h}}( \bla )}r_{\SSTS\SSTU}(a) \psi _{\SSTU},\]
    where the scalars $r_{\SSTS\SSTU}(a)$ are the scalars appearing in (3) of Theorem  \ref{cellularstructure}.  
 \end{defn}
 
 \begin{thm}
 For $\Bbbk$ a field, the algebra 
  $\algebra$  is quasi-hereditary with simple modules 
$$
D_\bs(\bla)= S_\bs(\bla) / {\rm rad}(S_\bs(\bla))
$$  
 for $ \bla \in \mathscr{P}_{\underline{h}}(n).  $
\end{thm}

\begin{rmk}
The  modules $S_\bs(\bla)$ are obtained (via specialisation) from the usual semisimple modules over $\mathbb{C}[q,Q_1,\dots, Q_\ell]$ with the same multipartition labels.  
 We remark that the integral form  (in the modular system by which we specialise)  is 
 constructed from the cylindric charge in \cite{manycell,cell4us} and can be seen as coming the quiver Cherednik algebra associated to $\bs \in \ZZ^\ell$.      
\end{rmk}

\begin{prop}
Let $\bla\in \mathscr{P}_\ell(n)$ with $h( \bla)= \aatchpair \in \mathbb{N}^\ell$.  
Then    $\bla\in \Cali^\bs(\ell)$
 if and only if    $\bla  \in \mathcal{F}_{\aatchpair}(n)$ the fundamental alcove of $\mathbb E_\aatchpair$.  

\end{prop}

\begin{proof}
Any $\bla\in \Cali^\bs(\ell)$ must satisfy  that $h(\bla)=\aatchpair$ is $\bs$-admissible (by Corollary \ref{admiss}) and so it is enough to restrict our attention to  
 $\bla\in {\mathbb E}_\aatchpair$ for some 
 $\bs$-admissible $\aatchpair\in \mathbb N^\ell$.  
 For  $\bla\in {\mathbb E}_\aatchpair$ 
the condition that the border strip is increasing is equivalent to the condition that $\bla\in   {\mathbb E}^{<}(\varepsilon_i-\varepsilon_{i+1},e)$ for $1\leq i <\ell$.  
Similarly, the condition that $\bla$ has   period at most $e$ is equivalent to the condition that $\bla \in   {\mathbb E}^{\less}(\varepsilon_1-\varepsilon_h,-e)$.  The result follows.  
 \end{proof}

   \subsection{{Diagrammatic Bott--Samelson} algebras}   \label{soergel}
These algebras were first defined   in \cite{MR3555156}.

\begin{defn}
Given  $\al \in \Pi$  we define the corresponding Soergel idempotent, 
${\sf 1}_{ {\al}}$,   to be a frame of width 1 unit, containing a single vertical strand   coloured with $\al  \in \Pi  $.   
We define ${\sf 1}_\emptyset$ to be an empty frame of width 1 unit.   
For $\w =s_{\alpha^{(1)}}\mydots s_{\alpha^{(p)}}$  with ${\alpha^{(i)}}\in S\cup\{1\}$ for $1\leq i \leq p$, we set 
  $${\sf 1}_\w= {\sf 1}_{\alpha^{(1)}}\otimes 
  {\sf 1}_{\alpha^{(2)}}\otimes \dots 
\otimes    {\sf 1}_{\alpha^{(p)}} $$
to be the diagram obtained by horizontal concatenation.  \end{defn}

\begin{defn} 
Let $\csigma,\ctau,\crho \in S$ with $m(\csigma,\ctau)=3$  and $m(\ctau,\crho)=2$.  
Given two words $\w $ and $\w'$ in the alphabet $S\cup\{1\}$ we define 
 a $(\w,\w')$-Soergel diagram  $D$ 
is defined to be any diagram obtained by horizontal and vertical concatenation (denoted $\otimes $ and $\circ$ respectively) of the 
following diagrams 
\begin{align}\label{pictures}
 \begin{minipage}{1cm}\begin{tikzpicture}[scale=1.2]
\draw[densely dotted,rounded corners](-0.5cm,-0.5cm)  rectangle (0.5cm,0.5cm);
\clip(0,0) circle (0.5cm);
 \end{tikzpicture}\end{minipage}
\qquad \; 
 \begin{minipage}{1cm}\begin{tikzpicture}[scale=1.2]
\draw[densely dotted,rounded corners](-0.5cm,-0.5cm)  rectangle (0.5cm,0.5cm);
\clip(0,0) circle (0.5cm);
\draw[line width=0.08cm, magenta](0,-1)--(0,+1);
\end{tikzpicture}\end{minipage}
\qquad \; 
\begin{minipage}{1cm}\begin{tikzpicture}[scale=1.2]
\draw[densely dotted,rounded corners](-0.5cm,-0.5cm)  rectangle (0.5cm,0.5cm);
 \clip(0,0) circle (0.5cm);
\draw[line width=0.08cm, magenta](0,-1)--(0,+0);
\fill[magenta] (0,0)  circle (5pt);
\end{tikzpicture}\end{minipage}
 \qquad \; 
\begin{minipage}{1cm}\begin{tikzpicture}[scale=1.2]
\draw[densely dotted,rounded corners](-0.5cm,-0.5cm)
  rectangle (0.5cm,0.5cm);
 \clip (-0.5cm,-0.5cm)
  rectangle (0.5cm,0.5cm);
 \draw[line width=0.08cm, magenta](0,0)to [out=-30, in =90] (10pt,-15pt);
 \draw[line width=0.08cm, magenta](0,0)to [out=-150, in =90] (-10pt,-15pt);
 \draw[line width=0.08cm, magenta](0,0)--++(90:1);
 \end{tikzpicture}\end{minipage}
\qquad \; 
\begin{minipage}{1.2cm}\begin{tikzpicture} [xscale=1.4,yscale=1.2]
\draw[densely dotted,rounded corners](-0.5cm,-0.5cm)
  rectangle (0.5cm,0.5cm);
  \clip(-0.5cm,-0.5cm)
  rectangle (0.5cm,0.5cm);
 \draw[line width=0.08cm, magenta](0,0)to [out=30, in =-90] (10pt,0.5cm);
 \draw[line width=0.08cm, magenta](0,0)to [out=150, in =-90] (-10pt,0.5cm);
 \draw[line width=0.08cm, magenta](0,0)--++(-90:1);
  \draw[line width=0.08cm, cyan](0,0)to [out=-30, in =90] (10pt,-0.5cm);
 \draw[line width=0.08cm, cyan](0,0)to [out=-150, in =90] (-10pt,-0.5cm);
 \draw[line width=0.08cm, cyan](0,0)--++(90:1);
 \end{tikzpicture}\end{minipage}\qquad \;  
\begin{minipage}{1.1cm}\begin{tikzpicture}[xscale=1.3,yscale=1.2]
\draw[densely dotted,rounded corners](-0.5cm,-0.5cm)  rectangle (0.5cm,0.5cm);
\clip(-0.5cm,-0.5cm)  rectangle (0.5cm,0.5cm);
 \draw[line width=0.08cm, darkgreen] (0,0) to [out=45, in =-90] (10pt,0.5cm);
 \draw[line width=0.08cm, darkgreen] (0,0) to [out=-135, in =90] (-10pt,-0.5cm);
    \draw[line width=0.08cm, cyan] (0,0) to [out=-45, in =90] (10pt,-0.5cm);
 \draw[line width=0.08cm, cyan] (0,0) to [out=135, in =-90] (-10pt,0.5cm);
    \end{tikzpicture}\end{minipage}
\end{align}
which we often denote these diagrams by
$$
{\sf 1}_\emptyset \quad 
{\sf 1}_\csigma \quad 
{\sf spot}^{{\color{magenta}\emptyset}}_\csigma \quad 
{\sf fork}_{\csigma\csigma}^{\csigma } \quad 
{\sf braid}^{\csigma\ctau\csigma}_{\ctau\csigma\ctau}\quad 
{\sf braid}^{{\ctau\color{darkgreen}\gamma}}_{ \color{darkgreen}\gamma\ctau}
$$ respectively along with 
their   flips through the horizontal axis and their isotypic deformations  such that the     north     and     south  edges  of the graph are given by the idempotents 
${\sf 1}_{\underline{w}}$ and  
${\sf 1}_{\underline{w}'}$ respectively.  
Here the vertical concatenation of a $(\w,\w')$-Soergel diagram on top of a 
$(\underline{v},\underline{v}')$-Soergel diagram is zero if $\underline{v}\neq \w'$.  
We define the degree of these generators (and their   flips) to be $0, 1, -1, 0$, and $0$ respectively.  We let $\ast$ denote the  map  which   flips a diagram   through its  horizontal axis. 
     \end{defn}

 Suppose that  $\w$ and $\w'$ are both words with the same underlying permutation and that they can be obtained from one another by a sequence of applications of the braid relations of $\widehat{\mathfrak{S}}_h$ (i.e. without applying the quadratic relation); 
 we let ${\sf braid}^{\w}_{\w'}$ (or  ${\sf braid}^{\SSTP_{\w}}_{\SSTP_{\w'}}$ if we wish to emphasise the corresponding paths) denote the product of the corresponding sequence  of the braid generator Soergel diagrams (the final two pictures in \ref{pictures}).

   \begin{defn}   \newcommand{\vvv}{{\underline{w} }} 
\renewcommand{\w}{{\underline{x}}}
\renewcommand{\x}{{\underline{y}}}
\renewcommand{\y}{{\underline{z}}}
Let  $\Bbbk$ be an integral domain.  We define the 
  {diagrammatic Bott--Samelson} endomorphism algebra, 
  $ \mathscr{S}_\aatchpair  ({n,{\bs }}),  $ 
  to be the   locally-unital  associative $\Bbbk$-algebra spanned by all $(\w,\w')$-Soergel diagrams 
   for $\w,\w' \in \Lambda_{\aatchpair}(n,{\bs })$   with multiplication given by vertical concatenation of diagrams modulo the following local relations and their duals.  We have the idempotent relations,
\begin{align*}
{\sf 1}_{\csigma} {\sf 1}_{\ctau}      =\delta_{\csigma,\ctau}{\sf 1}_{\csigma}    
\qquad   {\sf 1}_{\emptyset} {\sf 1}_{\csigma}       =0   
\qquad    {\sf 1}_{\emptyset}^2      ={\sf 1}_{\emptyset}   
\qquad    {\sf 1}_{\emptyset} {\sf spot}_{\csigma}^\emptyset {\sf 1}_{\csigma}      ={\sf spot}_{\csigma}^{\emptyset} 
\end{align*}
 \begin{align*}
  {\sf 1}_{\csigma} {\sf fork}_{\csigma\csigma}^{\csigma} {\sf 1}_{\csigma\csigma}      ={\sf fork}_{\csigma\csigma}^{\csigma}      
  \qquad
{\sf 1}_{\ctau\csigma\ctau}  {\sf braid}_{\csigma\ctau\csigma}^{\ctau\csigma\ctau} {\sf 1}_{\csigma\ctau\csigma}       = {\sf braid}_{\csigma\ctau\csigma}^{\ctau\csigma\ctau} 
\qquad
{\sf 1}_{\ctau \crho}  {\sf braid}_{\crho \ctau}^{\ctau\crho} {\sf 1}_{\ctau\crho}       =  {\sf braid}_{\crho \ctau}^{\ctau\crho}
\end{align*}
For each  $\csigma \in S $  we have  monochrome relations 
 $$
({\sf spot}_\csigma^\emptyset \otimes {\sf 1}_\csigma){\sf fork}^{\csigma\csigma}_{\csigma}
=
{\sf 1}_{\csigma}
 \qquad\quad  
  ({\sf 1}_\csigma\otimes {\sf fork}_{\csigma\csigma}^{ \csigma} )
({\sf fork}^{\csigma\csigma}_{\csigma}\otimes {\sf 1}_{\csigma})
=
{\sf fork}^{\csigma\csigma}_{\csigma}
{\sf fork}^{\csigma}_{ \csigma\csigma}
 $$
 $${\sf fork}_{\csigma\csigma}^{\csigma}
{\sf fork}^{\csigma\csigma}_{\csigma}=0 \qquad \quad ({\sf spot}_\csigma^\emptyset {\sf spot}^\csigma_\emptyset)\otimes {\sf 1}_{\csigma}
+
 {\sf 1}_{\csigma} \otimes ({\sf spot}_\csigma^\emptyset {\sf spot}^\csigma_\emptyset)
 =
2 ({\sf spot}^\csigma_\emptyset {\sf spot}_\csigma^\emptyset)
$$ For $m(\csigma,\ctau)=3$ and  $m( \ctau,\crho)=2$  we have the  two-colour barbell relations 
\begin{align*}
({\sf spot}_\ctau^\emptyset {\sf spot}^\ctau_\emptyset)\otimes {\sf 1}_{\csigma}
-
 {\sf 1}_{\csigma} \otimes ({\sf spot}_\ctau^\emptyset {\sf spot}^\ctau_\emptyset)
& =   
(1+\delta_{h,2})((  {\sf spot}^\csigma_{\emptyset} {\sf spot}_\csigma^\emptyset ) -
 {\sf 1}_{\csigma} \otimes ({\sf spot}_\csigma^\emptyset {\sf spot}^\csigma_\emptyset))
\\
({\sf spot}_\ctau^\emptyset {\sf spot}^\ctau_\emptyset)\otimes {\sf 1}_{\crho}
-
 {\sf 1}_{\crho} \otimes ({\sf spot}_\ctau^\emptyset {\sf spot}^\ctau_\emptyset)
& =   
0 \end{align*} 
and   the fork-braid relations 
$$
 {\sf braid}_{\csigma \ctau \csigma}^{\ctau\csigma \ctau} 
({\sf fork}^{\csigma}_{\csigma\csigma}\otimes{\sf 1}_{ \ctau\csigma}  	)
 ( {\sf 1}_\csigma \otimes{\sf braid}^{  \csigma \ctau  \csigma   }_
 {  \ctau \csigma\ctau} )
=
   ({\sf 1}_{ \ctau\csigma}   \otimes {\sf fork}_{\ctau\ctau}^{ \ctau})
( 
 {\sf braid}_{ \csigma \ctau  \csigma}^{  \ctau\csigma \ctau} 
\otimes {\sf 1}_\ctau) 
$$
 $$
 {\sf braid}_{\crho\ctau }^{\ctau\crho } 
({\sf fork}^{\crho}_{\crho\crho}\otimes{\sf 1}_{\ctau }  	)
 ( {\sf 1}_\crho \otimes{\sf braid}^{ \crho\ctau }_
 {  \ctau\crho} )
=
   ({\sf 1}_{ \ctau }    \otimes {\sf fork}_{\crho\crho}^{\crho})
( 
  {\sf braid}_{  \crho \ctau } ^{    \ctau \crho} 
\otimes {\sf 1}_\crho) 
$$ 
  the cyclicity  relation,  
\begin{align*}
({\sf 1}_{ \ctau\csigma\ctau} \otimes ({\sf spot}_{\csigma}^\emptyset {\sf fork}^{\csigma}_{\csigma\csigma}))({\sf 1}_\ctau \otimes
 {\sf braid}^{\csigma\ctau\csigma}_{ \ctau\csigma\ctau} 
\otimes {\sf 1}_\csigma) 
( (  {\sf fork}_{\ctau}^{\ctau\ctau} {\sf spot}^{\ctau}_\emptyset)\otimes 
{\sf 1}_{ \csigma \ctau\csigma} )
&=
  {\sf braid}^{\ctau \csigma  \ctau}_{\csigma \ctau \csigma}  
 \end{align*}
and the Jones--Wenzl relations 
\begin{align*} 
({\sf spot}_\ctau^{{{ \emptyset}} } \otimes {\sf 1} _{ \csigma\ctau} ) {\sf braid}_{\ctau\csigma\ctau}^{  \csigma\ctau\csigma}
&=
( {\sf fork}^\csigma_{\csigma\csigma}\otimes {\sf spot}^\ctau_{{{ \emptyset}}} )
({\sf 1}_\csigma \otimes {\sf spot}^
{{ \emptyset}}_\ctau \otimes {\sf 1}_\csigma)
\\
({\sf 1} _{\ctau } \otimes {\sf spot}_\crho^{ 
\emptyset}) {\sf braid}^{\ctau\crho}_{  \crho\ctau} 
&=
 ({\sf spot}^\emptyset_\crho \otimes {\sf 1}_{ \ctau   } )
 \intertext{ 
For   $m_{\csigma \crho}=
m_{  \crho\ctau}=m_{ \ctau \color{orange}\delta}=2$ and 
$m_{\csigma \ctau}=3$,  we have the braid-commutativity relations}
 ({\sf braid}_{\csigma\ctau\csigma}^{\ctau\csigma\ctau} \otimes {\sf 1}_
\crho)
{\sf braid}^{\csigma \ctau \csigma  \crho}_{\crho\csigma \ctau   \csigma} 
&=
{\sf braid}^{\ctau \csigma  \ctau \crho}_
{\crho\ctau \csigma  \ctau} 
({\sf 1}_
\crho\otimes {\sf braid}_{\csigma\ctau\csigma}^{\ctau\csigma\ctau}  )
\\
({\sf braid}_{\crho\ctau}^{\ctau\crho} \otimes {\sf 1}_
{\color{orange}\delta})
{\sf braid}^{       \crho \ctau {\color{orange}\delta}}_{{\color{orange}\delta}  \ctau   \crho} 
&=
{\sf braid}^{\ctau  \crho {\color{orange}\delta}}_
{{\color{orange}\delta}  \ctau \crho} 
({\sf 1}_
{\color{orange}\delta}
\otimes {\sf braid}^{ \ctau\crho}_{\crho\ctau}  ).
\end{align*}
    For 
     $m_{\al\bet}=3=m_{\al\gam}$ and $m_{\al\gam}=2$ we have   the Zamolodchikov relation
  \begin{align*}
  &{\sf braid}^{ \gam\al\gam \bet\al\gam}_{ \al\gam\al  \bet\al\gam    }    
   {\sf braid}^{ \al\gam  \al\bet\al \gam    }_{  \al\gam \bet\al\bet    \gam   }    
   {\sf braid}^{ \al\gam\bet \al \bet\gam  }_{ \al\bet\gam \al  \gam\bet   }  
   {\sf braid}^{   \al\bet	\gam\al\gam \bet   }_{ \al\bet \al\gam\al   \bet }   
 {\sf braid}^{  \al\bet	 \al  \gam\al \bet  } _{  \bet \al\bet  \gam\al \bet    }   
{\sf braid}^{\bet\al \bet\gam\al\bet     }_{ \bet\al \gam\bet \al\bet     } 
\\[5pt] = \;\;\;&  	{\sf braid}_{\gam\al\bet\gam \al\gam  }^{ \gam\al \gam\bet\al\gam  }  
  {\sf braid}^{\gam\al\bet \gam\al \gam }_{\gam\al\bet\al\gam\al} 
   {\sf braid}^{\gam\al\bet\al\gam\al}_{\gam\bet\al\bet\gam\al}
 {\sf braid}^{\gam\bet\al\bet\gam\al}_{\bet\gam\al\gam\bet\al}   
 {\sf braid} ^{\bet\gam\al\gam\bet\al} 
 _{\bet \al\gam\al\bet\al} 
  {\sf braid}^{\bet\al\gam\al\bet\al}_{\bet\al\gam\bet\al\bet} .
\end{align*}
For all diagrams ${\sf D}_1,{\sf D}_2,{\sf D}_3,{\sf D}_4$ and all enhanced words $\w,\x$,  we require the bifunctoriality relation
$$
 \big(({\sf D  }_1\circ	{\sf 1}_{ {\w}}   )\otimes  ({\sf D}_2 \circ {\sf 1}_{ {\x}}) \big)
\big(({\sf 1}_{\w} \circ {\sf D}_3)  \otimes ({\sf 1}_\x \circ{\sf D }_4)\big)
=
 ({\sf D}_1 \circ {\sf 1}_{{\w}} \circ {\sf D_3}) \otimes ({\sf D}_2 \circ {\sf 1}_{{\x}} \circ {\sf D}_4)
$$
and  the monoidal unit relation 
$$
{\sf 1}_{\emptyset} \otimes {\sf D}_1={\sf D}_1={\sf D}_1 \otimes {\sf 1}_{\emptyset}.
$$ Finally, we require the (non-local) cyclotomic relation
$$
{\sf spot}^\emptyset_\al {\sf spot}^\al_\emptyset \otimes {\sf 1}_\vvv =0 \qquad \text{for  all  $\vvv\in {\sf exp}(w)$, $w\in \widehat{\mathfrak{S}}_h,$  and all $s_\al \in  {\mathfrak{S}}_\aatchpair$.  }
$$

 \end{defn}

\subsubsection{Cellularity and quasi-hereditary structure}
  We can extend an alcove-tableau    $\SSTQ'\in\Std_{n,\bs}( \bla)$   to obtain a new alcove-tableau $\SSTQ$ in one of three possible ways 
$$
\SSTQ=\SSTQ'\otimes \SSTP_\al 	\qquad 	\SSTQ=\SSTQ'\otimes  \reflectpath   		\qquad \SSTQ=\SSTQ'\otimes \SSTP_\emptyset 
$$
for some  $\al \in \Pi$. 
  The first two cases each subdivide into a further two cases based on whether $\al$ is an upper or lower wall of the alcove containing $ \bla$.  

\begin{defn} \label{twins} 
Suppose that $ \bla $ belongs to an alcove which has a hyperplane labelled by $\al$ as an   upper alcove wall.  
Let 
   $\SSTQ'\in\Std _{n,\bs}( \bla)$. 
 If $\SSTQ= \SSTQ' \otimes \SSTP_\al$ then we set $\deg(\SSTQ)=\deg(\SSTQ')$  and  we define 
\begin{align*} 
 c_{\SSTQ}^\SSTP= 
 {\sf braid}_{\SSTP'\otimes \SSTP_\al}^{\SSTP}(
c_{\SSTQ'}^{\SSTP'} \;\otimes \; {\sf 1}_{  \al})  
\qquad
 \intertext{
If $\SSTQ= \SSTQ' \otimes  \reflectpath  $ then
we set  $\deg(\SSTQ)=\deg(\SSTQ')+1$ and we define 
} 
c_{\SSTQ}^\SSTP= 
 {\sf braid}_{\SSTP'\otimes \SSTP_\emp}^{\SSTP}(
 c_{\SSTQ'}^{\SSTP'} \;\otimes \; {\sf  spot}_{\al}^\emp) .
\end{align*}
Now suppose that $ \bla $ belongs to an alcove which has a hyperplane labelled by $\al$ as a lower alcove wall. 
Thus we can choose $ \SSTP_{\underline{v}} \otimes \SSTP_\al=\SSTP'\in\Std _{n,\bs}( \bla) $.  
For 
$\SSTQ= \SSTQ' \otimes \SSTP_\al$, 
we set  $\deg(\SSTQ)=\deg(\SSTQ') $   
 and   define 
  \begin{align*}
c_{\SSTQ}^\SSTP= &
{\sf braid}^{		
  \SSTP  }_{\SSTP_{\underline{v}  \emp\emp}}
\big( 
{\sf 1}_{ {\underline{v}}}\otimes 
({\sf  spot}_\al^\emp \circ{\sf  fork}_{\al\al}^{\al})\big) 
\big(
  c^{\SSTP'}
_{\SSTQ'} 
\otimes {\sf1}_{ \al}\big) 
\end{align*}
  and if $\SSTQ= \SSTQ' \otimes  \reflectpath  $ then 
  then
we set  $\deg(\SSTQ)=\deg(\SSTQ')-1$ and we define   
\begin{align*}
c_{\SSTQ}^\SSTP=& 
{\sf braid}^{		
  \SSTP  }_{  \SSTP_{\underline{v}\al \emp}}
\big( 
{\sf 1}_{ {\underline{v}}}\otimes {\sf  fork}_{\al\al}^{\al}\big) 
\big(
  c^{\SSTP'}
_{\SSTQ'} 
\otimes {\sf 1}_{ \al}\big)  .  
 \end{align*}
  \end{defn}

\begin{thm} 
[{\cite[Section 6.4]{MR3555156}}]  
   \label{LEW1} 
 For each $ \bla\in \mathscr{P}_{\aatchpair}{ ({n,{\bs }})}$, we fix an arbitrary 
reduced path   $
 \SSTP_{ \bla} \in \Std^+_{n,\bs}( \bla).
 $ 
 The algebra 
  $ \mathscr{S} _{\aatchpair}(n,\bs)$ is   quasi-hereditary  with  graded  integral cellular basis  
 $$
 \{ c^{\SSTS}_{ \SSTP_{ \bla}} 
  c^{ \SSTP_{ \bla}}_\SSTT
 \mid
  \SSTS,\SSTT \in \Std^+_{n,\bs}  ( \bla),  \bla \in  \mathscr{P}_{\aatchpair}{ ({n,{\bs }})} \}  
 $$
 with respect to the Bruhat ordering on 
$\Lambda_{\aatchpair}{ ({n,{\bs }})}$ and anti-involution $\ast$.  
That is, we have that 
  \begin{enumerate}[leftmargin=*]
    \item[$(1)$] Each    $c _{\SSTS\SSTT}$ is homogeneous
	of degree 
${\rm deg}
        (c _{\SSTS\SSTT})={\rm deg}(\SSTS)+{\rm deg}(\SSTT),$ for
        $ \bla \in\mathscr{P}_{\aatchpair}{ ({n,{\bs }})}$ and 
      $\SSTS,\SSTT\in  \Std^+_{n,\bs}  ( \bla)$.
    \item[$(2)$] The set $\{c _{\SSTS\SSTT}\mid\SSTS,\SSTT\in  \Std^+_{n,\bs}  ( \bla), \,
       \bla \in\mathscr{P}_{\aatchpair}{ ({n,{\bs }})} \}$ is a  
      $\Bbbk$-basis of $\mathscr{S}_\aatchpair  ({n,{\bs }})$.
    \item[$(3)$]  If $\SSTS,\SSTT\in  \Std^+_{n,\bs}  ( \bla)$, for some
      $ \bla \in\mathscr{P}_{\underline{ h }}(n)$, and $a\in \mathscr{S}_\aatchpair (n,\bs)$ then 
    there exist scalars $r_{\SSTS\SSTU}(a)$, which do not depend on
    $\SSTT$, such that 
      \[ac _{\SSTS\SSTT}  =\sum_{\SSTU\in
\Std^+_{n,\bs}( \bla)}r_{\SSTS\SSTU}(a)c _{\SSTU\SSTT}\pmod 
      {\mathscr{S}_\aatchpair      ^{\vartriangleright   \bla }},\]
      where $\mathscr{S}_\aatchpair ^{\vartriangleright   \bla} $ is the $ \Bbbk$-submodule of $\mathscr{S}_\aatchpair   (n,\bs )$ spanned by
$\{c _{\SSTQ\SSTR}\mid    \SSTQ,\SSTR\in   \Std^+_{n,\bs}  (\bmu) \text { for }\bmu\rhd  \bla  \}.$ 
    \item[$(4)$]  The $\Bbbk$-linear map $*:\mathscr{S}_\aatchpair  ({n,{\bs }})\to \mathscr{S}_\aatchpair  ({n,{\bs }})$ determined by
      $(c _{\SSTS\SSTT})^*=c _{\SSTT\SSTS}$, for all $ \bla \in\mathscr{P}_{\aatchpair}{ ({n,{\bs }})}$ and
      all $\SSTS,\SSTT\in  \Std^+_{n,\bs}  ( \bla)$, is an anti-isomorphism of $\mathscr{S}_\aatchpair  ({n,{\bs }})$.
   \end{enumerate} 
\end{thm}

\begin{defn}  Given   $ \bla \in\mathscr{P}_\aatchpair{(n,{\bs })}$, the   {\sf standard module} $\Delta( \bla)$ is the graded left $\mathscr{S}_\aatchpair  ({n,{\bs }})$-module
  with basis
    $\{c _{\SSTS }  \mid \SSTS\in   \Std^+_{n,\bs}  ( \bla)    \}$.  
    The action of ${\mathscr{S}_\aatchpair (n,\bs )}$ on $\Delta( \bla)$ is given by
    \[a c _{ \SSTS  }  =\textstyle \sum_{ \SSTU \in   \Std^+_{n,\bs}  ( \bla)   }r_{\SSTS\SSTU}(a) c_{\SSTU},\]
    where the scalars $r_{\SSTS\SSTU}(a)$ are the scalars appearing in (3) of Theorem \ref{LEW1}.  
\end{defn}

 \begin{thm}
 For $\Bbbk$ a field, the algebra 
  $ \mathscr{S}_\aatchpair (n,\bs)$  is quasi-hereditary with simple modules 
$$
L (\bla)= \Delta (\bla) / {\rm rad}(\Delta (\bla))
$$  
 for $ \bla \in \mathscr{P}_{\underline{h}}(n,\bs).  $
\end{thm}

\begin{rmk}
We remark that the algebra $\mathscr{S}_\aatchpair (n, \bs)$ together with its basis given in Theorem \ref{LEW1} is an example of a \emph{based quasi-hereditary algebra} with involution $*$, see \cite{brundanstroppel}. 
\end{rmk}

\subsection{The isomorphism and BGG resolutions}
We  are now ready to restrict our attention to regular blocks of $\algebra$.   
   Given $\al$ a simple reflection  or $\al=\emptyset$, 
 we have an 
 associated  path 
     $\SSTP_{ \color{magenta}\alpha} $, a trivial bijection $w^{\SSTP_\al} _{\SSTP_\al} =1\in \mathfrak{S}_{b_{ \color{magenta}\alpha} \aatch  }$,  and  an idempotent element of the quiver Hecke algebra
 $$ 
e_{\SSTP_{ \color{magenta}\alpha} }   := e_{\res(\SSTP_\al)}\in \mathscr{H}_\aatchpair 
({b_{ \color{magenta}\alpha} \aatch  },\bs ) .  
 $$
 More generally, given any $\w=s_{ \alpha^{(1)}}s_{ \alpha^{(2)}}\dots  s_{ \alpha^{(k)} }$ with $\bla \in \mathcal{F}_\aatchpair \w$, we have an 
 associated  path 
     $\SSTP_\w $, and  an element of the quiver Hecke algebra
 $$ 
e_{\SSTP_\w }   := e_{\res(\SSTP_\w)}=
e_{\SSTP_{\alpha^{(1)}}}
\otimes
e_{\SSTP_{\alpha^{(2)}}}	
\otimes \dots\otimes e_{\SSTP_{\alpha^{(k)}}}	 
 $$    
    and we define 
 \begin{equation} \label{redue}
    {\sf f}_{n,\bs}^+=\sum_{\begin{subarray}c
   \sts\in \Std^+_{\underline{h}}( \bla)
   \\
    \bla  \in \mathscr{P}_{\underline{h}}(n,{\bs } )
   \end{subarray}} { e}_{\SSTS}.  
 \end{equation}
  
\begin{thm}[{\cite[Theorems  A and B]{cell4us2}}]
Let  ${\bs }\in \ZZ^\ell $ and $\Bbbk$ be an arbitrary integral domain.  Let   $e  > h  $  and suppose that $\aatchpair \in \mathbb N^\ell$ is 
 ${\bs }$-admissible.  
We have an isomorphism of  graded $\Bbbk$-algebras, 
$$
    {\sf f}_{n,\bs}^+\left(\mathcal{H}_n(\bs )/\mathcal{H}_n(\bs ) {\sf y}_{\aatchpair }\mathcal{H}_n(\bs ) 		\right)     {\sf f}_{n,\bs}^+
\cong \mathscr{S}_\aatchpair (n,{\bs })
 $$
   Moreover, 
the  isomorphism preserves the graded cellular structures of these algebras, that is,   
  ${\sf f}^+_{n,\bs} S_\bs( \bla  ) \cong 
 \Delta(  \bla).$ 
   \end{thm}

We now recall the construction of the BGG resolutions  for  certain $\mathscr{S}_\aatchpair (n,{\bs })$-modules.  
In what follows, we let 
$\boldsymbol \bla \in \mathcal{F}_\aatchpair (n) \cdot w$, 
$\boldsymbol\mu \in \mathcal{F}_\aatchpair (n) \cdot x$, 
$\boldsymbol\nu \in \mathcal{F}_\aatchpair (n) \cdot y$, 
and 
$\boldsymbol\xi \in \mathcal{F}_\aatchpair (n) \cdot z$ for 
$w,x,y,z \in \mathfrak{S}^\aatchpair$.   
For the remainder of the paper we label       
$\mathscr{S}_\aatchpair (n,{\bs })$-modules by the cosets 
$w,x,y,z \in \mathfrak{S}^\aatchpair$, rather than by the multipartitions as this is more convenient for   indexing the homomorphisms in our resolutions.    

\begin{defn}
Given $w, y\in {\mathfrak{S}^\aatchpair }$, we say that   
 $( w, y)$ is a {\sf Carter--Payne pair} if $y\leq w$ and $\ell(y)=\ell(w)-1$.   
 Let     $(w, x)$ and $(x, {z})$  be Carter--Payne pairs.  
If there exists a (necessarily unique) $y\in {\mathfrak{S}^\aatchpair }$ such that $w\geq y\geq z$, then we refer to the quadruple $w,x,y,z$ as a {\sf diamond}.  If no such $y$ exists, we refer to the triple $w,x,z$ as a {\sf strand}.  

   \end{defn}

\begin{thm}We let $\boldsymbol \bla \in \mathcal{F}_\aatchpair (n) \cdot w$, 
$\boldsymbol\mu \in \mathcal{F}_\aatchpair (n) \cdot x$   and suppose that   $(w,x) $ is a  Carter--Payne pair.   Pick an arbitrary $\w=\sigma_1\dots \sigma_\ell$ and 
suppose that  $\x=\sigma_1\dots \sigma_{p-1}\widehat{\sigma}_p \sigma_{p+1}\dots \sigma_\ell$ is the subexpression  for  $x$ obtained by deleting precisely one element $\sigma_p \in S$.    We have that 
$$ 
{\rm Hom}_ {\mathscr{S}_\aatchpair  ({n,{\bs }})}(\Delta( {\w }),\Delta( {\x}))
$$ 
is $t ^1$-dimensional and spanned by the map 
\begin{align}\label{CP}
\varphi^\w_{\x} (c_\stt ) = 
 c_\stt 
 ({\sf1}_ {\sigma_1\cdots \sigma_{p-1}} \otimes 
 {\sf spot}^{ {\sigma}_p} _\emptyset 
   \otimes {\sf 1}_{ \sigma_{p+1}\cdots \sigma_\ell}   )
\end{align}
for $\stt \in \Std_{n,\bs}^+(\bla)$.  
\end{thm}   

 We    define a complex
of graded $\mathscr{S}_\aatchpair  ({n,{\bs }})$-modules
$
\cdots
\longrightarrow
\Delta_2 \stackrel{\delta_2}{\longrightarrow}
\Delta_1 \stackrel{\delta_1}{\longrightarrow}
\Delta_0 \stackrel{\delta_0}{\longrightarrow}
  0,
$ 
where
\begin{equation} \label{modi}
\Delta_\ell := \bigoplus_{
\begin{subarray}c
 \ell(\w)=\ell \end{subarray}} 
\Delta(\w) \langle  \ell(\w) \rangle.
\end{equation}
We will refer to this as the {\sf BGG complex}. 
 For  $w, x,y,z  $   a diamond we have homomorphisms  of $\mathscr{S}_\aatchpair  ({n,{\bs }})$-modules
$$
\begin{tikzpicture}
 \path(0,0)--++(90:30pt) coordinate(N);
 \path(0,0)--++(-90:30pt) coordinate(S);
 \path(0,0)--++(0:70pt) coordinate(E);
  \path(0,0)--++(180:70pt) coordinate(W);
\draw(N) node {$\Delta(\x)$};
\draw(S) node {$\Delta(\y)$};
\draw(W) node {$\Delta(\w)$};
\draw(E) node {$\Delta(\underline{z})$};
\draw[->](173:70pt)to node[midway,above]  {$\varphi^\w_{\x}$} (120:30pt)  ; 
\draw[->](-173:70pt)to node[midway,below]  {$\varphi^\w_{\y}$} (-120:30pt)  ; 
\draw[<-](7:70pt)to node[midway,above]  {$\varphi^\x_{\underline{z}}$} (60:30pt)  ; 
\draw[<-](-7:70pt)to node[midway,below]  {$\varphi^\y_{\underline{z}}$} (-60:30pt)  ;   \end{tikzpicture}
$$ 
given by our Carter--Payne homomorphisms of \ref{CP}.  
By an easy variation on
\cite[Lemma 10.4]{bgg}, it is possible to pick
a sign $\epsilon(\alpha,\beta)$ for each of the four Carter--Payne pairs  
such that for every diamond
the product of the signs associated to its four arrows
is equal to $-1$.
    For every strand $w,x,z$   have homomorphisms 
  $$
\begin{tikzpicture}
 \path(0,0)--++(90:0pt) coordinate(N);
 \path(0,0)--++(0:70pt) coordinate(E);
  \path(0,0)--++(180:70pt) coordinate(W);
\draw(N) node {$\Delta(\x)$};
\draw(W) node {$\Delta(\w)$};
\draw(E) node {$\Delta(\underline{z})$.};
\draw[->](180:55pt)to node[midway,above]  {$\varphi^\w_{\x}$} (180:15pt)  ; 
 \draw[<-](0:55pt)to node[midway,above]  {$\varphi^\x_{\underline{z}}$} (0:15pt)  ; 
 \end{tikzpicture} $$ 
    We can now define the $\mathscr{S}_\aatchpair  ({n,{\bs }})$-differential
$\delta_\ell:\Delta_{\ell } \rightarrow \Delta_{\ell-1}$ for $\ell\geq 1$ 
to be the sum of the maps
$$
\epsilon(\alpha,\beta) \varphi^{\underline{\alpha}}_{\underline{\beta}}:\Delta({\underline{\alpha}}) \langle  \ell \rangle
\rightarrow \Delta({\underline{\beta}}) \langle   \ell-1  \rangle.
$$
 We    set  $ C_\bullet(1_{\mathfrak{S}^\aatchpair }) = 
  \bigoplus_{\ell\geq 0 }{  \Delta}_\ell  \langle  \ell \rangle 
 $ 
together with the differential $(\delta_\ell)_{\ell\geq 0}$.  

\begin{thm}[{\cite[Theorem B]{withemily}}] 
The  complex $C_\bullet (1_{\mathfrak{S}^\aatchpair })$ 
 is exact except in degree zero, where    $$H_0(C_\bullet( 1_{\mathfrak{S}^\aatchpair } ))=L(  1_{\mathfrak{S}^\aatchpair } ).$$    
The underlying  graded character is as follows, 
$$
[L( 1_{\mathfrak{S}^\aatchpair })] = \sum_{w\in {\Lambda_\aatchpair{(n,{\bs })}}} (-t)^{\ell(w)}[\Delta(\w)]
$$

\end{thm}

We emphasise that the identity coset labels the fundamental alcove in the alcove geometry, by definition (and so has length 0). 
That is, the alcove $\mathcal{F}_\aatchpair (n)$ 
which contains all the calibrated simple modules
 $D_\bs(\bla)$ such that $h(\bla) = \aatchpair $.  
 We set the {\em length} of  $\mu \in  \mathcal{F}_\aatchpair (n) \cdot x$ to be given by  $\ell(\mu):= \ell(x)$.   
By the translation principle 
of \cite[Proposition 7.4]{cell4us2}, we immediately 
obtain the following: 

\begin{thm} \label{thm:BGGres}
Given   $ \aatchpair \in \mathbb N^\ell$ and $\bla \in \mathcal{F}_\aatchpair (n)$, 
 we have an associated  
  complex 
  $$   C_\bullet (\bla)=   \bigoplus_{
\begin{subarray}c 
\bmu \trianglerighteq  \bla  
\end{subarray}
}S_\bs(\bmu) \langle  \ell(\bmu)\rangle.
$$
This complex 
 is exact except in degree zero, where    $H_0(C_\bullet( \bla ))=D_\bs(\bla).$    
The underlying  graded character is as follows, 
$$
[D_\bs(\bla)] = 
\sum_{
\begin{subarray}c 
\bmu \trianglerighteq  \bla  
\end{subarray}
}  
 (-t)^{\ell(\mu)}[S_\bs(\bmu) \langle  \ell(\bmu)\rangle]
$$

\end{thm}

 
 By classical results for quasi-hereditary algebras, 
we  have the following immediate corollary:

 \begin{corollary} \label{cor:ext vanishing}
 Let $\bla\in \mathcal{F}_\aatchpair (n)$ and $\bmu \trianglerighteq \bla$.  We have that 
 $$
 \dim_\Bbbk(\Ext^i _{\algebra }( S_\bs(\bmu) , D_\bs(\bla)))
 = 
 \begin{cases}
1		 & i=\ell(\bmu)-\ell( \bla)	\\
 0		&\text{otherwise.}	
 \end{cases}
 $$
 
 \end{corollary}
 
Finally, we now use our BGG resolutions to construct the characteristic-free bases of Theorem~C.   
 
\begin{defn}
Given $ \bla \in   \mathcal{F}_{\aatchpair}(n)$, we define
 $\Path_{ {\aatchpair}}^\mathcal{F}( \bla)$ to be the set of all paths
 $$
 \SSTT= (\varnothing =\SSTT(0) ,\SSTT(1),\dots ,\SSTT(n)  = \bla)
 $$
  for which $\SSTT(k)\in  \mathcal{F}_{\aatchpair}(k)$ for $0\leq k \leq n$. 
\end{defn}

  \begin{thm}
Let $\Bbbk$ be a field.  Given 
   $ \bla  \in \mathcal{F}_{\aatchpair}(n)$, the calibrated simple module $D_\bs(\bla)$ has basis 
$ 
\{\psi_\sts \otimes _\ZZ\Bbbk \mid \sts \in \Path_{ {\aatchpair}}^\mathcal{F}( \bla)\}.
$ 
The action of $\mathscr{H}_n(\bs)$   is as follows:  
$$
y_k(\psi_\sts)=0\qquad e_{\underline{i}} (\psi_\sts)= \delta_{\underline{i},{\rm res}(\sts)}\qquad
\psi_k(\psi_\sts)=
\begin{cases}
\psi_{\sts_{k\leftrightarrow k+1}} &\text{if }|\res(\sts^{-1}(k)) - \res^{-1}(\sts(k+1))|>1 			\\
0   &\text{otherwise }
\end{cases}$$
where ${\sts_{k\leftrightarrow k+1}}$ is the tableau obtained from $\sts$ by swapping 
the entries $k$ and $k+1$. 
\end{thm}

\begin{proof}
Since each $\bmu \trianglerighteq \bla$ lies in an alcove, $|{\rm Rem}_r(\bmu)|\leq 1$ for $r\in \ZZ/e\ZZ$.  
By the branching rule \cite[Proposition 1.26]{cell4us}, each  $ S_\bs (\bmu)$ restricts to be a direct sum of Specht modules and therefore $$
\res_{\mathscr{H}_{n-1}(\bs)}( C_\bullet (\bla)) = 
\bigoplus_{
\begin{subarray}c
r\in\ZZ/e\ZZ
\\
\square_r \in {\rm Rem}_r(\bla)
\end{subarray}}
C_\bullet   (\la-\square_r) 
 $$
 and therefore 
 $$\res_{\mathscr{H}_{n-1}(\bs)}(D_\bs (\la)=
 \bigoplus_{
\begin{subarray}c
r\in\ZZ/e\ZZ
\\
\square_r \in {\rm Rem}_r(\bla)
\end{subarray}} 
D_\bs
 (\la-\square_r) $$
 where every simple  on the righthand-side is labelled by some 
 $\bla-\square_r \in \mathcal{F}_\aatchpair (n-1)$.  
 Thus the  basis result follows by restriction.  
 The action of the idempotents on this basis  is obvious.  The other zero-relations all follow because the product has non-zero degree (whereas the module $D_n(\lambda)$ is concentrated in degree 0).  
Finally, assume $|\res(\sts^{-1}(k)) - \res^{-1}(\sts(k+1))|>1 $.  
The strands terminating at  the $k$th and $(k+1)$th positions on the northern edge either do or do not cross.  In the former case, we can resolve the double crossing in $\psi_k c_\sts$ without cost by our assumption on the residues and the result follows.  The latter case is trivial.  
Finally, notice that ${\sts_{k\leftrightarrow k+1}}\in \Path_{ {\aatchpair}}^\mathcal{F}( \bla) $  under the assumption that 
$ |\res(\sts^{-1}(k)) - \res^{-1}(\sts(k+1))|>1 		$.  
  \end{proof}

 \begin{rmk}
 The algebra $\algebra$ is a quasi-hereditary 
 quotient of the Hecke algebra for $\bs\in \ZZ^\ell$ with respect to the $\bs$-cellular structure of \cite{manycell,cell4us}. 
 Therefore $\algebra$ is Morita equivalent to 
 the corresponding quotient of the Cherednik algebra with cylindric  charge  
 $\bs\in \ZZ^\ell$ (although 
 this requires some chasing through results of 
 \cite{manycell,cell4us,MR3732238}) 
 and thus one can lift all our results   to the Cherednik algebra  with   cylindric  charge  
 $\bs\in \ZZ^\ell$.  This (Morita equivalent) setting is of interest because it allows one to deduce geometric applications of our results, following the machinery of \cite{griffeth-subspace} and \cite[Section 9]{bns} (but we do not explore this here).  Moreover,  in this setting  
the    $\Ext$-groups calculated in Corollary \ref{cor:ext vanishing}  can also be  computed using a cyclotomic version of Littlewood--Richardson coefficients, see  \cite[Theorem 1.1]{fgm}.
 As the authors of \cite{fgm} observe, {\em``for small values of $\ell$   and $n$, there is a certain tendency for the dimensions of the relevant Ext groups are always 0 or 1"}.
 It is not at all evident from their formula that a multiplicity one result should hold for a wide class of representations, such as those of full support.  Our Corollary \ref{cor:ext vanishing} provides evidence for their postulation (which they posit in the context of arbitrary charges $\bs \in \ZZ^\ell$) for representations with full support. 
 
Finally, for level $\ell=2$ with a {\em non}-cylindric  charge   
 $\bs\in \ZZ^\ell$, 
 it is shown in  \cite{emilytypeb} 
that all unitary representations of the  Cherednik algebra   can be reduced to the level $\ell=1$ case.  
 Thus for level~2,   every unitary 
  representation  of  a Cherednik algebra (regardless of the charge) admits a BGG resolution.  
 \end{rmk}

\appendix

\section{Unitary representations of the 
\\ Hecke algebra of the symmetric group}\label{sec:level1}
 
\subsection{Preliminaries} Recall that the type $A$ finite Hecke algebra $H_{q}(n)$ is the subalgebra of $\AHA_q(n)$ generated by $T_{1}, \dots, T_{n-1}$. For our purposes, it is better to realize $H_{q}(n)$ as a quotient of $\AHA_q(n)$. Indeed,
$$
H_{q}(n) \cong \AHA_q(n)/(X_1 - 1)
$$
and thanks to this, we have notions of calibrated and unitary $H_{q}(n)$-representations. Note that in this setting the charge $\bs$ consists of a single number $s$ that we may assume to be $0$, and the irreducible representations of $H_{q}(n)$ are labeled by partitions obtained by applying crystal operators $\tilde{f}_{i}$ to the empty partition.  

\color{black}

\subsection{Calibrated representations of $H_{q}(n)$.} Let $\ba \in \cal^{\aff}$ be a calibrated weight. It is clear that $M_{[\ba]} \in H_{q}(n)\text{-mod}$ if and only if $b_{1} = 1$ for every $\bb \in [\ba]$. This means that, $a_{1} = 1$ and for $i > 1$ $\{qa_{i}, q^{-1}a_{i}\} \cap \{a_{1}, \dots, a_{i-1}\} \neq \emptyset$. Indeed, otherwise we would be able to find a sequence of admissible transpositions from $\ba$ to some weight $\bb \in [\ba]$ with $b_{1} = a_{i}$ and $a_{i} \neq 1$, a contradiction. Inductively, we can see that if $M_{[\ba]} \in H_{q}(n)\text{-mod}$ then for every $i$ there exists $m_{i} \in \Z$ such that $a_{i} = q^{m_{i}}$. 

Let $e = 0$ if $q$ is not a root of unity; or $e > 0$ minimal such that $q^{e} = 1$. Let $I = \Z/e\Z$. We will identify the weight $\ba = (q^{m_{1}}, \dots, q^{m_{n}})$ with $ \bm= (m_{1}, \dots, m_{n}) \in I^{n}$. Then, similarly to \cite[Lemma 4.5]{MR2266877}\footnote{Note that Ruff uses the terminology  \lq\lq completely splittable\rq\rq\, while we use \lq\lq calibrated\rq\rq} we can see the following. 

\begin{lemma}\label{lemma:ruff}
Let $ \bm = (m_{1}, \dots, m_{n}) \in I^{n}$. Then, $\bm$ is a weight of a calibrated $H_{q}(n)$-module if and only if the following conditions are satisfied
\begin{itemize}
\item for every $i < j$, if $m_{i} = m_{j}$, then $m_{i}+1, m_{i} -1 \in \{m_{i+1}, \dots, m_{j-1}\}$ (that is, $\bm$ is calibrated)
\item $m_{1} = 0$.
\item for every $i > 1$, $\{m_{i} - 1, m_{i} + 1\}\cap\{m_{1}, \dots, m_{i-1}\} \neq \emptyset$
\end{itemize}
\end{lemma}

We let $\cal \subseteq I^m$ be the set of weights satisfying the conditions of Lemma \ref{lemma:ruff}. Note that if $\bm \in \cal$ and $s_{i}$ is an admissible transposition of $\bm$ (that is, $m_{i} - m_{i+1} \neq \pm1$) then $s_{i}\bm \in \cal$. Thus, the set of irreducible calibrated $H_{q}(n)$-modules is parametrized by $\cal/\sim$. 

Now let $\uni \subseteq \cal$ be the set of weights appearing in \emph{unitary} $H_{q}(n)$-modules. In general, $\uni$ is only a \emph{proper} subset of $\cal$. For certain values of $q$ we do have $\uni = \cal$, as the following result shows.

\begin{proposition}
Let $e > 0$ and let $q = \exp(2\pi\sqrt{-1}/e)$. Then, $\uni = \cal$ and an $H_{q}(n)$-module is calibrated if and only if it is unitary. 
\end{proposition}
\begin{proof}
By our choice of $q$, the only power of $q$ whose real part is greater than that of $q$ is $q^{e} = 1$. Now let $\bm \in \cal$, and $i < j$ such that $m _{i} \neq m_{j}$. Since $q^{m_{i} - m_{j}} \neq 1$ we get $\Re(q^{m_{i} - m_{j}}) \leq \Re(q)$. The result now follows from Lemma \ref{lemma:easy}
\end{proof}

\begin{remark}
If we take another primitve $e$-th root of unity, we may find calibrated representations which are not unitary. For example, if $e$ is odd and  $q = \exp((e-1)\pi\sqrt{-1}/e)$, then $\Re(q^k) \geq \Re(q)$ for every $k \in \Z$, so any unitary $H_{q}(n)$-module is 1-dimensional. 
\end{remark}




Calibrated modules for $H_{q}(n)$ have been classified in previous work \cite{MR1383482,MR2266877} (as well as being a  special case of Theorem B).  

\begin{lemma}\label{lemma:calibrated}
If $q$ is not a root of unity, then $D(\lambda)$ is calibrated for any partition $\lambda$. If $q$ is a primitive $e$-th root of unity, then $D(\lambda)$ is calibrated if and only if 
$|\Bo^\bs(\la)|  < e$. 
\end{lemma}

Our goal is now to use Lemma \ref{lemma:calibrated} to   completely classify of unitary representations of $H_{q}(n)$. 

\subsection{Admissible tableaux} Note that for $k \leq n$ we have a natural inclusion $H_{q}(k) \subseteq H_{q}(n)$ and thus we have an exact restriction functor $\Res^{n}_{k}: H_{q}(n)\text{-mod} \to H_{q}(k)\text{-mod}$.
We will say that   $\tab\in \Std(\la)$ is \emph{$q$-admissible} if, for every $k$, the representation $D(\Shape(\stt{\downarrow}_{\{1,\dots ,k\}}))$ of $H_{q}(k)$ is nonzero and a subrepresentation of $\Res^{n}_{k}(D(\la))$. Equivalently, this means that $D(\lambda)$ is nonzero and the box labeled by $k$ is good removable box of $\Shape(\stt{\downarrow}_{\{1,\dots ,k\}})$ for every $k$. If $q$ is not a root of unity then every tableau on $\lambda$ is $q$-admissible, but this is not the case if $q$ is a root of unity.




If $D(\la)$ is calibrated, admissible tableaux correspond to weights as follows. Let $\tab$ be an admissible tableau of $\lambda$. Then, $m_{\tab} = (-\content(\tab^{-1}(1)), \dots, -\content(\tab^{-1}(n)))$ is a weight of $D(\lambda)$. This defines a bijection between weights of $D(\lambda)$ and admissible tableaux on $\lambda$. 

Now let us denote by $\column$ the column-reading tableau on $\lambda$, that is the tableau obtained by placing $\{1, 2, \dots, \lambda_1^{t}\}$ on the boxes in the first column, $\{\lambda^{t}_1 + 1, \dots, \lambda^{t}_1 + \lambda^{t}_2\}$ on the boxes in the second column, and so on, see Definition \ref{revertable}. The following result will be very important in our arguments.

\begin{lemma}\label{lemma:columnreading}
Assume that $D(\lambda)$ is a calibrated $H_{q}(n)$-module. Then, the column-reading tableau $\column$ on $\lambda$ is admissible.
\end{lemma}
\begin{proof}
If $q$ is not a root of unity, there is nothing to show. Let us assume that $q$ is a primitive $e$-th root of unity, so that $|\Bo^\bs(\bla)| + 1 < e$. We claim that $\column^{-1}(n)$ is a good removable box of $\lambda$. To see this note that, since $|\Bo^\bs(\bla)|+ 1 < e$, $\lambda$ has at most one removable box of each residue, so the claim will follow if we check that there is no addable box of the same residue as $\column^{-1}(n)$ to the left of $\column^{-1}(n)$. But this is clear since the last column of $\lambda$ has at most $(e-1)$-boxes.

Now, $D(\lambda \setminus \{\column^{-1}(n)\})$ is a calibrated $H_{q}(n-1)$-module. The column-reading tableau of $\lambda \setminus \{\column^{-1}(n)\}$ is simply the restriction of the column-reading tableau of $\lambda$. So the result follows by an inductive argument.  
\end{proof}


\subsection{Unitary loci} In this section, we classify the unitary representations of $H_{q}(n)$ for any $q \in \C^{\times}$, $n > 0$. Recall that it only makes sense to speak about unitary representations when $q$ lies in the unit circle, so the following definition is sensible.

 \begin{definition}[\cite{MR2898660,stoica2009unitary}] Let $\lambda \vdash n$ be a partition. We define the \emph{unitary locus} of $\lambda$ to be 
$$
U(\lambda) := \{c \in (-1/2, 1/2] : D(\lambda) \neq 0 \; \text{is a unitary representation of} \; H_{\exp(2\pi\sqrt{-1}c)}(n)\}
$$
\end{definition}

We will classify unitary representations via a complete, explicit description of the unitary locus of every partition.  
To state our result, we first fix some notation.  
For  $\bla\in \mathscr{P}_1(n)$ with $h(\bla)=h\in \mathbb N$, 
 we define the hook length of a node $(i,j)\in \la$ as follows
 $$H_{(i,j)}(\la)=\la_i+\la_j^t + 1 - i -i  $$
and we note that  $|\Bo^\bs(\bla)|=H_{1,\la_h}$.  
%
%
%
%
%
In what follows we will  set 
 $\ell := H_{(1,1)}(\la)  $ 
   and   $m := |\Bo^\bs(\bla)|  $. 
 

\begin{theorem}\label{thm:mainA}
Let $\lambda$ be a partition. The unitary locus of $\lambda$ is described as follows.
\begin{enumerate}
\item If $\lambda$ consists of a single row, then $U(\lambda) = (-1/2, 1/2]$.
\item If $\lambda$ consists of a single column, say $\lambda = (1^n)$, then $U(\lambda) = (-1/2, 1/2] \setminus \{\pm a/e : 1 < e \leq n, \gcd(a;e) = 1\}$.
\item Assume that $\lambda$ is an almost rectangle, i.e. $\lambda$ has the form $\lambda = (a^{x}, (a-1)^{y})$ for some $a > 1, x > 0, y \geq 0$. Then
 $$
U(\lambda) = 
\left[ {-1}/{\ell},  {1}/{\ell}\right] \bigcup \{\pm 1/L : m \leq L \leq \ell\} \bigcup \{\pm d/m : \gcd(d;m) = 1\}.
$$
  Note that, in this case, $m = x + y + 1$. 
\item Else, $U(\lambda) = [-1/\ell, 1/\ell] \cup \{\pm 1/L : m \leq L \leq \ell\}$. 
\end{enumerate}
\end{theorem}

\begin{remark}
We remark that Theorem \ref{thm:mainA} has appeared in work of Stoica \cite{stoica2009unitary}, with two differences. First, our conventions are dual to those of Stoica, so the statements differ by taking the transpose of $\lambda$. Second, and most importantly, there is an oversight in \cite[Theorem 4.2]{stoica2009unitary}, which does not consider elements of the form $d/m$ in case (3) above. But in this case the representation $D(\lambda)$ is $1$-dimensional, a fortiori unitary. The oversight in \cite{stoica2009unitary} seems to stem from the  computation of the branching rule for the Hecke algebra in \cite[Proposition 4.3]{stoica2009unitary}. Finally, we remark that Venkateswaran has computed the unitary loci under the assumption that $q$ is not a root of unity in \cite{MR4037563}.
\end{remark}


The proof of Theorem \ref{thm:mainA} is contained in the next several lemmas. To start, we have the following easy result, which covers cases (1) and (2).

\begin{lemma}\label{lemma:trivial}
The unitary locus of the trivial partition is $U(n) = (-1/2, 1/2]$. The unitary locus of the sign partition is $U(1^n) = (-1/2, 1/2] \setminus \{\frac{a}{d} : e \leq n, \gcd(a;e) = 1\}$. 
\end{lemma}
\begin{proof}
The representation $D(n)$ is always 1-dimensional, while the representation $D(1^n)$ is 1-dimensional whenever it is defined, and it is defined if and only if $q$ is not an $e$-th root of unity with $e \leq n$. The result follows.
\end{proof}

From now on, we will assume that $\lambda$ is neither the trivial nor the sign partition. Let us start with the case when $q$ is not a root of unity, equivalently, $c$ is irrational.

\begin{lemma}\label{lemma:irrational}
Let $\lambda$ be a partition of $n$, $\lambda \neq (n), (1^n)$. Let $c \in (-1/2, 1/2]$ be irrational. Then, $c \in U(\lambda)$ if and only if $c \in (-1/\ell, 1/\ell)$ where, recall, $\ell$ is the length of the longest hook of $\lambda$. 
\end{lemma}
\begin{proof}
From the column-reading tableau of $\lambda$, we can see from Lemma \ref{lemma:easy} that $D(\lambda)$ is $H_{q}(n)$-unitary if and only if $\Re(q^{i}) \leq \Re(q)$ for every $i = 1, \dots, \ell$. The result follows. 
\end{proof}

\begin{remark}
Lemma \ref{lemma:irrational} also follows from the main result of \cite{MR4037563}, where the signature of the form $\langle \cdot, \cdot\rangle$ on $D(\lambda)$ is computed under the assumption that $q$ is not a root of unity.
\end{remark}

Now we need to consider the case of \emph{rational} $c$, i.e., when $q$ is a root of unity. First, we consider the case where $|c| \leq 1/\ell$.

\begin{lemma}
Let $\lambda \neq (1^n), (n)$ and $c \in [-1/\ell, 1/\ell]$. Then, $c \in U(\lambda)$. 
\end{lemma}
\begin{proof}
In view of Lemma \ref{lemma:irrational}, we only need to consider the case where $q := \exp(2\pi\sqrt{-1}c)$ is a root of unity. Let $e>0$ be minimal such that $q^{e} = 1$. Since $c \in [-1/\ell, 1/\ell]$, we have that $e \geq \ell$. If  either $\lambda^{t}_{1} < \ell$ or $\ell < e$, it follows from Lemma \ref{lemma:calibrated} that $D_{\lambda}$ is calibrated and then, by Lemma \ref{lemma:columnreading}, that the column-reading tableau of $\lambda$ is $e$-admissible, The result now follows just as in the proof of Lemma \ref{lemma:irrational}. If $\lambda_{1} = \ell = e$, then the partition is the one-row partition $(e)$. But this is not $e$-restricted and we are done.
\end{proof}

As corollary we note that $[-1/\ell, 1/\ell] \subseteq U(\lambda)$. Now we take care of rational numbers $c \in (-1/2, 1/2] \setminus [-1/\ell, 1/\ell]$. We separate in two cases, according to the denominator of $c$. 

\begin{lemma}
Let $\lambda \neq (1^n), (n)$ and $c = \frac{a}{e} \in (-1/2, 1/2]$ with $\gcd(a;e) = 1$. Assume that $e > \ell$. Then, $c \in U(\lambda)$ if and only if $c \in [-1/\ell, 1/\ell]$.
\end{lemma}
\begin{proof}
Again, we have that the column-reading tableau is admissible and $D(\lambda)$ is unitary if and only if $\Re(q^{i}) \leq \Re(q)$ for every $i = 1, \dots, H_{(1,1)}(\la)
$. The result follows. 
\end{proof}

\begin{lemma}
Let $\lambda \neq (1^n), (n)$ and $c = \frac{a}{e} \in (-1/2, 1/2]$ with $\gcd(a;e) = 1$. If $e < m$, $c \not\in U(\lambda)$.
\end{lemma}
\begin{proof}
By Lemma \ref{lemma:calibrated}, $D(\lambda)$ is not a calibrated representation of $H_{q}(n)$, so it cannot be unitary, either.
\end{proof}

\begin{lemma}\label{lemma:pm1}
Let $\lambda \neq (1^n), (n)$ and $c = \frac{a}{e} \in (-1/2, 1/2]$, with $m \leq e \leq \ell$ and $\gcd(a;e) = 1$. Assume also that $\lambda$ is not an almost rectangle, as defined in the statement of Theorem \ref{thm:mainA}.  
 Then, $c \in U(\lambda)$ if and only if $a = \pm1$.
\end{lemma}
\begin{proof}
Let $q = \exp(2\pi\sqrt{-1}c)$. Note that $m = \lambda^{t}_{1} + \#\{i : 0 < \lambda^{t}_{i} < \lambda_{1}\}$. By our assumptions on $\lambda$, the cardinality of the set $\{i : 0 < \lambda_{i} < \lambda_{1}\}$ is at least 2. So we have $e \geq m > \lambda_{1} + 1$. In particular, $\lambda$ is $e$-restricted and, moreover, the column-reading tableaux is admissible. Also note that, for every $i$, $q^{-\lambda^{t}_{i} + i}q^{-i} = q^{-\lambda^{t}_{i}} \neq q^{\pm1}$. In other words, the permutation $s_{\lambda^{t}_{1} + \dots + \lambda^{t}_{i}}$ is admissible for the weight given by the column-reading tableau. Using admissible permutations now, it is straightforward to see that $\lambda$ is unitary if and only if $\Re(q^{i}) \leq \Re(q)$ for every $i = 2, \dots, e-2$, and the result follows. 
\end{proof}

Finally, we deal with the case of an almost rectangle.

\begin{lemma}\label{lemma:1dim}
Let $\lambda = (a^x, (a-1)^y)$ for some $a > 0, x > 0, y > 0$, and let $c = \frac{d}{e} \in (-1/2, 1/2]$ with $m \leq e  \leq \ell$. Then, $c \in U(\lambda)$ if and only if $e = m$ or $d = \pm 1$.
\end{lemma}
\begin{proof}
If $e > m = x+y+1$, then we can proceed just as in Lemma \ref{lemma:pm1} to conclude that $c \in U(\lambda)$ if and only if $a = \pm 1$. If $e = m$, note that $D(\lambda)$ is the $1$-dimensional sign representation of $H_{q}(|\lambda|)$, so $c \in U(\lambda)$, regardless of the value of the numerator $a$.
\end{proof}

Theorem \ref{thm:mainA} now follows from Lemmas \ref{lemma:trivial}--\ref{lemma:1dim}. Together with results from \cite{MR2534594, griffeth} this implies the following result, that has been proven using different techniques by Shelley-Abrahamson in \cite{MR4054879}, in the more general case of real reflection groups.

\begin{corollary}
Let $c \in \R$ and let $L_{c}(\lambda)$ be an unitary representation of the rational Cherednik algebra $\RCA_{c}(n)$. Then, $\KZ(L_{c}(\lambda))$ is either zero or unitary. 
\end{corollary}

\begin{remark}\label{rmk:quasi vs unitary}
We remark that if $D$ is a unitary representation of $H_{q}(n)$ it is, in general, \emph{not} possible to find a parameter $c \in \R$ with $q = \exp(2\pi\sqrt{-1}c)$ and a unitary representation $L$ of the rational Cherednik algebra $\RCA_{c}(n)$ such that $\KZ(L) = D$. Indeed, if $q = \exp(2a\pi\sqrt{-1}/n)$ with $\gcd(a;n) = 1$ and $a > 1$, it is not in general possible to find a unitary representation of the Cherednik algebra whose image under $\KZ$ is the $1$-dimensional trivial representation of $H_{q}(n)$, see \cite{griffeth}. 
(We remark that our conventions are transposed from those of \cite{griffeth}.) 


More generally, let $W$ be a complex reflection group and $c: S \to \C$ a parameter for the rational Cherednik algebra, where $S \subseteq W$ is the set of reflections. We assume that the parameter $c$ is real, in the sense that $c(s^{-1}) = \conj{c(s)}$ for every $s \in S$. Under this condition on $c$, there is a notion of unitary representations of $\RCA_c(W)$, as well as of $H_{q}(W)$ where $q$ is related to $c$ via a precise exponential formula, see \cite{GGOR}. Let $L$ be an irreducible representation of $\RCA_c(W)$ with full support. Let us say that $L$ is \emph{quasi-unitary} if $\KZ(L)$ is unitary. This coincides with the notion of quasi-unitarity presented in \cite[Definition 3.5.3]{MR4054879} in the case when $W$ is a real reflection group, and it is expected to coincide always. It is easy to see that the socle of the polynomial representation is always quasi-unitary. We expect, however, that for most parameters it is \emph{not} unitary. See \cite{MR2959035, griffeth, griffeth-subspace} for related work.

Finally, this shows that the problems of determining unitary representations of the Hecke algebra and unitary representations with full support of the rational Cherednik algebra, while closely related, are not equivalent, and moreover none implies the other: the set of unitary representations of the Hecke algebra is larger than that of fully-supported unitary representations of the Cherednik algebra and there is not, to the best of our knowledge, an explicit criterion to find which unitary representations of the Hecke algebra indeed give unitary representations of the Cherednik algebra. 
\end{remark}
 

\begin{thebibliography}{AvLTVJ20}

\bibitem[AK94]{ak94}
S.~Ariki and K.~Koike, \emph{\href{http://dx.doi.org/10.1006/aima.1994.1057}{A
  {Hecke} algebra of $(\mathbb{Z}/r\mathbb{Z})\wr\mathfrak{S}_n$ and
  construction of its irreducible representations}}, Adv.\ Math. \textbf{106}
  (1994), 216--243.

\bibitem[Ari02]{MR1911030}
S. Ariki, \emph{Representations of quantum algebras and combinatorics of
  {Y}oung tableaux}, University Lecture Series, vol.~26, American Mathematical
  Society, Providence, RI, 2002, Translated from the 2000 Japanese edition and
  revised by the author. 

\bibitem[ALTV20]{aldvreal}
J. Adams, M. van Leeuwen, P. Trapa, and D. Vogan~Jr,
  \emph{Unitary representations of real reductive groups}, Ast{\'e}risque
  (2020), no.~417, V--+.

 


  
\bibitem[BC15]{BC15b}
D. Barbasch and D. Ciubotaru,   
{\em  Ladder representations of ${\rm GL}(n,\mathbb{Q}_p)$,} Representations of reductive groups, 117--137, 
Progr. Math., 312, Birkhauser/Springer, Cham, 2015

\bibitem[BC20]{barbasch2015star}
D. Barbasch and D. Ciubotaru, \emph{Star operations for affine {H}ecke
  algebras,} Representation Theory, Automorphic Forms \& Complex Geometry (A Tribute to Wilfried Schmid). Int. Press Boston, 2020, 107--138
  
  \bibitem[BZGS14]{conjecture}
C.~Berkesch~Zamaere, S.~Griffeth, and S.~V. Sam, \emph{Jack polynomials as
  fractional quantum {H}all states and the {B}etti numbers of the
  {$(k+1)$}-equals ideal}, Comm. Math. Phys. \textbf{330} (2014).
 
 \bibitem[BGG75]{bgg}
I.~N. Bernstein, I.~M. Gelfand, and S.~I. Gelfand, \emph{Differential operators
  on the base affine space and a study of {${\mathfrak{ g}}$}-modules},  Lie groups and their representations (Proc. Summer School, Bolyai J\'anos Math. Soc., Budapest, 1971), pp. 21--64. Halsted, New York, 1975. 
 

\bibitem[Bow17]{manycell}
C.~Bowman, \emph{The many graded cellular bases of {H}ecke algebras},
  \href{https://arxiv.org/abs/1702.06579}{arXiv:1702.06579}.

 


\bibitem[BCH20]{cell4us2}
C.  Bowman, A. Cox, and A. Hazi, \emph{Path isomorphisms between quiver
  {H}ecke and diagrammatic {B}ott-{S}amelson endomorphism algebras},  arXiv:2005.02825 (2020).


\bibitem[BCHM20]{cell4us}
C.  Bowman, A. Cox,   A. Hazi, and D. Michailidis, \emph{Path combinatorics and light leaves for quiver Hecke algebras},  arXiv:2009.01705 (2020).


\bibitem[BHN20]{withemily}
C.  Bowman, A. Hazi, and E.  Norton, \emph{The modular {W}eyl-{K}ac
  character formula},   arXiv:2004.13082.  


  
\bibitem[BNS18]{bns}
C~Bowman, E~Norton, and J~Simental, \emph{Characteristic-free bases and bgg
  resolutions of unitary simple modules for quiver {H}ecke and {C}herednik
  algebras},  arXiv:1803.08736 (2018).

\bibitem[BS21]{brundanstroppel}
J. Brundan and C. Stroppel, \emph{Semi-infinite highest weight
  categories}. Memoirs of the AMS, to appear. 

\bibitem[BK09]{MR2551762}
J.~Brundan and A.~Kleshchev, \emph{Blocks of cyclotomic {H}ecke algebras and
  {K}hovanov-{L}auda algebras}, Invent. Math. \textbf{178} (2009).

\bibitem[BMR98]{MR1637497}
M. Brou\'{e}, G.  Malle, and R.  Rouquier, \emph{Complex
  reflection groups, braid groups, {H}ecke algebras}, J. Reine Angew. Math.
  \textbf{500} (1998), 127--190. 
 
 
 \bibitem[CG10]{MR2838836}
N. Chriss and V. Ginzburg, \emph{Representation theory and complex
  geometry}, Modern Birkh\"{a}user Classics, Birkh\"{a}user Boston, Ltd.,
  Boston, MA, 2010, Reprint of the 1997 edition. 



\bibitem[EH04]{MR2037715}
T. Enright and M.~Hunziker, \emph{Resolutions and {H}ilbert series of
  determinantal varieties and unitary highest weight modules}, J. Algebra
  \textbf{273} (2004).


\bibitem[ES09]{MR2534594}
P.~Etingof and E.~Stoica, \emph{Unitary representations of rational {C}herednik
  algebras}, Represent. Theory \textbf{13} (2009), 349--370, With an appendix
  by S.  Griffeth.

\bibitem[EW16]{MR3555156}
B. Elias and G. Williamson, \emph{Soergel calculus}, Representation
  Theory of the American Mathematical Society \textbf{20} (2016), no.~12,
  295--374.

\bibitem[FGM21]{fgm}
S. Fishel, S.  Griffeth, and E. Manosalva, \emph{Unitary
  representations of the {C}herednik algebra: ${V}^{*}$-homology},
  Mathematische Zeitschrift (2021), 1--41.

\bibitem[FLO+99]{FLOTW99}
O. Foda, B. Leclerc, M. Okado, J.-Y. Thibon, and T. Welsh,
  \emph{Branching functions of
  {$A^{(1)}_{n-1}$} and {Jantzen--Seitz} problem for {Ariki--Koike} algebras},
  Adv.\ Math. \textbf{141} (1999), 322--365.

\bibitem[FS12]{MR2959035}
M.~Feigin and C.~Shramov, \emph{On unitary submodules in the polynomial
  representations of rational {C}herednik algebras}, Int. Math. Res. Not. IMRN
  (2012).

\bibitem[GTW17]{affinebraid}
A. Gadbled, A. Thiel, and E. Wagner, \emph{Categorical
  action of the extended braid group of affine type A}, Communications in
  Contemporary Mathematics \textbf{19} (2017), no.~03, 1650024.

\bibitem[GGOR03]{GGOR}
V.~Ginzburg, N.~Guay, E.~Opdam, and R.~Rouquier,
  \emph{\href{http://dx.doi.org/10.1007/s00222-003-0313-8}{On the category
  {$\mathcal{O}$} for rational {Cherednik} algebras}}, Invent.\ Math.
  \textbf{154} (2003), 617--651.


\bibitem[Gri18]{griffeth}
S.  Griffeth, \emph{Unitary representations of cyclotomic rational
  {C}herednik algebras}, J. Algebra \textbf{512} (2018), 310--356. 

\bibitem[Gri21]{griffeth-subspace}
\bysame, \emph{{Subspace arrangements and {C}herednik algebras}}, to appear in  Int. Math. Res. Not. IMRN
 (2021).  


\bibitem[Gro99]{grojnowski1999affine}
I. Grojnowski, \emph{Affine $\mathfrak{sl}_p$ controls the representation
  theory of the symmetric group and related {H}ecke algebras}, arXiv preprint
  math/9907129 (1999).

\bibitem[HMLSZ14]{formalaffine}
A. Hoffnung, J. Malag{\'o}n-L{\'o}pez, A. Savage, and K. 
  Zainoulline, \emph{Formal {H}ecke algebras and algebraic oriented cohomology
  theories}, Selecta Mathematica \textbf{20} (2014), no.~4, 1213--1245.

\bibitem[HW18]{huangwong}
J. Huang and K. Wong, \emph{A {C}asselman--{O}sborne theorem
  for rational {C}herednik algebras}, Transformation Groups \textbf{23} (2018),
  no.~1, 75--99.


\bibitem[Jac04]{J04} N. Jacon, \emph{On the parametrization of the simple modules for {A}riki--{K}oike algebras at roots of unity}. J. Math. Kyoto Univ. 44 (2004), no. 4, 729--767.

\bibitem[KK79]{BGG2}
V.~G. Kac and D.~A. Kazhdan, \emph{Structure of representations with highest
  weight of infinite-dimensional {L}ie algebras}, Adv. in Math. \textbf{34}
  (1979), no.~1, 97--108. 


\bibitem[KL09]{MR2525917}
M.~Khovanov and A.~Lauda, \emph{A diagrammatic approach to categorification of
  quantum groups. {I}}, Represent. Theory \textbf{13} (2009), 309--347.

\bibitem[Kle96]{MR1383482}
A.~Kleshchev, \emph{Completely splittable representations of symmetric groups},
  J. Algebra \textbf{181} (1996).

\bibitem[Los17]{losev-cat}
I. Losev, \emph{Rational cherednik algebras and categorification},
  Categorification and Higher Representation Theory, Contemp. Math \textbf{683}
  (2017), 1--40.
  
\bibitem[Mac92]{Mac92}  
G.   Mackey,  
{\em Harmonic analysis and unitary group representations: the development from 1927 to 1950}, L\'{e}mergence de l'analyse harmonique abstraite (1930--1950) (Paris, 1991), 13--42, 
Cahiers S. Hist. Math. S\'{e}r. 2, 2, Univ. Paris VI, Paris, 1992.

\bibitem[Nor20]{emilytypeb}
E.  Norton, \emph{Unitary representations of type {B} rational {C}herednik
  algebras and crystal combinatorics},   arXiv:2008.05464 (2020).

\bibitem[Ram03a]{MR1988991}
A.~Ram, \emph{Skew shape representations are irreducible}, Combinatorial and
  geometric representation theory ({S}eoul, 2001), Contemp. Math., vol. 325,
  Amer. Math. Soc., Providence, RI, 2003, pp.~161--189. 

\bibitem[Ram03b]{MR1976700}
A.  Ram, \emph{Affine {H}ecke algebras and generalized standard {Y}oung
  tableaux}, Journal of Algebra \textbf{260} (2003), no.~1, 367--415.

\bibitem[Rou08]{ROUQ}
R.~Rouquier, \emph{$2$-{Kac}--{Moody} algebras},
  \href{http://arxiv.org/abs/0812.5023}{arXiv:0812.5023}, 2008, preprint.

\bibitem[Ruf06]{MR2266877}
O.~Ruff, \emph{Completely splittable representations of symmetric groups and
  affine {H}ecke algebras}, J. Algebra \textbf{305} (2006).

\bibitem[SA20]{MR4054879}
S. Shelley-Abrahamson, \emph{The {D}unkl weight function for rational
  {C}herednik algebras}, Selecta Math. (N.S.) \textbf{26} (2020), no.~1, Paper
  No. 8, 63. 


\bibitem[Sto09]{stoica2009unitary}
E. Stoica, \emph{Unitary representations of {H}ecke algebras of complex
  reflection groups}, 2009.

\bibitem[Sto10]{MR2898660}
E. Stoica, \emph{Unitary {R}epresentations of {R}ational {C}herednik
  {A}lgebras and {H}ecke {A}lgebras}, ProQuest LLC, Ann Arbor, MI, 2010, Thesis
  (Ph.D.)--Massachusetts Institute of Technology. 

\bibitem[Ven20]{MR4037563}
V. Venkateswaran, \emph{Signatures of representations of {H}ecke algebras
  and rational {C}herednik algebras}, J. Algebra \textbf{547} (2020), 22--69.

\bibitem[Vog00]{MR1845669}
D.~A. Vogan, Jr., \emph{Unitary representations of reductive {L}ie groups},
  Atti Accad. Naz. Lincei Cl. Sci. Fis. Mat. Natur. Rend. Lincei (9) Mat. Appl.
  (2000), Mathematics towards the third millennium (Rome, 1999).


\bibitem[Web17]{MR3732238}
B.~Webster, \emph{Rouquier's conjecture and diagrammatic algebra}, Forum Math.
  Sigma \textbf{5} (2017), e27, 71.

\bibitem[Wig39]{Wig39}
E. Wigner, 
{\em On unitary representations of the inhomogeneous {L}orentz group}, 
Ann. of Math. (2) 40 (1939), no. 1, 149--204. 

\end{thebibliography}
\begin{acknowledgements}
The first author is grateful for funding from EPSRC  grant EP/V00090X/1.
This project has received funding from the European Research Council (ERC) under the European Union's Horizon 2020 research and innovation programme (grant agreement No 677147).
\end{acknowledgements}

\end{document}